\documentclass[a4paper, 11pt]{amsart}
\usepackage[utf8]{inputenc}
\usepackage[T1]{fontenc}
\usepackage[english]{babel}
\usepackage{graphicx}
\usepackage{stmaryrd}
\usepackage{tikz}
\usepackage{tikz-cd} 
\usepackage[all]{xy}
\usepackage{comment}
\usepackage{enumerate}
\usepackage{bm}
\usepackage{comment}
\usepackage{amsmath,amsfonts,amssymb}
\usepackage[a4paper]{geometry}
\usepackage{amsthm}
\usepackage{float}
\usepackage{hyperref}
\usepackage{xcolor}
\hypersetup{
    colorlinks=true,
    linkcolor={blue},
    citecolor={blue},
    urlcolor={blue}}
\newtheorem{te}{Theorem}[section]

\newtheorem{prop}[te]{Proposition}

\newtheorem{co}[te]{Corollary}

\newtheorem{qu}[te]{Question}
\newtheorem{lemme}[te]{Lemma}

\theoremstyle{definition}
\newtheorem{fact}[te]{Fact}
\newtheorem{notation}[te]{Notation}
\newtheorem{de}[te]{Definition}
\newtheorem{ex}[te]{Example}

\theoremstyle{remark}
\newtheorem{rk}[te]{Remark}

\newlength{\plarg}
\usepackage{IEEEtrantools}

\usetikzlibrary{cd}
\calclayout

\renewcommand{\leq}{\leqslant}
\renewcommand{\geq}{\geqslant}

\title{Homogeneity in Coxeter groups and split crystallographic groups}
\date{\today}

\author{Simon Andr{\' e}}

\author{Gianluca Paolini}

\begin{document}

\begin{minipage}{\linewidth}

\vspace{0mm}

\begin{abstract}
We prove that affine Coxeter groups, even hyperbolic Coxeter groups and one-ended hyperbolic Coxeter groups are homogeneous in the sense of model theory. More generally, we prove that many (Gromov) hyperbolic groups generated by torsion elements are homogeneous. In contrast, we construct split crystallographic groups that are not homogeneous, and hyperbolic (in fact, virtually free) Coxeter groups that are not homogeneous (or, to be more precise, not $\mathrm{EAE}$-homogeneous). We also prove that, on the other hand, irreducible split crystallographic groups and torsion-generated hyperbolic groups are almost homogeneous. We also prove that finitely generated abelian-by-finite groups are homogeneous if and only if they are profinitely homogeneous, i.e., any tuple of words from the group is profinitely rigid. We use this to deduce that affine Coxeter groups are profinitely homogeneous, a result of independent interest in the profinite context.
\end{abstract}

\thanks{Research of Gianluca Paolini was  supported by project PRIN 2022 ``Models, sets and classification'', prot. 2022TECZJA, and by INdAM Project 2024 (Consolidator grant) ``Groups, Crystals and Classifications''.}
	
\end{minipage}

\maketitle

\tableofcontents

\section{Introduction}

\thispagestyle{empty}

\smallskip

In recent years, the notion of homogeneity has been central in the model theory of finitely generated groups. Recall that the \emph{type} of a finite tuple $u$ of elements of a group $G$, denoted by $\mathrm{tp}(u)$, is the set of first-order formulas $\phi(x)$ (where $x$ denotes a tuple of variables of the same arity as $u$) such that $\phi(u)$ is satisfied by $G$. Obviously, two finite tuples $u,v$ that are in the same $\mathrm{Aut}(G)$-orbit have the same type; the group $G$ is said to be $\aleph_0$-\emph{homogeneous}, or simply \emph{homogeneous}, if the converse holds: for any integer $n\geq 1$ and tuples $u,v\in G^n$ having the same type, there is an element $\sigma\in \mathrm{Aut}(G)$ such that $\sigma(u)=v$. One of the major results in this area is the homogeneity of finitely generated free groups, proved by Perin and Sklinos \cite{PS12} and independently by Ould Houcine \cite{OH11}, relying on techniques introduced by Sela in his work on the Tarski problem for free groups (see \cite{Sel06} and other papers in the series), but the question of homogeneity remains open for many interesting classes of finitely generated groups, in particular in the presence of torsion.

\smallskip

The main motivation behind the present work is to continue our development of the model theory of (finitely generated) Coxeter groups (see \cite{MUHLHERR2022297,PS23,AP24}). In this paper we study the problem of homogeneity for this class of groups. Interestingly, our investigations on this topic lead to results and questions of independent interest, notably on first-order rigidity, profinite rigidity and {\em profinite homogeneity}. This last notion has already appeared elsewhere under the name {\em profinite rigidity of words}; we arrived at it independently. We will elaborate further on the connection with profinite rigidity of words below.

\smallskip

Recall that a \emph{Coxeter group} is a group that admits a presentation of the form \[\langle s_1,\ldots,s_n \ \vert \ (s_is_j)^{m_{ij}}=1, \ \text{for all} \ i,j \rangle,\] where $m_{ii}=1$ and $m_{ij} = m_{ji} \in \mathbb{N}^{\ast}\cup \lbrace \infty\rbrace$ for every $1\leq i,j\leq n$ (the relation $(s_is_j)^{\infty}=1$ means that $s_is_j$ has infinite order). Such a presentation is called a \emph{Coxeter presentation}. A Coxeter group is said to be \emph{even} (respectively \emph{right-angled}) if it admits a Coxeter presentation such that $m_{ij}$ is even or infinite for all $i\neq j$ (respectively $m_{ij}$ belongs to $\lbrace 2,\infty\rbrace$ for all $i\neq j$). A Coxeter group is said to be \emph{spherical} if it is finite (in which case it is obviously homogeneous) and \emph{affine} if it is virtually abelian, infinite and without a non-trivial normal finite subgroup. Among the infinite and non-affine Coxeter groups, a class of particular interest is that of Coxeter groups that are hyperbolic in the sense of Gromov (at the intersection of the irreducible affine Coxeter groups and the hyperbolic Coxeter groups there is only the infinite dihedral group). 

%Coxeter groups form a very well-studied class of groups with interesting properties that can be described combinatorially, algebraically and geometrically. Recall that a Coxeter group is said to be \emph{irreducible} if it cannot be written as a non-trivial direct product (equivalently, if its Coxeter diagram is connected). Every Coxeter group is the direct product of irreducible Coxeter groups, called its \emph{irreducible components}, that correspond to the connected components of its Coxeter diagram.

\smallskip

%(recall that a hyperbolic group is said to be \emph{non-elementary} if it is not virtually abelian, or equivalently if it contains a free group on two generators)

In this paper, we give a complete solution to the problem of homogeneity of affine Coxeter groups, and hyperbolic even or one-ended Coxeter groups, and in both cases these results lead to more general results in two important classes of finitely generated groups: split crystallographic groups and torsion-generated hyperbolic groups (that is, hyperbolic groups generated by elements of finite order). Recall that a finitely generated group is \emph{one-ended} if it does not split non-trivially as an HNN extension or as an amalgamated product over a finite group. Our main result on Coxeter groups is the following (see Theorem~\ref{AE} for a more general result on homogeneity in hyperbolic groups generated by torsion).

\begin{de}\label{AEdef}
An \emph{$\mathrm{AE}$-formula} or \emph{$\forall\exists$-formula} (in the language of groups) is a first-order formula of the form $\phi(x):\forall y  \ \exists z \ \theta(x,y,z)$ where $\theta(x,y,z)$ is a quantifier-free formula and $x,y,z$ denote finite tuples of variables. The \emph{$\mathrm{AE}$-type} of a finite tuple $u$ of elements of a group $G$ is the set of $\mathrm{AE}$-formulas $\phi(x)$ (where $x$ denotes a tuple of variables of the same arity as $u$) such that $\phi(u)$ is satisfied by $G$. The group $G$ is said to be \emph{$\mathrm{AE}$-homogeneous} if, for any integer $n\geq 1$ and any tuples $u,v\in G^n$ with the same $\mathrm{AE}$-type, there is an automorphism $\sigma$ of $G$ such that $\sigma(u)=v$. We define in a similar way \emph{$\mathrm{EAE}$-formulas}, the \emph{$\mathrm{EAE}$-type} of a finite tuple of elements, and \emph{$\mathrm{EAE}$-homogeneity}.
\end{de}

\begin{te}\label{theorem1}
Affine Coxeter groups, hyperbolic even Coxeter groups and hyperbolic one-ended Coxeter groups are homogeneous (in fact, $\mathrm{AE}$-homogeneous).
\end{te}

For example, all the triangle groups are homogeneous. In fact, from some of the results of Section \ref{direct_products}, it follows that any direct product of finitely many such groups is homogeneous (see Corollary \ref{cor_triangle}). Recall that triangle groups correspond to regular tessellations of the sphere, the Euclidean plane or the hyperbolic plane, and so they are, respectively, spherical, affine or one-ended hyperbolic Coxeter groups. 

\smallskip

Theorem \ref{theorem1} may seem restrictive at first glance, but our next result shows that non-homogeneous groups exist in classes slightly larger than those considered in Theorem \ref{theorem1}. Recall that a \emph{crystallographic group} of dimension $n\geq 1$ is a discrete and cocompact subgroup of $\mathrm{Isom}(\mathbb{R}^n)$, the group of isometries of the Euclidean space $\mathbb{R}^n$. Equivalently, by the First Bieberbach Theorem, a crystallographic group is a finitely generated virtually abelian group without a non-trivial normal finite subgroup; in particular, such a group $G$ admits a splitting as a short exact sequence of the form $1\rightarrow T\rightarrow G\rightarrow G_0\rightarrow 1$, where $T$ is a normal subgroup isomorphic to $\mathbb{Z}^n$, called the \emph{translation subgroup} of $G$, and $G_0$ is a finite group (acting faithfully on $\mathbb{Z}^n$), and $G$ is said to be \emph{split} if this exact sequence is split. Recall that affine Coxeter groups are crystallographic groups. It is known that there exist non-homogeneous polycyclic-by-finite groups, but, to the best of our knowledge, no example of non-homogeneous finitely generated abelian-by-finite group is known (however, without the assumption of finite generation, there are known examples: in Section 6 of \cite{NS90}, the authors prove that there exist non-homogeneous groups that are elementarily equivalent to $\mathbb{Z}$). In contrast to Theorem \ref{theorem1}, we prove the following result, where a group is called $\mathrm{EAE}$-homogeneous if tuples satisfying the same existential-universal-existential first-order formulas are automorphic (see Definition~\ref{AEdef} for a precise definition). 

\begin{te}\label{theorem2}
There exist split crystallographic groups that are not homogeneous, and there exist virtually free Coxeter groups that are not $\mathrm{EAE}$-homogeneous.
\end{te}

%Finite groups are the simplest examples of homogeneous groups, and it is well known (and not hard to see) that finitely generated abelian groups are homogeneous as well. But the problem becomes much more difficult when it comes to abelian-by-finite groups. Theorem \ref{theorem1} asserts that affine Coxeter groups are homogeneous, but there exist non-homogeneous split crystallographic groups, as shown by the following example.

More precisely, on the crystallographic side, we will prove the following result (on the hyperbolic side, we refer the reader to Theorem \ref{EAE_intro} and Subsection \ref{example} for an explicit example of a virtually free Coxeter group that is not $\mathrm{EAE}$-homogeneous).

\begin{te}\label{counterexample_theorem}Let $G_1$ and $G_2$ be non-isomorphic split crystallographic groups such that $\widehat{G}_1\simeq \widehat{G}_2$ (meaning that $G_1$ and $G_2$ have the same set of finite quotients), then $G_1 \times G_2$ is not homogeneous.\end{te}

Notice that it is known that for every prime number $p \geq 23$ there exist split crystallographic groups $G_1,G_2$ of the form $\mathbb{Z}^{p-1} \rtimes\mathbb{Z}/p\mathbb{Z}$ such that $G_1\not\simeq G_2$ but $G_1,G_2$ have the same set of finite quotients (see \cite{bri,FNP}). 

\smallskip

However, we prove the following result, where a group $G$ is said to be \emph{uniformly almost homogeneous} if there exists an integer $n\geq 1$ (which only depends on $G$) such that for any $k\geq 1$ and any $u\in G^k$, the set of $k$-tuples of elements of $G$ having the same type as $u$ is the union of at most $n$ orbits under the action of $\mathrm{Aut}(G)$. Note that the group $G$ is homogeneous if and only if one can take $n=1$ in this definition. This notion was introduced by the first named author in \cite{And18b}, where it was proved that finitely generated virtually free groups are uniformly almost homogeneous. Recall that a crystallographic group $1\rightarrow T\rightarrow G\rightarrow G_0\rightarrow 1$ is called \emph{irreducible} if the natural morphism $\rho: G_0\rightarrow \mathrm{GL}_n(\mathbb{Z})$ (see Subsection \ref{prelim_crystallo} for details), viewed as a representation $G_0\rightarrow\mathrm{GL}_n(\mathbb{Q})$, is irreducible (that is, the only linear subspaces of $\mathbb{Q}^n$ that are stable under \mbox{the action of $\rho(G_0)$ are $\mathbb{Q}^n$ and $\lbrace 0\rbrace$).}

\begin{te}\label{almost_h}
Irreducible split crystallographic groups and torsion-generated hyperbolic groups are uniformly almost homogeneous. In particular, hyperbolic Coxeter groups are uniformly almost homogeneous.\end{te}

\smallskip

The proof of Theorem \ref{counterexample_theorem} relies on the following fundamental result of Oger (see \cite{Oger88}): every finitely generated abelian-by-finite group $G$ is an elementary submodel of its profinite completion $\widehat{G}$. At this point, it is appropriate to make a small digression on another result of Oger on abelian-by-finite groups proved in \cite{Oger88}. Recall that a finitely generated group $G$ is said to be \emph{first-order rigid} if every finitely generated group $G'$ that is elementarily equivalent to $G$ is isomorphic to $G$, and that a finitely generated residually finite group $G$ is said to be \emph{profinitely rigid} if every finitely generated residually finite group $G'$ such that $\widehat{G}\simeq \widehat{G}'$ is isomorphic to $G$. Oger proved that two finitely generated abelian-by-finite groups are elementarily equivalent if and only if they have the same finite quotients, and thus that a finitely generated abelian-by-finite group $G$ is first-order rigid if and only if it is profinitely rigid. The crucial connection between these two notions of rigidity led to a proof of profinite rigidity of affine Coxeter groups via model theory due to the second named author of this paper and Sklinos, thus solving an open problem of M\"oller and Varghese posed in \cite{MV}. This problem was also solved in \cite{CHMV} by means of purely group-theoretic arguments. In the same spirit as Oger's work \cite{Oger88}, we will prove that the question of homogeneity for finitely generated abelian-by-finite groups can actually be phrased in profinite terms. This led us to the rediscovery of a notion which, as pointed out to us by an anonymous referee, has attracted much attention in the profinite group theory literature, that is the notion of {\em profinite rigidity of words}, which we recall in the next definition. 

\begin{de}Let $G$ be a finitely generated residually finite group. We say that $u \in G^n$ is profinitely rigid if whenever $v \in G^n$ is such that there is an automorphism of $\widehat{G}$ that sends $u$ to $v$, then there exists an automorphism of $G$ that sends $u$ to $v$.
\end{de}

If we ask profinite rigidity for any tuple of any length, we then arrive at what we call {\em profinite homogeneity} (the term is justified in the second paragraph after the definition and also in Theorem \ref{prof_homogeneity}).

\begin{de}\label{def_prof_hom} Let $G$ be a finitely generated residually finite group. We say that $G$ is \emph{profinitely homogeneous} if every tuple of every length in $G$ is profinitely rigid. 
\end{de}

For non-abelian free groups, the question of profinite homogeneity was originally asked by Lubotzky, and phrased using the terminology of measures induced by free group elements on finite groups via word maps. This is often studied word-by-word (note that the notion of a profinitely rigid word can be compared with the notion of being type-determined in the model theoretic context, see below). The known cases are words $w\in F$ such that the one-relator group $F/\langle\langle w\rangle\rangle$ is a free product of a free group and a surface group (one of which may be trivial), and powers of such $w$ (see \cite{PuderParzanchevski2015} for primitive elements, \cite{HananyMeiriPuder2020} for commutators of primitives, powers of profinitely rigid words, \cite{Wilton2021} for surface words and \cite{AF25} for the rest of the cases). In the context of groups which are not free, Theorem~\ref{prof_homogeneity} below feeds into an active field of research and improves a result of \cite{Singh2024}, which proves the same in the more restrictive setting of free abelian groups.

A word of explanation concerning Definition \ref{def_prof_hom} is in order. Although we stated our notion of homogeneity at the beginning of the introduction using types, this definition could alternatively be given in the following terms: a structure $M$ is homogeneous if and only if for every integer $n$ and $u,v \in M^n$, if there exists an automorphism of a monster model $\mathfrak{M}$ of $M$ (a sufficiently saturated model of $\mathrm{Th}(M)$) that sends $u$ to $v$, then there exists an automorphism of $M$ that sends $u$ to $v$. Our notion of profinite homogeneity is thus naturally inspired by this model-theoretic fact and its introduction is justified by the following theorem, which can be considered as the analogue of Oger's result on first-order rigidity of abelian-by-finite groups in the context of model theoretic homogeneity. Recall that a tuple $u$ of elements of a group $G$ is said to be \emph{type-determined} if, for every tuple $v$ such that $\mathrm{tp}(v)=\mathrm{tp}(u)$, the tuples $v$ and $u$ are automorphic in $G$.
	
\begin{te}\label{prof_homogeneity} A finitely generated abelian-by-finite group $G$ is homogeneous if and only if it is profinitely homogeneous, that is, every tuple of words from $G$ is profinitely rigid. More precisely, for $u \in G^n$, the tuple $u$ is type-determined \mbox{if and only if it is profinitely rigid.}
\end{te}	

As a corollary of Theorems \ref{theorem1}, \ref{counterexample_theorem} and \ref{prof_homogeneity}, we immediately obtain the following results.

\begin{co}
Affine Coxeter groups are profinitely homogeneous, i.e., in an affine Coxeter group any tuple of words (of any length) is profinitely rigid. 
\end{co}

\begin{co}Let $G_1$ and $G_2$ be non-isomorphic split crystallographic groups such that $\widehat{G}_1\simeq \widehat{G}_2$ (meaning that $G_1$ and $G_2$ have the same set of finite quotients), then $G =G_1 \times G_2$ is not profinitely homogeneous. That is, there is $u \in G^{< \omega}$ which is not profinitely rigid.
\end{co}

To the best of our knowledge this is the first example of non-profinite rigidity of (tuples of) words in abelian-by-finite groups.

%Our first main result in this direction is the following theorem, which answers the question of first-order homogeneity of irreducible affine Coxeter groups.
	
%\begin{theorem} Irreducible affine Coxeter groups are profinitely homogeneous. More generally this is true of split crystallographic groups that are generated by a finite number of finite subgroups $G_1,...,G_n$ such that the centralizer of $G_i \cap G_{i+1}$ is contained in $G_i$ for every $i \in [1, n]$.\end{theorem}

Let us now discuss the irreducibility assumption in Theorem \ref{almost_h}. Notice that the group $G_1\times G_2$ that appears in Theorem \ref{counterexample_theorem} is obviously not irreducible, leading to the following open question.

\begin{qu} Are irreducible (split) crystallographic groups homogeneous?
\end{qu}

We do not know the answer to this question (except for irreducible affine Coxeter groups, for which the answer is positive by Theorem \ref{theorem1}), but our next result shows that the failure of homogeneity in a putative non-homogeneous irreducible split crystallographic group cannot be caused by the translation subgroup. 

\begin{te}\label{type_det1} Let $G$ be a split irreducible  crystallographic group. Then tuples from the translation subgroup are type-determined. More generally, if $u=(u_1,\ldots,u_n)\in G^n$ is such that the subgroup of $G$ generated by $\lbrace u_1,\ldots,u_n\rbrace$ is infinite, then $u$ is type-determined; equivalently, the tuple $u=(u_1,\ldots,u_n)\in G^n$ is profinitely rigid (cf. \ref{prof_homogeneity}).
\end{te}

In Section~\ref{another_count_sec} we will give an example showing that the assumption of irreducibility is necessary in \ref{type_det1}. Note also that the group $G_1\times G_2$ that appears in Theorem \ref{counterexample_theorem} is (by construction) non profinitely rigid, which leads to the following open question. 
	
\begin{qu} Are profinitely rigid crystallographic groups homogeneous?
\end{qu}

%This concludes the crystallographic part of the introduction, but, before moving to the hyperbolic side of the story, we leave the following question, which is outside the scope of the present article but which we believe to be of independent interest.
	
%\begin{qu} Are finitely generated free groups profinitely homogeneous?
%\end{qu}

 This concludes the crystallographic part of the introduction. We now turn to homogeneity results in hyperbolic groups. First, it is worth noting that finitely presented groups enjoying a strong rigidity property called the strong co-Hopf property are homogeneous (see Definition 1.4 and Lemma 3.5(ii) in \cite{OH11}); prominent examples of such groups are $\mathrm{SL}_n(\mathbb{Z})$ for $n\geq 3$ (and many more higher-rank lattices as a consequence of Margulis superrigidity), $\mathrm{Out}(F_n)$, $\mathrm{Aut}(F_n)$, $\mathrm{Mod}(\Sigma_g)$ for $n,g$ not too small (see for instance \cite{And20} for details), and rigid hyperbolic groups (i.e.\ hyperbolic groups that do not split non-trivially as an HNN extension or as an amalgamated product over a finite or virtually cyclic group) by the works of Sela \cite{Sel97,Sel09} and Paulin \cite{Pau97} and generalizations to hyperbolic groups with torsion (see \cite{RW14,Moi13}). For instance, the fundamental group of a closed hyperbolic $n$-manifold where $n\geq 3$ is homogeneous.

\smallskip

%$\mathrm{Out}(F_n)$ for $n \geq 3$, $\mathrm{Aut}(F_n)$ for $n \geq 2$, $\mathrm{Mod}(\Sigma_g)$ for $g\geq 4$

Hyperbolic groups that admit non-trivial splittings are much more complicated to deal with. As already mentioned, Ould Houcine and independently Perin and Sklinos proved in \cite{OH11,PS12} that finitely generated free groups are homogeneous, using tools developed by Sela and others (notably the theory of JSJ decomposition of groups and the machinery developed to solve the famous Tarski problem on the elementary equivalence of non-abelian free groups). Note that the proof of the homogeneity of $F_2$ goes back to the work of Nies \cite{Nie03}, but this case is much easier than the general case of free groups and it can be treated by means of elementary techniques. It was also proved in \cite{PS12} that the fundamental group of the orientable closed hyperbolic surface of genus $g \geq 3$ is not homogeneous, and a complete characterization of homogeneous torsion-free hyperbolic groups was later given in \cite{BP19}.

\smallskip

In the presence of torsion, new phenomena appear, and there is strong evidence that finitely generated virtually free groups are not homogeneous in general (see \cite{And18b,And21} for more details on this problem, and see also Theorem \ref{EAE} below). The main difference between the torsion-free case and the general case lies in the following fact: when a group splits as a free product $A\ast B$, any automorphism of $A$ extends (in the obvious way) to an automorphism of the whole group, whereas this is not the case if the free product is replaced with an amalgam $A\ast_C B$ with $C$ finite. At the moment, a characterization of homogeneous hyperbolic groups seems to be out of reach. However, we will prove that many torsion-generated hyperbolic groups are homogeneous (see Theorem \ref{AE} below).

\smallskip

The only known results concerning homogeneity of Gromov hyperbolic Coxeter groups are the following ones, proved by the second named author and Sklinos in \cite{PS23}: the universal Coxeter group of rank $n$ (that is the free product of $n$ copies of $\mathbb{Z}/2\mathbb{Z}$) is homogeneous (see \cite[Theorems 1.4 and 4.8]{PS23}), and one-ended hyperbolic right-angled Coxeter groups are homogeneous (see \cite[Theorem 1.5 and Proposition 4.9]{PS23}). Our main theorem is a broad generalization of these results (see \ref{AEdef} for the definition of $\mathrm{AE}$-homogeneity).

%This is a generalization of the hyperbolic part of Theorem \ref{main}.

\begin{te}\label{AE}
Let $G$ be a torsion-generated hyperbolic group. Suppose that the following condition holds: for every edge group $C$ of a reduced Stallings splitting of $G$ (see Definition \ref{Stallings}), the image of the natural map $N_G(C)\rightarrow \mathrm{Aut}(C)$ is equal to $\mathrm{Inn}(C)$. Then $G$ is $\mathrm{AE}$-homogeneous (and so homogeneous). In particular, the following groups are $\mathrm{AE}$-homogeneous:
\begin{enumerate}[(1)]
    \item[$\bullet$] hyperbolic even Coxeter groups;
    \item[$\bullet$] torsion-generated hyperbolic one-ended groups.
\end{enumerate}
\end{te}

%We also prove that every torsion-generated hyperbolic group $G$ is uniformly almost homogeneous, meaning that there exists an integer $n\geq 1$ (which only depends on $G$) such that for any $k\geq 1$ and any $u\in G^k$, the set of $k$-tuples of elements of $G$ having the same type as $u$ is the union of at most $n$ orbits under the action of $\mathrm{Aut}(G)$. Note that the group $G$ is homogeneous if and only if one can take $n=1$ in this definition. This notion was introduced by the first author in \cite{And18b}, where it was proved that finitely generated virtually free groups are uniformly almost homogeneous.

%\begin{te}\label{almost_h}Torsion-generated hyperbolic group are uniformly almost homogeneous. In particular, hyperbolic Coxeter groups are uniformly almost homogeneous.\end{te}

Theorem \ref{AE} shows that if the normalizer of each edge group in a Stallings splitting of a torsion-generated hyperbolic group is not too complicated, then the group is homogeneous. But this result does not remain true when the assumption on edge groups is removed: in Theorem \ref{EAE_intro} below, the lack of homogeneity of $G$ comes from the fact that the edge group $C$ has (in some sense) complicated normalizer in $G$. 

\begin{te}\label{EAE_intro}There exists a hyperbolic Coxeter group that is not $\mathrm{EAE}$-homogeneous. More precisely, we construct such a group of the form $A\ast_C B$ where $A,B$ are finite Coxeter groups and $C$ is a special subgroup of $A,B$. In particular, this group is virtually free.
\end{te}

The proof of Theorem \ref{EAE_intro} goes as follows: we construct an $\mathrm{EAE}$-extension $G'$ of $G$ (see Definition \ref{EAE_def}) and two elements $x,y\in G$ that are automorphic in $G'$ but not in $G$. These elements $x,y$ being automorphic in $G'$, they have the same type in $G'$, and so they have the same $\mathrm{EAE}$-type in $G$ (as $G'$ is an $\mathrm{EAE}$-extension of $G$). Hence, $G$ is not $\mathrm{EAE}$-homogeneous.

\smallskip

This result strongly suggests that not all hyperbolic Coxeter groups are homogeneous, as Sela proved quantifier elimination down to Boolean combinations of $\mathrm{AE}$-formulas in torsion-free hyperbolic groups. However, the analogue of this quantifier elimination result in the presence of torsion remains an open problem.

\smallskip

In Section \ref{direct_products} we prove that, under certain conditions, homogeneity behaves well with respect to direct products. However, the following example shows that the direct product of two non-elementary homogeneous hyperbolic groups need not be homogeneous. 

\begin{ex}Consider $G=F(a,b)=\langle a,b\rangle$ and $G'=F(a',b')=\langle a',b'\rangle$ two free groups of rank 2. The elements $a$ and $a'$ have the same type in $F(a,b)\times F(a',b')$ (take the automorphism swapping $G$ and $G'$ in the obvious way). The free group $F(a,b,c)$ of rank three is an elementary extension of $F(a,b)$ (see \cite{Sel06,KM06}), so $F(a,b,c)\times F(a',b')$ is an elementary extension of $F(a,b)\times F(a',b')$, and thus $a$ and $a'$ still have the same type in $F(a,b,c)\times F(a',b')$. But any automorphism of $F(a,b,c)\times F(a',b')$ maps each factor to itself, therefore there is no automorphism mapping $a$ to $a'$. Hence, $F(a,b,c)\times F(a',b')$ is not homogeneous. \end{ex}

The following result shows that this construction, which relies crucially on the fact that $F_3$ is not strictly minimal (which means that it contains a proper elementarily embedded subgroup), cannot be extended to direct product of torsion-generated non virtually cyclic hyperbolic groups.

\begin{te}\label{strictly_minimal_intro}Every torsion-generated hyperbolic group $G$ is strictly minimal. In fact, $G$ has no proper $\mathrm{AE}$-embedded subgroup.\end{te}

\begin{rk}
We refer the reader to Subsection 11.2 of \cite{GLS} for a characterization of strict minimality in torsion-free hyperbolic groups. Note that no such characterization exists (yet) for hyperbolic groups with torsion.
\end{rk}

This result leads to the following questions.

\begin{qu}Let $G_1$ and $G_2$ be homogeneous torsion-generated hyperbolic groups. Is $G_1\times G_2$ homogeneous?
\end{qu}

\begin{qu}Let $G_1$ and $G_2$ be homogeneous strictly minimal hyperbolic groups. Is $G_1\times G_2$ homogeneous?
\end{qu}

In Section \ref{direct_products}, we give a very partial answer to the second question.

\subsection{Acknowledgments} We are grateful to the anonymous referees for their careful reading of earlier versions of this paper and for their many valuable remarks and suggestions. We also thank Davide Carolillo for pointing out a gap in an earlier version of the proof of the homogeneity of affine Coxeter groups.

\subsection{Notation}\label{section_notation}

Throughout the paper we adopt the following notations: for $g \in G$, $\mathrm{ad}(g)$ denotes the inner automorphism $h\in G\mapsto g h g^{-1}$. Sometimes, we also use the notation $h^g$ to denote the conjugate $g h g^{-1}$ of $h$. If a group $G$ acts on a set $X$, we denote by $G\cdot x$ the orbit of $x\in X$ under the action of $G$.

\section{Homogeneity and rigidity in split crystallographic groups}\label{homogeneity_crystallographic}

\subsection{Preliminaries on crystallographic groups}\label{prelim_crystallo}

We recall that an action of a group $G$ by homeomorphisms on a topological space $X$ is said to be \emph{properly discontinuous} if for every compact subset $K \subset X$, there are only finitely many $g\in G$ such that $g\cdot K \cap K \neq \varnothing$. A \emph{crystallographic group} $G$ of dimension $n\geq 1$ is a properly discontinuous (equivalently, discrete) and cocompact subgroup of $\mathrm{Isom}(\mathbb{R}^n)$, the group of isometries of the Euclidean space $\mathbb{R}^n$. Note that this implies that any crystallographic group $G$ is finitely presented, as $G$ is isomorphic to the orbifold fundamental group of $\mathbb{R}^n/G$.

Recall that $\mathrm{Isom}(\mathbb{R}^n)$ is isomorphic to a semidirect product $\mathbb{R}^n\rtimes \mathrm{O}_n(\mathbb{R})$. If $G$ is a crystallographic group, the normal subgroup $H=G\cap \mathbb{R}^n$ is called the \emph{translation subgroup} of $G$. By the First Bieberbach Theorem, this subgroup $H$ is isomorphic to $\mathbb{Z}^n$ and is of finite index in $G$. The finite quotient $G/H$ is called the \emph{point group} of $G$, denoted by $G_0$. Moreover, $H$ is maximal abelian in $G$. Conversely, Zassenhaus proved in 1948 that a group $G$ is isomorphic to a crystallographic group of dimension $n\geq 1$ if it has a normal subgroup $H$ which is isomorphic to $\mathbb{Z}^n$, of finite index and maximal abelian.

%https://math.stackexchange.com/questions/1708832/discrete-subgroups-and-discontinuous-subgroups-of-the-isometries-of-the-euclidea

%https://math.stackexchange.com/questions/580808/discrete-subgroups-of-isometry-group-mathbbrn

%https://math.stackexchange.com/questions/3211710/properly-discontinuous-actions-and-discrete-groups-in-complete-riemannian-manifo

This description gives rise to a short exact sequence $1\rightarrow H\rightarrow G \overset{p}{\rightarrow} G_0\rightarrow 1$. This sequence is not split in general (recall that the sequence is said to be \emph{split} if there exists a morphism $q:G_0\rightarrow G$ such that $p\circ q=\mathrm{id}_{G_0}$, in which case the group $G$ can be written as a semidirect product of the form $H\rtimes G_0$ where the underlying morphism $G_0\rightarrow \mathrm{Aut}(H)$ is given by $g\in G_0\mapsto \mathrm{ad}(q(g))_{\vert H}$); in fact, some crystallographic groups are torsion-free. However, there is still a natural action of $G_0$ on $H$ induced by the action of $G$ on $H$ by conjugation; more precisely, the action of $G_0$ on $H$ is defined for $g\in G_0$ and $h\in H$ by $g\cdot h=g'hg'^{-1}$ where $g'$ is any preimage of $g$ by $p$. This action is faithful since $H$ is maximal abelian in $G$. This gives rise to an injective morphism $\rho:G_0\rightarrow \mathrm{GL}_n(\mathbb{Z})$, called the {\em integral holonomy representation}. Conversely, if a group $G$ has a normal subgroup $H\simeq \mathbb{Z}^n$ of finite index such that the natural action of $G_0=G/H$ on $H$ is faithful, then clearly $H$ is maximal abelian, and thus $G$ is isomorphic to a crystallographic group of dimension $n\geq 1$.

\begin{de}\label{def_irreducible}We say that $G$ is \emph{irreducible} if $\rho$, viewed as a representation $G_0\rightarrow\mathrm{GL}_n(\mathbb{Q})$, is irreducible (meaning that the only linear subspaces of $\mathbb{Q}^n$ that are stable under the action of $\rho(G_0)$ are $\mathbb{Q}^n$ and $\lbrace 0\rbrace$).\end{de}

%\textcolor{red}{I don't understand, $\mathbb{Q}$ IS a field, so why do you have to use $K$? I mean $\mathbb{Z}$ is not a field and so in order to define reducibility for $\mathbb{Z}$-representations you need a field but do you need to use $K$? I mean why don't we simply say: viewed as a representation $G_0 \rightarrow \mathrm{GL}_n(\mathbb{Q})$, is irreducible (that is, the only linear subspaces of $\mathbb{Q}^n$ that are stable under the action of $\rho(G_0)$ are $\mathbb{Q}^n$ and $\lbrace 0\rbrace$)?}

%\textcolor{red}{Explain that this is equivalent to another standard definition, so there is no ambiguity.}

%https://phavi.umcs.pl/at/attachments/2018/0711/202235-lublin18.pdf

%https://arxiv.org/pdf/1910.09845

If the short exact sequence is split, then $G$ is called a \emph{split crystallographic group}. Equivalently, according to the previous paragraph, a split crystallographic group is a group $G$ isomorphic to a semidirect product of the form $\mathbb{Z}^n \rtimes_{\rho} G_0$, where the morphism $\rho:G_0 \rightarrow \mathrm{GL}_n(\mathbb{Z})$ is injective. 

We will need the following easy lemmas.

\begin{lemme}\label{finitely_many}
Let $G$ be a crystallographic group. Then $G$ has only finitely many conjugacy classes of finite subgroups.
\end{lemme}

\begin{proof}
Every finite subgroup $F$ of $G$ has a fixed point $p\in X=\mathbb{R}^n$; indeed, if $x$ is any point in $X=\mathbb{R}^n$, the point $p=\frac{1}{\vert F\vert}\sum_{g\in F}g\cdot x$ (that is the barycenter of $F\cdot x$) is fixed by $F$. By cocompacity of the action, there exists a compact $K$ in $X$ such that any point in $X$ has a $G$-translate in $K$, so there exists $g\in G$ such that $g\cdot p$ is in $K$, and this point is fixed by $gFg^{-1}$, which shows that every finite subgroup of $G$ has a conjugate that fixes a point of the compact $K$. Finally, since the action is properly discontinuous, the set $\lbrace g\in G \ \vert \ g\cdot K \cap K \neq \varnothing \rbrace$ is finite, which shows that there is only a finite number of subgroups (necessarily finite) of $G$ that fix a point of $K$.\end{proof}

\begin{lemme}\label{no_normal_finite_subgroup}Let $G$ be a crystallographic group. Then $G$ does not have any non-trivial finite normal subgroup.\end{lemme}

\begin{proof}Suppose that $G$ is a crystallographic group of dimension $n\geq 1$. Let $F$ be a normal finite subgroup of $G$, and let $p\in \mathbb{R}^n$ be a point fixed by $F$ (cf. the beginning of the proof of \ref{finitely_many}). By the First Bieberbach Theorem, $G$ contains $n$ translations $t_1,\ldots,t_n$ generating an abelian group of rank $n$. It follows that the points $p,t_1(p),\ldots,t_n(p)$ are $n+1$ points not lying in any affine hyperplane of $\mathbb{R}^n$. But $F=t_iFt_i^{-1}$ fixes $t_i(p)$, so $F$ is trivial.\end{proof}

%\begin{proof}First, suppose that $n\geq 2$. Let $F$ be a normal finite subgroup of $G$, and let $p\in \mathbb{R}^n$ be a point fixed by $F$. By Bieberbach's theorem, $G$ contains $n$ non-colinear translations $t_1,\ldots,t_n$. Note that $F=t_iFt_i^{-1}$ fixes $t_i(p)$. But the points $p,t_1(p),\ldots,t_n(p)$ are $n+1$ points not lying in any affine hyperplane of $\mathbb{R}^n$, and thus $F$ is trivial. Then, suppose that $n=1$. Let $t\in G$ be a non-trivial translation. Then $F$ fixes $t(p)$, so $F$ is either the trivial subgroup or $F=\langle r\rangle$ where $r$ denotes the reflection fixing the middle $m$ of $[p,t(p)]$. But the same argument is true for $t^2(p)$ instead of $t(p)$, and $r$ does not fix the middle of $[p,t^2(p)]$. Hence, $F$ is trivial (and $G$ is isomorphic to $\mathbb{Z}$ or to the infinite dihedral group).\end{proof}

\begin{co}
A finitely generated virtually abelian group is a crystallographic group if and only if it does not have any non-trivial finite normal subgroup.
\end{co}

\begin{proof}
By Lemma \ref{no_normal_finite_subgroup}, a crystallographic group does not have any non-trivial finite normal subgroup. Conversely, let $G$ be a finitely generated virtually abelian group without a non-trivial finite normal subgroup. Let $H$ be a finite-index abelian normal subgroup of $G$ with $[G:H]$ minimal. This subgroup is torsion-free since $G$ has no non-trivial finite normal subgroup, and it is finitely generated since $[G:H]$ is finite, thus $H\simeq \mathbb{Z}^n$ for some $n\geq 1$. Moreover, $H$ is maximal abelian since $[G:H]$ is minimal. Therefore, $G$ is isomorphic to a crystallographic group of dimension $n$ by Zassenhaus' Theorem recalled above.\end{proof}

\subsection{Type-determinacy and almost-homogeneity in crystallographic groups}\label{type-determinacy}

	Let us recall that the notions of $\mathrm{AE}$-type, $\mathrm{AE}$-homogeneity, etc. were defined in the introduction (see Definition \ref{AEdef}).

\begin{de}\label{type-det}Let $G$ be a group, and let $u$ be a tuple of elements of $G$. We say that $u$ is $\mathrm{AE}$-determined if any tuple that has the same $\mathrm{AE}$-type as $u$ is in the $\mathrm{Aut}(G)$-orbit of $u$.
\end{de}

We will prove the following result (for the definition of an irreducible crystallographic group, see Definition \ref{def_irreducible}). 

\begin{te}\label{main_crystallo0}
Let $G$ be an irreducible split crystallographic group. Let $u=(u_1,\ldots,u_{\ell})$ be a tuple of elements of $G$. If the subgroup $U$ of $G$ generated by $\lbrace u_1,\ldots,u_{\ell}\rbrace$ is infinite, then $u$ is $\mathrm{AE}$-determined.
\end{te}

\begin{rk}
In Subsection~\ref{counter}, we will give a counterexample showing that the result is false if one removes the assumption that $G$ is irreducible. Moreover, in Section \ref{homogeneity_Coxeter}, we will see that this result is not true if the subgroup $U$ is finite.
\end{rk}

Let us introduce the following definition, which will be used in various places throughout the paper (compare with \cite[Definition 5.3]{And21}).

\begin{de}\label{class-permuting}Let $G$ be a group and let $\varphi$ be an endomorphism of $G$. We say that $\varphi$ is \emph{class-permuting} if the following condition holds: for any two non-conjugate finite subgroups $H,H'$ of $G$, $\varphi(H)$ and $\varphi(H')$ are non-conjugate.\end{de}

\begin{rk}
The terminology comes from the following easy fact: if $G$ has only finitely many conjugacy classes of finite subgroups, then a class-permuting morphism $\varphi$ induces a permutation of the conjugacy classes of finite subgroups of $G$. In particular, $\varphi$ is injective on each finite subgroup of $G$.
\end{rk}

The following key lemma will be used in the proof of Theorem \ref{main_crystallo0}.

\begin{lemme}\label{key_lemma_1}
Let $G=\langle s_1,\ldots,s_n\ \vert \ r_1,\ldots,r_k\rangle$ be a finitely presented group, let $H$ be a definable finitely generated subgroup of $G$ (without parameters) and let $u,u'$ be two finite tuples of elements of $G$. Suppose that $G$ has only finitely many conjugacy classes of finite subgroups. If $u$ and $u'$ have the same type in $G$, then there is an endomorphism $\varphi$ of $G$ such that the following conditions hold:
\begin{enumerate}[(1)]
    \item $\varphi(u)=u'$;
    \item $\varphi(H)\subset H$;
    \item $\varphi$ is class-permuting (see Definition \ref{class-permuting}).
\end{enumerate}
Moreover, if $H$ is isomorphic to $\mathbb{Z}^{\ell}$ and $u$ and $u'$ have the same type in $G$, then there is an endomorphism $\varphi$ of $G$ that satisfies conditions (1) to (3) and that is injective on $H$. Furthermore, if $H$ is definable by a universal formula, then it is sufficient to assume that $u$ and $u'$ have the same $\mathrm{EA}$-type in $G$ to conclude that there exists such an endomorphism $\varphi$ of $G$. 
\end{lemme}

\begin{proof}
The map $\varphi\in \mathrm{End}(G)\mapsto (\varphi(s_1),\ldots,\varphi(s_n))\in G^n$ induces a bijection between $\mathrm{End}(G)$ and $E=\lbrace (x_1,\ldots,x_n)\in G^n \ \vert \ r_i(x_1,\ldots,x_n)=1 \text{ for } 1\leq i\leq k\rbrace$ (its inverse is $(x_1,\ldots,x_n)\in E\mapsto (f : x_i\mapsto s_i)$). Write each component $u_i$ of $u$ as a word $w_i(s_1,\ldots,s_n)$ in the generators of $G$, and let $\theta(y)$ be a first-order formula defining $H$, meaning that $H=\lbrace g\in G \ \vert \ G\models \theta(g)\rbrace$. Let $h_1(s_1,\ldots,s_n),\ldots,h_m(s_1,\ldots,s_n)$ be a generating set for $H$. Let $F_1,\ldots,F_k$ be representatives of the conjugacy classes of finite subgroups of $G$, and write the elements of $F_i$ as words $f_{i,1}(s_1,\ldots,s_n),\ldots,f_{i,\vert F_i\vert}(s_1,\ldots,s_n)$. Now consider the following condition:
\begin{enumerate}[$(\star)$]
\item the identity of $G$ is a morphism, maps $u$ to $u$, satisfies $\mathrm{id}(H)\subset H$ and maps $F_i$ and $F_j$ to non-conjugate subgroups if $i\neq j$.
%\gianluca{Simon, you wrote this but that I look at it is this correct? Could you fix it?} Simon: I don't see the point...
\end{enumerate}
Then condition $(\star)$ can be expressed in first-order logic via the following formula, where $z=(z_1,\ldots,z_{t})$ denotes a tuple of variables of the same arity as $u$, and $x$ denotes a $n$-tuple of variables:
\begin{align*}
\gamma(z): \exists x \forall & g \ \bigwedge_{i=1}^{t}(z_i=w_i(x)) \ \bigwedge_{i=1}^k(r_i(x)=1) \bigwedge_{i=1}^m\theta(h_i(x))  \\
           &\bigwedge\limits_{\substack{i,j=1 \\ i \neq j \\ \vert F_i\vert\geq \vert F_j\vert}}^{k}\bigvee_{d=1}^{\vert F_i\vert}\bigwedge_{\ell=1}^{\vert F_j\vert}(gf_{i,d}(x)g^{-1}\neq f_{j,\ell}(x)) \\
\end{align*}
Note that $G\models \gamma(u)$ because one can take $x=(s_1,\ldots,s_n)$. Since $u$ and $u'$ have the same type, we have $G\models \gamma(u')$, which provides a tuple $(g_1,\ldots,g_n)\in G^n$ such that the map $s_i\mapsto g_i$ for $1\leq i\leq n$ extends to an endomorphism $\varphi$ of $G$ mapping $u$ to $u'$, such that $\varphi(H)\subset H$ and such that $\varphi(F_i)$ and $\varphi(F_j)$ are non-conjugate for $i\neq j$. 

\smallskip \noindent Now, suppose that $H$ is isomorphic to $\mathbb{Z}^{\ell}$. Define $r=2^{\ell}-1$. Observe that $\vert \mathbb{Z}^{\ell}/2\mathbb{Z}^{\ell}\vert>r$, or equivalently that there exists a set $X\subset \mathbb{Z}^{\ell}$ of cardinality $r$ such that the sum of any two distinct elements in $X$ is not divisible by 2. Therefore, using multiplicative notation, there exist $r$ elements $w_1(s_1,\ldots,s_n),\ldots,w_r(s_1,\ldots,s_n)\in H$ such that, for every $h\in H$ and $1\leq i\neq j\leq r$, $w_i(s_1,\ldots,s_n)w_j(s_1,\ldots,s_n)$ is not the square of $h$. Then, the following first-order formula $\gamma'(z)$ encodes the same morphism $\varphi$ as $\gamma(z)$ and ensures, moreover, that the rank of $\varphi(H)$ is at least (and thus equal to) $r$, which means that $\varphi_{\vert H}$ is injective:
\begin{align*}
\gamma'(z): \exists x \forall y \forall & g \ \bigwedge_{i=1}^{t}(z_i=w_i(x)) \ \bigwedge_{i=1}^k(r_i(x)=1) \bigwedge_{i=1}^m\theta(h_i(x))  \\
           &\bigwedge\limits_{\substack{i,j=1 \\ i \neq j \\ \vert F_i\vert\geq \vert F_j\vert}}^{k}\bigvee_{d=1}^{\vert F_i\vert}\bigwedge_{\ell=1}^{\vert F_j\vert}(gf_{i,d}(x)g^{-1}\neq f_{j,\ell}(x)) \\
           & \bigwedge_{i=1}^r\theta(w_i(x))\wedge \left(\theta(y)\implies \bigwedge_{1\leq i\neq j\leq r} w_i(x)w_j(x)\neq y^2\right)
\end{align*}
 Moreover, note that if $H$ is definable by a universal formula, then $\gamma(z)$ is an $\mathrm{EA}$-formula.\end{proof}

We will also need the following result.

\begin{fact}[{\cite[Proposition 3.4]{PS23}}]\label{definable}
Let $G=H\rtimes G_0$ be a split crystallographic group. The subgroup $H$ is definable in $G$ without parameters. More precisely, if $G_0$ has order $m$, $H$ is definable by the following universal formula: $\chi(x): \forall y \ ([x, y^m] = 1).$
\end{fact}

Theorem \ref{main_crystallo0} will follow easily from the follow proposition. We recall that, as defined in Subsection~\ref{section_notation}, for $g \in G$, $\mathrm{ad}(g)$ denotes the inner automorphism induced by $g$ on $G$.

\begin{prop}\label{15janvier2026}
Let $G=H\rtimes G_0$ be an irreducible split crystallographic group. Let $\ell\geq 1$ and $u=(u_1,\ldots,u_{\ell})$ be a tuple of elements of $G$. Suppose that the subgroup $U$ of $G$ generated by $\lbrace u_1,\ldots,u_{\ell}\rbrace$ is infinite and that there exists a class-permuting endomorphism $\theta$ of $G$ such that $\theta(u)=u$ and $\theta(H)\subset H$. Then $\theta$ is an automorphism of $G$.
\end{prop}

\begin{proof}
By Lemma \ref{finitely_many}, $G$ has only finitely many conjugacy classes of finite subgroups. Let $F_1,\ldots,F_m$ be a collection of representatives of these conjugacy classes. Let $\mathcal{C}=\lbrace [F_1],\ldots,[F_k]\rbrace$, where $[F_i]$ denotes the conjugacy class of $F_i$ in $G$. Since the set $\mathcal{C}$ is finite, and since $\theta$ permutes the conjugacy classes of finite subgroups of $G$ (see the remark following Definition \ref{class-permuting}), there is an integer $N\geq 1$ such that $\theta^N$ acts trivially on $\mathcal{C}$. Hence, $\theta^N([G_0])=[G_0]$ and thus there is an element $g\in G$ such that $\theta^N(G_0)=gG_0g^{-1}$. Write $g=hg_0$ for some $h\in H$ and $g_0\in G_0$. The endomorphism $\theta'=\mathrm{ad}(h^{-1})\circ \theta^N$ satisfies $\theta'(H)\subset H$ and $\theta'(G_0)=G_0$. As $G_0$ has finite automorphism group, there is an integer $M\geq 1$ such that $\theta'^M$ induces the identity of $G_0$. Define $f=\theta'^M$. As $\theta'=\mathrm{ad}(h^{-1})\circ \theta^N$, there is an element $h'\in H$ such that $\theta'^M=\mathrm{ad}(h')\circ \theta^{MN}$.

\smallskip \noindent
We will prove that $f$ is in fact the identity of $G$ (not only of $G_0$). It follows that $\theta'$ is an automorphism of $G$, and thus that $\theta$ is an automorphism of $G$.

%We will prove that $f$ is in fact the identity of $G$ (not only of $G_0$). But $f=\mathrm{ad}(h^{-1})\circ \theta^N$ with $\theta=\varphi'\circ \varphi$, so $\varphi'$ must be surjective, and thus $\varphi'$ must be an automorphism of $G$ (indeed $G$ is Hopfian, as it embeds in $\mathrm{GL}_{n+1}(\mathbb{Z})$). Moreover, $\varphi'(u')=u$, which shows that $u$ and $u'$ are automorphic in $G$.

\smallskip \noindent
It remains to prove that $f$ is the identity of $G$. Since $U$ is infinite (by assumption), it contains an element $x$ of infinite order (indeed, as well known, finitely generated linear periodic groups are finite), hence $y:=x^{\vert G_0\vert}$ is a non-trivial element of $H$. But $\theta(u)=u$, so $\theta_{\vert U}=\mathrm{id}_{ U}$ and $f_{\vert U\cap H}=\mathrm{id}_{U\cap H}$ (since $f=\mathrm{ad}(h')\circ \theta^{MN}$ with $h'\in H$), and therefore $f(y)=y$.

\smallskip \noindent
By assumption, $G$ is irreducible, which means that the linear representation $\rho : G_0\rightarrow \mathrm{GL}_n(\mathbb{Z})\subset \mathrm{GL}_n(\mathbb{Q})$ is irreducible (see Definition \ref{def_irreducible}), where $\rho$ denotes the action of $G_0$ on $H$ in the semidirect product $G=H\rtimes G_0$.  Let $V$ be the linear subspace of $\mathbb{Q}^n$ spanned by the finite set $G_0\cdot y$. By irreducibility and the fact that $y$ is non-trivial, $V$ must coincide with $\mathbb{Q}^n$. It follows that $G_0\cdot y$ contains $n$ vectors $h_1,\ldots, h_n\in H=\mathbb{Z}^n$ that are linearly independent over $\mathbb{Q}$. Let $H'$ be the subgroup of $H$ generated by $\lbrace h_1,\ldots, h_n\rbrace$ and let $A\in \mathrm{M}_n(\mathbb{Z})$ be the matrix whose columns are the vectors $h_1,\ldots,h_n$ written in the canonical basis $e_1=(1,0,\ldots,0),\ldots,e_n=(0,\ldots,0,1)$ of $H=\mathbb{Z}^n$. By the inverse of matrix formula, there is a matrix $B\in \mathrm{M}_n(\mathbb{Z})$ (namely $B=A^\intercal$) such that $AB=dI_n$, with $d=\mathrm{det}(A)\neq 0$ since $h_1,\ldots, h_n$ are linearly independent over $\mathbb{Q}$. It follows that $dH\subset H'$ (using additive notation).

%\emph{A fortiori}, these vectors are linearly independent over $\mathbb{Z}$, and thus $H'$ is a subgroup of $H=\mathbb{Z}^n$ isomorphic to $H$. 

%(in other words, the subgroup $H'=\langle h_1,\ldots, h_n\rangle$ is a lattice in $K^n$)

\smallskip \noindent
For each $1\leq i\leq n$, as $h_i$ belongs to $G_0\cdot y$, one can write $h_i=g_i(y)$ for some $g_i\in G_0$. Using the fact that $G_0$ acts on $H$ by conjugation in the semidirect product $H\rtimes G_0$, let us write $h_i=g_iyg_i^{-1}$. As we have proved in the previous paragraphs that $f(y)=y$ and that the restriction of $f$ to $G_0$ is the identity, we have $f(h_i)=f(g_iyg_i^{-1})=f(g_i)f(y)f(g_i)^{-1}=g_iyg_i^{-1}=h_i$ for each $1\leq i\leq n$. Hence $f$ coincides with the identity on $H'$.

\smallskip \noindent
Now, recall that $dH\subset H'=\langle h_1,\ldots,h_n\rangle$ with $d\in\mathbb{Z}^{\ast}$. Therefore, for every integer $1\leq i\leq n$, the element $de_i$ belongs to $H'$. But we have just proved that $f$ is the identity on $H'$, so we have $f(de_i)=de_i$, hence $df(e_i)=de_i$ and $f(e_i)=e_i$. Conclusion: $f_{\vert H}$ is the identity of $H$, and so $f$ is the identity of $G$.\end{proof}

We are now ready to prove the main theorem.

\begin{proof}[Proof of Theorem \ref{main_crystallo0}]
Let $G=H\rtimes G_0$ be an irreducible split crystallographic group. Let $\ell\geq 1$ be an integer and let $u=(u_1,\ldots,u_{\ell})\in G^{\ell}$ and $u'=(u'_1,\ldots,u'_{\ell})\in G^{\ell}$ be two $\ell$-tuples. Suppose that the subgroup $U$ of $G$ generated by $\lbrace u_1,\ldots,u_{\ell}\rbrace$ is infinite and that $u$ and $u'$ have the same EA-type in $G$. By Lemma \ref{finitely_many}, $G$ has only finitely many conjugacy classes of finite subgroups. By Lemma \ref{key_lemma_1} and Fact \ref{definable}, there exist two class-permuting endomorphisms $\varphi,\varphi'$ of $G$ such that $\varphi(u)=u'$, $\varphi'(u')=u$, $\varphi(H)\subset H$ and $\varphi'(H)\subset H$. Hence, $\theta=\varphi'\circ\varphi$ satisfies the assumptions of Proposition \ref{15janvier2026}. This proposition tells us that $\theta$ is an automorphism of $G$. But $\theta=\varphi'\circ \varphi$, so $\varphi'$ must be surjective, and thus $\varphi'$ must be an automorphism of $G$ (indeed $G$ is Hopfian, as it embeds in $\mathrm{GL}_{n+1}(\mathbb{Z})$). Moreover, $\varphi'(u')=u$, which shows that $u$ and $u'$ are automorphic in $G$.\end{proof}

The following definition was introduced by the first-named author in \cite{And18b}, where it was proved that virtually free groups are uniformly almost homogeneous (in fact, uniformly almost $\mathrm{AE}$-homogeneous).

\begin{de}
A group $G$ is almost homogeneous if for any $k\geq 1$ and $u\in G^k$, there exists an integer $n\geq 1$ such that the set of $k$-tuples having the same type as $u$ is the union of $\leq n$ orbits under the action of $\mathrm{Aut}(G)$, and $G$ is uniformly almost homogeneous if $n$ can be chosen independently from $u$ and $k$. Note that $G$ is homogeneous if and only if one can take $n= 1$ in this definition.\end{de}

We will prove the following corollary of Theorem \ref{main_crystallo0}.

\begin{co}\label{coro_almost_homogeneous}
Irreducible split crystallographic groups are uniformly almost homogeneous (in fact, uniformly almost $\mathrm{AE}$-homogeneous).
\end{co}

We need the following easy lemma, whose proof is very similar to that of Lemma \ref{key_lemma_1}.

\begin{lemme}\label{key_lemma_2}
Let $G=\langle s_1,\ldots,s_n\ \vert \ r_1,\ldots,r_k\rangle$ be a finitely presented group, and let $u,u'$ be two finite tuples of elements of $G$. Let $F$ be a finite subset of $G$. If $u$ and $u'$ have the same existential type in $G$, then there is an endomorphism $\varphi$ of $G$ such that $\varphi(u)=u'$ and $\varphi$ is injective on $F$.
\end{lemme}

\begin{proof}[Proof of Corollary \ref{coro_almost_homogeneous}]
Let $G$ be an irreducible split crystallographic group. Let $\ell\geq 1$ and $u=(u_1,\ldots,u_{\ell})\in G^{\ell}$. According to Theorem \ref{main_crystallo0}, if the subgroup $U$ of $G$ generated by $\lbrace u_1,\ldots,u_{\ell}\rbrace$ is infinite, then $u$ is $\mathrm{AE}$-determined. So, let us assume that $U$ is finite. Let $m\geq 1$ denote the number of conjugacy classes of finite subgroups of $G$ and let $n=\vert \mathrm{Aut}(U)\vert$. Define $N=nm+1$. Let $v_1,v_2\ldots,v_N$ be tuples such that $\mathrm{tp}_{\exists}(v_k)=\mathrm{tp}_{\exists}(u)$ for every $1\leq k\leq N$. Therefore, for every $1\leq k\leq N$, the subgroup $V_k$ of $G$ generated by the components of $v_k$ is finite and isomorphic to $U$. Moreover, according to the (strong) pigeonhole principle, there are at least $n+1$ subgroups in the collection $\lbrace V_1,\ldots,V_N\rbrace$ that belong to the same conjugacy class. Hence, after renumbering $v_1,\ldots,v_N$ and replacing $v_1,\ldots,v_N$ with conjugates if necessary, we can assume that $V_1=\ldots = V_{n+1}$. By Lemma~\ref{key_lemma_2}, for every $1\leq k\leq n+1$, there exists a morphism $\varphi_k:G\rightarrow G$ such that $\varphi_k(u)=v_k$ and $\varphi$ is injective on $U$. Thus the restriction of $\varphi_k$ to $U$ is an isomorphism between $U$ and $V_k$ that maps $u$ to $v_k$. Hence, since $\mathrm{Isom}(U,V_k)$ is finite of order $n$ and $V_1=\ldots = V_{n+1}$, there are $2\leq k < t < m+2$ such that ${\varphi_k}_{\vert U}={\varphi_{t}}_{\vert U}$, therefore $u_k=u_{t}$.\end{proof}

\subsection{Non-homogeneous crystallographic groups}\label{counter}

In this section, we prove Theorem~\ref{counterexample_theorem}. More precisely, we prove the following result.

\begin{te}\label{counterexample_theorem2}Let $G_1=H_1\rtimes K_1$ and $G_2=H_2\rtimes K_2$ be non-isomorphic split crystallographic groups such that $\widehat{G}_1\simeq \widehat{G}_2$, then $G_1 \times G_2$ is not homogeneous. More precisely, writing $K_1=\lbrace k_1,\ldots,k_n\rbrace$, the tuple $(k_1,\ldots,k_n)$ is not type-determined (see Def.~\ref{type-det}).\end{te}

For every integer $n$ such that the class number of the cyclotomic field $\mathbb{Q}(\zeta_n)$ is strictly greater than 1 (this is true for every $n\geq 85$), there exist split crystallographic groups $G_1,G_2$ of dimension $\phi(n)$ such that $G_1\not\simeq G_2$ but $\widehat{G}_1\simeq \widehat{G}_2$ (see \cite{bri}), where $\phi(n)$ is Euler's totient function. For example, for any $p\geq 23$, there are such groups of the form $\mathbb{Z}^{p-1} \rtimes\mathbb{Z}/p\mathbb{Z}$.

The failure of homogeneity in Theorem \ref{counterexample_theorem2} comes from the point group $K_1\times K_2$ of $G_1 \times G_2$, but we will give another example showing that elements of the translation subgroup are not type-determined in general (this second example should be compared with Theorem~\ref{main_crystallo0}, stating that elements from the translation subgroup are type-determined provided that $G$ is irreducible).

\subsubsection{First counterexample (proof of Theorem \ref{counterexample_theorem2})}\label{count_sec}

%Recall that if $G$ is a split crystallographic group and $T$ is its translation subgroup, then $T$ is the maximal normal torsion-free abelian subgroup of $G$ and so every automorphism of $G$ sends $T$ onto $T$. The same happens replacing $G$ with $\widehat{G}$ and $T$ with $\widehat{T}$.

We recall that if $H$ and $K$ are finitely generated residually finite groups and $\phi : K\rightarrow \mathrm{Aut}(H)$ is a morphism, the inclusions $H\subset \widehat{H}$ and $K\subset \widehat{K}$ induce an isomorphism $\widehat{H\rtimes_{\phi}K}\simeq \widehat{H}\rtimes_{\widehat{\phi}}\widehat{K}$ where $\widehat{\phi}$ denotes the composition of $\widehat{K}\rightarrow\widehat{\mathrm{Aut}(H)}$ and $\widehat{\mathrm{Aut}(H)}\rightarrow\mathrm{Aut}(\widehat{H})$ (see for instance \cite[Proposition~2.6]{GZ11}).

\smallskip \noindent
Let $A = \mathbb{Z}^n \rtimes_\alpha K$ and $B = \mathbb{Z}^n \rtimes_\beta K$ be irreducible split crystallographic groups such that, letting $T = \mathbb{Z}^n$ (the translation subgroup), we have the following:
\begin{enumerate}[(1)]
	\item $T \rtimes_\alpha K \not\cong T \rtimes_\beta K$;
	\item $\widehat{T} \rtimes_\alpha K \cong \widehat{T} \rtimes_\beta K$.
	\end{enumerate}
Now, let $G = A \times B \cong (T_A \oplus T_B) \rtimes_\gamma (K_1 \times K_2)$ with $T_A = T_B = T$, $K_1 = K_2 = K$
 and:
$$\gamma(k_1, k_2)((t_1, t_2)) = (\alpha(k_1)(t_1), \beta(k_2)(t_2)).$$
Clearly this is a split crystallographic group, as $T_1 \oplus T_2$ has the additive structure of a finitely generated torsion-free abelian group and the action $\gamma$ is faithful. 

\smallskip \noindent
Now, fix an isomorphism:
\[f: \widehat{T} \rtimes_\alpha K_1 \cong \widehat{T} \rtimes_\beta K_2.\]Then, since $f$ is an isomorphism, the group $\widehat{T} \rtimes_\beta K_2$ admits an internal semi-direct product decomposition of the form $\widehat{T} \rtimes f(K_1)$. We will need the following lemma.

\begin{lemme}\label{WLOG}Let $G$ be a group. Suppose that $G$ splits as a semidirect product in two different ways: $G=H\rtimes K$ and $G=H\rtimes K'$ with $H$ abelian. Then the map $\varphi : G\rightarrow G$ that is the identity on $H$ and that maps every $k\in K$ to the unique $k'\in K'$ such that $k=hk'$ with $(h,k')\in H\times K'$ is an automorphism of $G$.
\end{lemme}

\begin{rk}
Note that this lemma is not true when $H$ is not abelian: for $n\geq 6$, the symmetric group $S_n$ can be written as $A_n\rtimes \langle (12)\rangle=A_n\rtimes \langle (12)(34)(56)\rangle$, but for $n\geq 7$ there is no automorphism of $S_n$ mapping $(12)$ to $(12)(34)(56)$.
\end{rk}

\begin{proof}
We just have to prove that $\varphi$ is a morphism. Let $g_1=h_1k_1$ and $g_2=h_2k_2$ in $H\rtimes K$. We have $g_1g_2=
h_3k_3$ with $h_3=h_1k_1h_2k^{-1}_1\in H$ and $k_3=k_1k_2\in K$. Write $k_3=h'_{3}k'_3$ with $h'_3\in H$ and $k'_3\in K'$, so that $\varphi(g_1g_2)=h_3k'_3$. Then, write $k_i=h'_ik'_i$ for $i\in\lbrace 1,2\rbrace$, with $h'_i\in H$ and $k'_i\in K'$. Note that $\varphi(g_i)=h_ik'_i$. Now, $\varphi(g_1)\varphi(g_2)=h'_3k''_3$ with $h'_3=h_1k'_1h_2k'^{-1}_1\in H$ and $k''_3=k'_1k'_2\in K'$. But $k_1$ and $k'_1$ act by conjugation on $H$ in the same way (because $H$ is abelian), so $h_3=h'_3$. Then $k_3=k_1k_2=h'_1k'_1h'_2k'_2=(h'_1k'_1h'_2k'^{-1}_1)k'_1k'_2$, so $k'_3=k'_1k'_2=k''_3$. Hence $\varphi(g_1g_2)=\varphi(g_1)\varphi(g_2)$.\end{proof}

%But, as well-known (see e.g. \cite[Lemma~1]{grunewal_segal}), semidirect product decompositions are in bijective correspondence with the automorphisms from \ref{semi_dec_aut} of the form:$$C = C_G = \{ \begin{pmatrix}1 & \beta \\0 & 1 \end{pmatrix} :  \forall k, k' \in K, \beta(kk') = \beta(k) \beta(k')^{k} \}.$$

\smallskip \noindent
Hence, let $\varphi$ be the automorphism of $\widehat{T} \rtimes f(K_1)=\widehat{T} \rtimes K_2$ that is the identity on $\widehat{T}$ and that maps $f(K_1)$ to $K_2$. Replacing $f$ with $\varphi \circ f$, we can assume without loss of generality that $f(K_1) = K_2$.
%using \ref{Exercise1} we can assume that $f$ is the identity on $K = K_1 = K_2$.
Then define the following automorphism of $\widehat{G} = (\widehat{T}_A \oplus \widehat{T}_B) \rtimes_\gamma (K_1 \times K_2)$: $(\bar{t}_1, \bar{t}_2, k_1, k_2) \rightarrow (f^{-1}(\bar{t}_2), f(\bar{t}_1), f^{-1}(k_2), f(k_1))$.

\smallskip \noindent
This means that there is an automorphism of $\widehat{G}$ which shuffles the elements of $\widehat{T}_A$ within themselves and shuffles the elements of $\widehat{T}_B$ within themselves in such a way that $K_2$ acts on $\widehat{T}_A$ as $K_1$ does and $K_1$ acts on $\widehat{T}_B$ as $K_2$ does.
Hence,  since $G$ is an elementary subgroup of $\widehat{G}$ (by \cite{Oger88}), for every $k_1 \in K_1$ and $k_2 \in K_2$ we have:
$$\mathrm{tp}^G(k_1, e) = \mathrm{tp}^G(e, f(k_1))$$
$$\mathrm{tp}^G(e, k_2) = \mathrm{tp}^G(f^{-1}(k_2), e).$$
But, as proved below, there cannot be an automorphism $\pi$ of $G$ which reflects the identity of types above, as this would induce an isomorphism of $T \rtimes_\alpha K$ onto $ T \rtimes_\beta K$, contrary to our standing assumption that $T \rtimes_\alpha K$ and $ T \rtimes_\beta K$ are not isomorphic. Notice in fact that $(e, k_2)$ acts trivially on $T_A$ and similarly $(k_1, e)$ acts trivially on $T_B$, since by definition $G = A \times B = (T_A \rtimes_\alpha K) \times (T_B \rtimes_\beta K)$. In detail, let $u \in T_A$ be such that $u \neq 0_{T_A} = \bar{0}_A$. We look at where $\pi$ can map~$u$. 
\newline \underline{Case 1}. $\pi(u) = u_A$ with $u_A \in T_A$ (so necessarily $u_A \neq 0_{T_A}$).
\newline As $A$ is irreducible (i.e., the action $\alpha$ is irreducible), there exists $k_1 \in K_1$ such that $K_1 \models u^{k_1} \neq u$, but then in $G$ we have that $u^{(k_1, e)} \neq u$. But then we have that $\pi(u^{(k_1, e)}) = \pi(u)^{\pi(k_1, e)} = u_A^{(e, f(k_1))} = u_A = \pi(u)$, and this is a contradiction as $K_1 \models u^{k_1} \neq u$ and $\pi$ is an automorphism \mbox{(to see that $u_A^{(e, k_1)} = u_A$ recall that $G = A \times B$ and $u_A \in A$).}
\newline \underline{Case 2}. $\pi(u) = u_Au_B$ with $u_A \in T_A$, $u_B \in T_B$ and $u_A \neq 0_{T_A}$ and $u_B \neq 0_{T_B}$.
\newline As $A$ is irreducible (i.e., the action $\alpha$ is irreducible) there is $k_1 \in K_1$ such that $u_A^{k_1}=k_1u_A{k^{-1}_1} \neq u_A$. Let $k_2 = f(k_1)$. Observe now that $u^{(e, k_2)} = u$ as $u \in T_A$ and $G = A \times B$. But then we reach a contradiction as follows:
\[ \begin{array}{rcl}	
u^{(e, k_2)} = u
& \Leftrightarrow & \pi(u)^{\pi(e, k_2)} = \pi(u)\\
& \Leftrightarrow & (u_Au_B)^{(f^{-1}(k_2), e)} = u_Au_B\\
& \Leftrightarrow & u^{f^{-1}(k_2)}_A u_B = u_A u_B\\
& \Leftrightarrow & u^{f^{-1}(k_2)}_A = u_A \\
& \Leftrightarrow & u^{f^{-1}(f(k_1))}_A = u_A \\
& \Leftrightarrow & u_A^{k_1} = u_A.
\end{array} \]
\newline \underline{Case 3}. $\pi(u) = u_B$ with $u_B \in T_B$.
\newline This is the only case possible, as Case 1 and Case 2 are impossible, so $\pi(T_A)$ is contained in $T_B$. But the situation is symmetric in $A$ and $B$ (recall that we assume that also $B$ is irreducible, i.e., also $\beta$ is irreducible), so $\pi(T_B)$ is contained in $T_A$, and so necessarily $T_A$ is mapped onto $T_B$ and $T_B$ is mapped onto $T_A$ and so we are done, i.e., the automorphism $\pi$ actually induces an isomorphism of $A$ onto $B$, which is impossible, and so $\pi$ cannot exist.

%In more details, suppose that there is such an automorphism $\pi$ of $A \times B$ sending the first tuple to the second, then we define an isomorphism $g : A \cong B$ as follows: for any $(\bar{a}, k) \in A$ we define $g(\bar{a}, k) = (\bar{b}, k')$, where $(\bar{b}, k')$ is the unique element in $B$ such that:
%$\pi(\bar{a}, \bar{0}_B, k, e) = (\bar{0}_A, \bar{b},  e, k')$ for some $k' \in K$.

%the such that $id$ sends the first copy of $\mathbb{Z}^{22}$

\subsubsection{Second counterexample}\label{another_count_sec}

We modify the counterexample from Subsection~\ref{count_sec} and so we rely on the notation from there, in particular $K_1 = K = K_2$ are as there, as well as $A$, $B$, $T_A$ and $T_B$. Our aim is to show that if the split crystallographic group is not irreducible, then tuples from the translation subgroup need not be type-determined. 
	
\smallskip \noindent

	 As in Subsection~\ref{count_sec}, fix an isomorphism:
$$f: \widehat{T}_B \rtimes_\alpha K_1 \cong \widehat{T}_A \rtimes_\beta K_2$$
and assume without loss of generality that $f(K_1) = K_2$ (see Lemma \ref{WLOG} and Subsection~\ref{count_sec} for details).

 \smallskip \noindent
 Let $T^*_A$ be a free abelian group isomorphic to $T_A$ and $A' = T^*_A \oplus T_A \rtimes_{\alpha'} K_1$ be such that $K_1$ acts on $T_A$ as $\alpha$ and $K_1$ acts on the standard basis $b^*_A$ of $T^*_A$ in an irreducible way $\zeta$. Let $T^*_A = T^*_B$ and $B' = T^*_B \oplus T_B \rtimes_{\beta'} K_2$ with $K_2$ acting on $T_B$ as $\beta$ and $K_2$ acting on the standard basis $b^*_A = b^*_B$ as follows: for every $k \in K_2$ we have that $k$ acts on $b^*_B$ as $f^{-1}(k)$ acts on $b^*_A$. Notice that we then have that, for every $k \in K_1$, $f(k)$ acts on $T^*_B$ as $f^{-1}(f(k)) = k$ acts on $T^*_A$.

 \smallskip \noindent
Consider then the split crystallographic group $A' \times B'$. Now, in a similar fashion as in Section~\ref{count_sec} passing to $\widehat{A' \times B'}$ we can find an automorphism of $\widehat{A' \times B'}$ which swaps $K_1$ and $K_2$, $b^*_A$ and $b^*_B$ and $\widehat{T}_A$ and $\widehat{T}_B$. In more detail, we define: 
\[(\bar{t}^*_A, \bar{t}_A, \bar{t}^*_B, \bar{t}_B, k_A, k_B) \rightarrow (\bar{t}^*_B, f^{-1}(\bar{t}_B), \bar{t}^*_A, f(\bar{t}_A), f^{-1}(k_B), f(k_A)),\]
where this makes sense as we are requiring that $T^*_A = T^*_B$ and crucially this works because (see above): for every $k \in K_1$, $f(k)$ acts on $T^*_B$ as $k$ acts on $T^*_A$.

\smallskip \noindent
Then clearly $b^*_A$ and $b^*_B$ have the same type in $A' \times B'$ (as our automorphism swaps them), but there cannot exist an automorphism of $A' \times B'$ that swaps $b^*_A$ and $b^*_B$ as this would induce an automorphism of $K_1 \times K_2$ which swaps $K_1$ and $K_2$ (notice for example that $K_2$ acts trivially on $T^*_A$), which in turn would lead to a contradiction as in Section~\ref{count_sec}. Thus, $b^*_A, b^*_B$ are tuples from the translation subgroup of $A' \times B'$ which have the same type in $A' \times B'$ but are not automorphic in $A' \times B'$, as desired.

%\subsection{First-order and profinite rigidity}

\subsection{Homogeneity and profinite homogeneity in finitely generated abelian-by-finite groups}

In this section, we give a characterization of homogeneity in finitely generated abelian-by-finite groups which we believe to be of independent interest (note in particular that this characterization applies to all crystallographic groups).

%\begin{notation} Given $m < \omega$ we denote by $\mathbb{Z}_m$ the group $\mathbb{Z}/m\mathbb{Z}$.\end{notation}

%\begin{fact}\label{finite_quotient_fact} Let $1 \leq n < \omega$ and $f \in \mathrm{Aut}(\mathbb{Z}^n)$. Let $M$ be the matrix corresponding to $f$ with respect to the canonical basis of $\mathbb{Z}^n$. Then for every  $1 \leq k < \omega$ we have that $M$ induces an automorphism of $\mathbb{Z}^n_{k}$ via the matrix $\check{M}$ considered with entries taken modulo ${k}$ with respect to the canonical basis of $\mathbb{Z}^n_{k}$. Thus, every homomorphism $a: W_0 \rightarrow \mathrm{Aut}(\mathbb{Z}^n)$ induces a homomorphism $W_0 \rightarrow \mathrm{Aut}(\mathbb{Z}^n_{k})$, still denoted by $a$.\end{fact}

%\begin{lemme}\label{the_iso_lemma} In the context of \ref{finite_quotient_fact}, the following holds:$$(\mathbb{Z}^n \rtimes_a W_0)/k\mathbb{Z}^n \cong \mathbb{Z}^n/{k} \mathbb{Z}^n \rtimes_{a} W_0$$via the map:$$(\bar{z}, w)/k\mathbb{Z}^n \mapsto (\bar{z}/k\mathbb{Z}^n, w),$$where of course we identify $k\mathbb{Z}^n$ with $(k\mathbb{Z}^n, e)$.\end{lemme}

\begin{fact}[{\cite[Prop~0.1]{Oger1984}}]\label{oger_fact}
Let $G$ be a polycyclic-by-finite group and $n \geq 1$ an integer. There is an integer $k(n) \geq 1$ such that $G^n$ is defined in $G$ by the formula:
\[(\exists x_1 \cdots \exists x_{k(n)})(x = x_1^n \cdots x_{k(n)}^n).\]
%If $S$ is elementarily equivalent to $G$, $S^n$ is defined in $S$ by the same formula and $S/S^n$ is isomorphic to the finite group $G/G^n$.
\end{fact}

\begin{fact}[{\cite{grune_annals_1980}}]\label{grune_fact} Let $G$ be a polycyclic-by-finite group. For every $k < \omega$ we have that $G/G^k$ is finite. Furthermore, the profinite completion $\widehat{G}$ of $G$ is isomorphic to the inverse limit of $\{G/G^{k!} : 0 < k < \omega \}$.
\end{fact}

In the rest of this section, for a polycyclic-by-finite group $G$ and for every $k < \omega$, the finite quotient $G/G^{k!}$ will be denoted by $G_k$ and the quotient map $G \rightarrow G_k$ will be denoted by $\pi_{k!}$.
We recall that by the fundamental work \cite{segal_annals} every automorphism of a finitely generated profinite group is continuous, and so there is no ambiguity on which automorphisms we consider when we write $\mathrm{Aut}(\widehat{G})$, for $G$ finitely generated.

\begin{prop}\label{the_continous_auto_prop}Let $G$ be a finitely generated abelian-by-finite group. If, for every $k < \omega$, $\pi_{k!}(\bar{a})$ and $\pi_{k!}(\bar{b})$ are automorphic in $G_k := G/G^{k!}$, then $\bar{a}$ and $\bar{b}$ are automorphic in $\widehat{G}$.
\end{prop}

\begin{proof} First of all, recall that, by \ref{grune_fact}, we have that $\widehat{G}$ is isomorphic to the inverse limit of the inverse system $\{G/G^{k!} : 0 < k < \omega \}$. Recall also that $\pi_{k!}(G) = G_k = G/G^{k!}$. So suppose that there are $\bar{a}, \bar{b} \in G^\ell$ such that for every $k < \omega$ we have that $\pi_{k!}(\bar{a})$ and $\pi_{k!}(\bar{b})$ are automorphic in $G_k$.
%\begin{enumerate}[(A)]
%	 \item for every $k < \omega$ we have that $\pi_{k!}(\bar{a})$ and $\pi_{k!}(\bar{b})$ are automorphic in $\mathbb{Z}^2_{k!} \rtimes_a K$;
%Suppose that there are $\bar{a}, \bar{b} \in G^\ell$ s.t.:
%\begin{enumerate}[(A)]
%	 \item for every $k < \omega$ we have that $\pi_{k!}(\bar{a})$ and $\pi_{k!}(\bar{b})$ are automorphic in $\mathbb{Z}^2_{k!} \rtimes_a K$;
%	 \item $\bar{a}, \bar{b}$ are {\em not} automorphic in $G$.
%\end{enumerate}
%
%For ease of notation, for every $k < \omega$, let $G_k = \mathbb{Z}^m_{k!} \rtimes_a K$.
Observe now that if $k \leq n < \omega$, then every $f_n \in \mathrm{Aut}(G_n)$ such that $f_n(\pi_{n!}(\bar{a})) = \pi_{n!}(\bar{b})$ induces a $f_{(f_n, k)} \in \mathrm{Aut}(G_k)$ such that $f_{(f_n, k)}(\pi_{k!}(\bar{a})) = \pi_{k!}(\bar{b})$. Now, for every $n < \omega$, fix $f_n \in \mathrm{Aut}(G_n)$ such that $f_n(\pi_{n!}(\bar{a})) = \pi_{n!}(\bar{b})$ (notice that this is possible by our assumptions). Let $\mathcal{U}$ be a non-principal ultrafilter on $\omega$. Fix $k < \omega$ and let $f^1_k, ..., f^{m(k)}_k \in \mathrm{Aut}(G_k)$ be an injective enumeration of the automorphisms witnessing that $\pi_{k!}(\bar{a})$ and $\pi_{k!}(\bar{b})$ are automorphic in $G_k$ and notice that by our assumption we  have that $m(k) \geq 1$. For every $1 \leq i \leq m(k)$, let:
$$Y^i_k = \{n < \omega : k \leq n \text{ and } f_{(f_n, k)} = f^i_k\}.$$
Clearly, for $1 \leq i < j \leq m(k)$ we have that $Y^i_k \cap Y^j_k = \emptyset$. Further,  $Y^1_k, ..., Y^{m(k)}_k = \omega \setminus \{0, ..., k-1\} \in \mathcal{U}$ (as $\mathcal{U}$ is non-principal), and so, $\mathcal{U}$ being an ultrafilter, we can find $f^*_k \in \mathrm{Aut}(G_k)$ such that:
$$Y_k = \{n < \omega : k \leq n \text{ and } f_{(f_n, k)} = f^*_k\} \in \mathcal{U}.$$
Notice now that for $k_1 \leq k_2 < \omega$ we have that $f_{(f^*_{k_2}, k_1)} = f^*_{k_1}$: indeed, since $Y_{k_1}, Y_{k_2} \in \mathcal{U}$ we have that $Y_{k_1} \cap Y_{k_2} \neq \emptyset$ and so we can find $n \in Y_{k_1} \cap Y_{k_1}$. But then necessarily $n \geq k_1, k_2$ and we have that: 
$$f_{(f^*_{n}, k_1)} = f^{*}_{k_1} \text{ and } f_{(f^*_{n}, k_2)} = f^{*}_{k_2}.$$
from which it follows that:
$$f_{(f^*_{k_2}, k_1)} = f^{*}_{k_1}.$$
Hence, we have that $\prod_{k < \omega} f^*_k$ is a (continuous) automorphism of $\widehat{G}$ that sends $\bar{a}$ to $\bar{b}$, modulo the obvious embedding of $G$ into $\widehat{G}$ (recall that $G$ is residually finite and we indeed have an embedding of $G$ into $\widehat{G}$), and so we are done.
%Hence, we have that $\bar{a}$ and $\bar{b}$ have the same type in $\hat{G}$ and so they also have the same type in $G$, as $G$ is an elementary substructure of $\hat{G}$ (cf. \cite{Oger88}). It follows that $G$ is not homogeneous, as desired.
\end{proof}

We are ready to prove Theorem \ref{prof_homogeneity}, which is recalled below.

\begin{te} A finitely generated abelian-by-finite group $G$ is homogeneous if and only if it is profinitely homogeneous (cf. Definition~\ref{def_prof_hom}), that is, every tuple of words from $G$ is profinitely rigid. More precisely, for $u \in G^n$, the tuple $u$ is type-determined if and only if it is profinitely rigid.\end{te}	

\begin{proof}
Let $G$ be a finitely generated abelian-by-finite group. Suppose that $G$ is profinitely homogeneous, and let us prove that it is homogeneous. Let $\bar{a},\bar{b}$ be two tuples of elements of $G$, and suppose that they have the same type in $G$. Recall that $G^n$ is definable in $G$ without parameters for every integer $n$ (by \ref{oger_fact}), therefore the images of $\bar{a},\bar{b}$ have the same type in $G_k$ for any $k < \omega$, and thus they are automorphic in $G_k$ for any $k < \omega$ (since $G_k$ is finite, by \ref{grune_fact}). By Proposition \ref{the_continous_auto_prop}, there is an automorphism of $\widehat{G}$ mapping $\bar{a}$ to $\bar{b}$, moreover $G$ is profinitely homogeneous by assumption, hence there is an automorphism of $G$ mapping $\bar{a}$ to $\bar{b}$. Conversely, suppose that $G$ is homogeneous, and let us prove that it is profinitely homogeneous. Let $\bar{a},\bar{b}$ be two tuples of elements of $G$, and suppose that there is an automorphism of $\widehat{G}$ mapping $\bar{a}$ to $\bar{b}$. It follows that $\bar{a},\bar{b}$ have the same type in $\widehat{G}$. But $G$ is an elementary substructure of $\widehat{G}$ by \cite{Oger88}, so $\bar{a},\bar{b}$ have the same type in $G$ and thus they are automorphic in $G$.\end{proof}

\section{Homogeneity in affine Coxeter groups}\label{homogeneity_Coxeter}

Recall that a Coxeter group is a group that admits a presentation of the form \[\langle s_1,\ldots,s_n \ \vert \ (s_is_j)^{m_{ij}}=1, \ \text{for all} \ i,j \rangle,\]where $m_{ii}=1$ and $m_{ij}\in \mathbb{N}^{\ast}\cup \lbrace \infty\rbrace$ for every $1\leq i,j\leq n$ (the relation $(s_is_j)^{\infty}=1$ means that $s_is_j$ has infinite order). Note that each generator $s_i$ has order two and that $m_{ij}=2$ if and only if $s_i$ and $s_j$ commute. The \emph{Coxeter graph (or diagram)} of $\langle s_1,\ldots,s_n \ \vert \ (s_is_j)^{m_{ij}}=1, \ \text{for all} \ i,j \rangle$ is the graph with $n$ vertices labelled with $s_1,\ldots,s_n$, such that there is no edge between two vertices if the corresponding generators $s_i,s_j$ commute, an edge without a label if $(s_is_j)^3=1$, and an edge labelled with $n\geq 4$ (possibly $\infty$) if $(s_is_j)^n=1$. A Coxeter group is said to be \emph{irreducible} if its defining Coxeter graph is connected, \emph{spherical} if it is finite and \emph{affine} if it is virtually abelian and infinite without a nontrivial normal finite subgroup. The irreducible spherical and irreducible affine Coxeter groups were classified by Coxeter \cite{Cox32,Cox34}; see also Witt \cite{Witt41}. Figure~\ref{figure_affine} gives a complete classification of irreducible affine Coxeter groups in terms of their Coxeter graphs. Note that a Coxeter group is affine if and only if all its irreducible components are affine. 

%The set $S=\lbrace s_1,\ldots,s_n\rbrace$ is called a \emph{Coxeter generating set}, and the pair $(W,S)$ is called a \emph{Coxeter system}. A subgroup $H\subseteq W$ is called \emph{$S$-special} if $H$ is generated by a subset of $S$.

%A Coxeter system is even if $m_{ij}$ is \emph{even} or $\infty$ for every $1\leq i,j\leq n$, in which case it can be proved that $(W,S)$ is the only Coxeter system for $W$, and therefore it makes sense to say that $W$ is an even Coxeter group. 

%An even Coxeter group is called a \emph{right-angled Coxeter group} if $s_is_j=s_js_i$ or $(s_is_j)^{\infty}=1$ for every $1\leq i,j\leq n$. 

%Hyperbolic Coxeter groups form an important subclass of Coxeter groups. Moussong proved in his PhD thesis that a Coxeter group is hyperbolic if and only if it does not contain $\mathbb{Z}^2$, if and only if there is no pair of disjoint subsets $T_1,T_2\subseteq T$ such that $\langle T_1\rangle$ and $\langle T_2\rangle$ commute and are infinite, and there is no subset $T\subseteq S$ such that $(\langle T\rangle,T)$ is an affine Coxeter system of rank $\geq 3$ (note that affine Coxeter systems have been completely classified, and that affine Coxeter groups are virtually abelian).

%A Coxeter system $(W,S)$ is said to be \emph{2-spherical} if $m_{ij}$ is finite for every $1\leq i,j\leq n$. Note that such a Coxeter group $W$ has property (FA) of Serre (and therefore $W$ is one-ended) because it is generated by a set $S$ composed of involutions such that, for every $s_1,s_2\in S$, $s_1s_2$ has finite order. 

Note that in the context of Coxeter groups, the two notions of irreducibility coincide: for every irreducible affine Coxeter group $G$, there exists a finite Coxeter group $G_0$ and an irreducible representation $\rho : G_0 \rightarrow  \mathrm{GL}_n(\mathbb{Z})\subset \mathrm{GL}_n(\mathbb{Q})$ such that $G=\mathbb{Z}^n\rtimes_{\rho} G_0$ (see \ref{def_irreducible} for the definition of an irreducible representation and \cite[Chapter 6, paragraph 2]{bourbaki} for a proof of this result). In particular, $G$ is a split crystallographic group. Moreover, the Coxeter graph of $G$ is obtained from the Coxeter graph of $G_0$ by adding another vertex and one or two additional edges, as shown in Figure \ref{figure_affine}. More precisely, the following holds:
\begin{enumerate}[(1)]
    \item\label{item1} if $G$ is not isomorphic to $\tilde{A}_n$, then the Coxeter graph of $G$ is obtained from the Coxeter graph of $G_0$ by adding one vertex and one edge (with no label or labelled with 4 if $G$ is isomorphic to $\tilde{C}_n$);
    \item\label{item2} if $G$ is isomorphic to $\tilde{A}_n$, then the Coxeter graph of $G$ is obtained from the Coxeter graph of $G_0$ by adding one vertex and two edges with no label.
\end{enumerate}

%\begin{figure}[h]\includegraphics[width=12cm]{Finite_coxeter.svg.png}\caption{The irreducible finite Coxeter groups and their Coxeter diagrams.}\label{figure_affine}\end{figure}

\begin{center}
\begin{figure}[h]
\includegraphics[width=14cm]{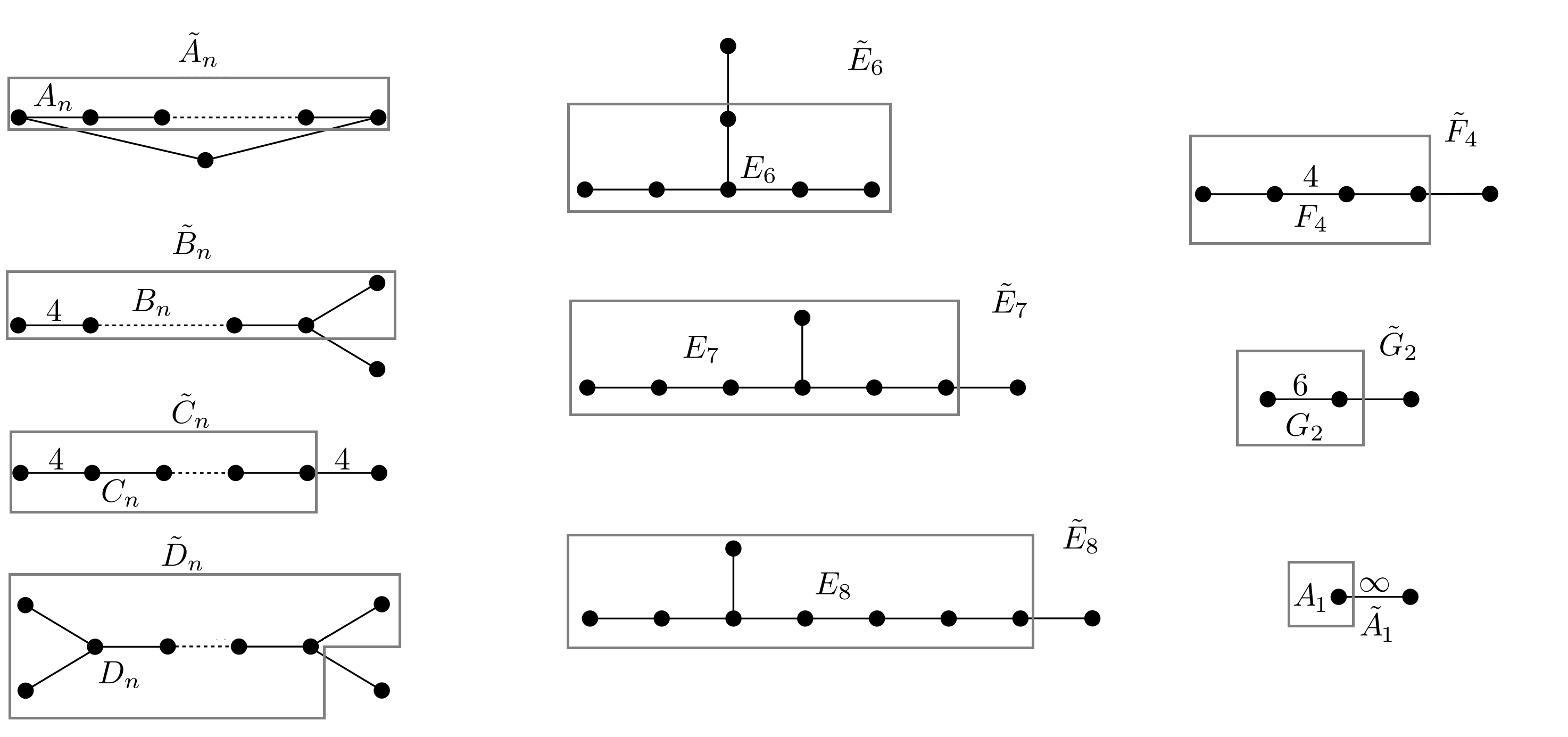}
\caption{The irreducible affine Coxeter groups and their Coxeter graphs. The grey boxes show the corresponding finite Coxeter groups. The groups $\tilde{A}_n,\tilde{C}_n$ on the left-hand side are defined for $n\geq 2$, the group $\tilde{B}_n$ is defined for $n\geq 3$ and the group $\tilde{D}_n$ is defined for $n\geq 4$.}
\label{figure_affine}
\end{figure}
\end{center}

We will prove that irreducible affine Coxeter groups are AE-homogeneous. The strategy of the proof is as follows (more details are given in Subsection \ref{details3}): let $(G,S)$ be an affine Coxeter group. For simplicity, let us assume that $G$ is irreducible (see Theorem \ref{main_affine_homogeneous} for the general case). Let $u,v$ be tuples of elements of $G$, and suppose that $u$ and $v$ have the same AE-type in $G$. If the subgroup of $G$ generated by the components of $u$ or $v$ is infinite, then we can conclude using Theorem \ref{main_crystallo0} (proved in the previous section) that there exists an automorphism $\sigma$ of $G$ such that $\sigma(u)=v$. So, let us suppose that the subgroups of $G$ generated by the components of $u$ and $v$ respectively are finite. There exists a maximal finite subgroup $G_u$ of $G$ containing the components of $u$. Using the assumption that $u$ and $v$ have the same AE-type, Lemma \ref{key_lemma_1} provides a class-permuting (cf. Definition~\ref{class-permuting}) endomorphism $\varphi$ of $G$ that maps $u$ to $v$. Define $G_v=\varphi(G_u)$. We can prove that $G_v$ is a maximal finite subgroup of $G$, and by Lemma \ref{lemma2} we can assume without loss of generality that $G_u=G_v$, except if $G$ is of type $\tilde{B}_n$ with $n$ even, in which case a refined argument is needed (see the second case in the proof of Theorem \ref{details}). The endomorphism $\varphi$ is not an automorphism of $G$ in general (see Example \ref{example2} below), but $\varphi_{\vert G_u}$ is an automorphism of $G_u$, so the question becomes: does there exist an automorphism $\sigma$ of $G$ such that $\varphi_{\vert U}=\sigma_{\vert U}$? We will answer this question positively in Subsection \ref{extension_section}, which will conclude the proof of the theorem (note that, by Theorem \ref{counter}, this strategy cannot work in the larger class of crystallographic groups).

%The proof is largely independent of the type-determinacy result proved in the previous subsection and relies on the following lemma.

\subsection{Preliminary results}

Let $G$ be a Coxeter group, and let $S$ be a Coxeter generating set for $G$. Recall that a \emph{standard parabolic subgroup} is a subgroup generated by a subset $I$ of $S$. This subgroup is denoted by $G_I$. We say that $G_I$ is \emph{maximal} if $\vert I\vert+1=\vert S\vert$. A subgroup of $G$ is called a \emph{parabolic subgroup} if it is conjugate to a standard parabolic subgroup. A \emph{reflection} is a conjugate of an element of $S$. We say that an endomorphism $\varphi$ of $G$ is \emph{reflection-preserving} if every reflection is mapped to a reflection. When $G$ is an irreducible affine Coxeter group, every maximal standard parabolic subgroup is finite (this can be verified by inspecting Coxeter diagrams of irreducible affine Coxeter groups). Moreover, these subgroups are maximal finite subgroups (by Corollary 1.30 in William Franzsen's thesis \cite{franz}), and every maximal finite subgroup of $G$ is conjugate to a maximal standard parabolic subgroup (by Lemma 1.23 in \cite{franz}).

\subsubsection{Reflection-preserving morphisms}

%For the references in William Franzsen's thesis, see also \cite{FH03}.

%\begin{de}\label{class-permuting}Let $G$ be a group and let $\varphi$ be an endomorphism of $G$. We say that $\varphi$ is \emph{class-permuting} if the following condition holds: for any two non-conjugate finite subgroups $H,H'$ of $G$, $\varphi(H)$ and $\varphi(H')$ are non-conjugate.\end{de}

Recall that an endomorphism $\varphi$ of a group $G$ is class-permuting if the following condition holds: for any two non-conjugate finite subgroups $H,H'$ of $G$, $\varphi(H)$ and $\varphi(H')$ are non-conjugate.

\begin{lemme}\label{lemma1}
Let $G$ be an affine Coxeter group of finite rank and let $\varphi$ denote an endomorphism of $G$. If $\varphi$ is class-permuting, then it is reflection-preserving.\end{lemme}

%\begin{rk}The terminology comes from the fact that, if $G$ has only finitely many conjugacy classes of finite subgroups, then such a morphism $\varphi$ induces a permutation of the conjugacy classes of finite subgroups.\end{rk}

%\begin{rk}A class-permuting morphism $\varphi$ will be obtained via a first-order sentence, using the fact that the property of being class-permuting is first-order expressible. Alternatively, we could use the fact that in $2$-spherical Coxeter groups $S^G$ is definable in $G$ without parameters by 9.1 in \cite{MUHLHERR2022297}.\end{rk}

\begin{proof}Write $G=G_1\times \cdots\times G_n$ where each $G_i$ is an irreducible affine Coxeter group. Let $S_i=\lbrace s_{1,i},\ldots,s_{k_i,i}\rbrace$ be a Coxeter generating set for $G_i$ and define $S=\cup_{1\leq i\leq n}S_i$. For each $1\leq i\leq n$ and $1\leq \ell\leq k_i$, define $H_{\ell,i}$ as the subgroup of $G_i$ generated by $S_{\ell,i}=S_i\setminus \lbrace s_{\ell,i}\rbrace$. By definition of the generating set $S_{\ell,i}$, $H_{\ell,i}$ is a maximal standard parabolic subgroup of $G_i$. Moreover, $H_{\ell,i}$ is finite, so $H_{\ell,i}$ is a maximal finite standard parabolic subgroup of $G_i$. Therefore, for any tuple $\bar{\ell}=(\ell_1,\ldots,\ell_n)$ with $1\leq \ell_i\leq k_i$, the group $H_{\bar{\ell}}= H_{\ell_1,1}\times \cdots\times H_{\ell_n,n}$ is a maximal finite standard parabolic subgroup of $G$. By Corollary 1.30 in \cite{franz}, since $G$ is infinite, each $H_{\bar{\ell}}$ is a maximal finite subgroup of $G$.

\smallskip \noindent Fix $1\leq i\leq n$ and $1\leq j\leq k_i$, and define $s=s_{j,i}\in S_i$ and $E=\lbrace \bar{\ell}=(\ell_1,\ldots,\ell_n) \ \vert \ \ell_i\neq j\rbrace$. Note that we can write \[\langle s\rangle=\bigcap_{\bar{\ell}\in E}H_{\bar{\ell}} \ \text{ and so } \ \varphi(\langle s\rangle)=\bigcap_{\bar{\ell}\in E}\varphi(H_{\bar{\ell}}).\]

\smallskip \noindent We claim that $\varphi(H_{\bar{\ell}})$ is a maximal finite subgroup of $G$: indeed, if $\varphi(H_{\bar{\ell}})$ is contained in a finite subgroup $H$ of $G$, then $\varphi^n(H_{\bar{\ell}})$ is contained in $\varphi^{n-1}(H)$ for every integer $n\geq 1$. But $\varphi$ induces a permutation of the conjugacy classes of the finite subgroups of $G$, so after choosing an appropriate $n$ we can assume that $\varphi^n(H_{\bar{\ell}})=gH_{\bar{\ell}}g^{-1}$ for some $g\in G$, thus $H_{\bar{\ell}}$ is contained in $g^{-1}\varphi^{n-1}(H)g$, which is a finite subgroup. By maximality of $H_{\bar{\ell}}$, we obtain $H_{\bar{\ell}}=g^{-1}\varphi^{n-1}(H)g$, so $H_{\bar{\ell}}$ and $\varphi^{n-1}(H)$ have the same order. But $\varphi$ is injective on finite subgroups, so $H_{\bar{\ell}}$ and $H$ have the same order, so $\varphi(H_{\bar{\ell}})=H$, which proves that $\varphi(H_{\bar{\ell}})$ is a maximal finite subgroup. Moreover, since $\varphi(H_{\bar{\ell}})$ is finite, it is contained in a finite parabolic subgroup of $G$ (see, for instance, Lemma 1.23 in \cite{franz}), so by maximality $\varphi(H_{\bar{\ell}})$ is a parabolic subgroup. Then, by Corollary 1.27 in \cite{franz}, the intersection of a finite number of parabolic subgroups is a parabolic subgroup, so $\cap_{\bar{\ell}\in E}\varphi(H_{\bar{\ell}})$ is a parabolic subgroup, i.e., $\varphi(\langle s\rangle)$ is a parabolic subgroup. But the only parabolic subgroups of order 2 are the ones generated by reflections, so $\varphi(s)$ is a reflection. Hence, $\varphi$ is reflection-preserving.\end{proof}

\subsubsection{Maximal parabolic finite subgroups}\label{list}

We say that a group $G$ is \emph{indecomposable} if it cannot be written as a direct product of non-trivial subgroups, and that $G$ is \emph{decomposable} otherwise. A theorem due to Krull, Remak and Schmidt states that every non-trivial finite group can be written as a direct product of indecomposable groups in a unique way (up to isomorphism and the order of the factors). By Theorem 7.2 in \cite{Paris} and the remark following the theorem, a finite irreducible Coxeter group $G$ that is not dihedral is decomposable if and only if $G$ is of type $B_n$ for $n\geq 3$ and $n$ odd, in which case $B_n\simeq D_n\times \mathbb{Z}/2\mathbb{Z}$, or $G$ is of type $H_3$ or $E_7$. Moreover, note that there are no isomorphisms between the groups $B_n$ (for $n\geq 1$) and $D_m$ (for $m\geq 3$), where we use the following convention: $A_1=B_1\simeq \mathbb{Z}/2\mathbb{Z}$ and $D_3=A_3$. From these observations we deduce the following lemma.

\begin{lemme}\label{prelim_exception}The following isomorphisms hold, for $i,j,k,\ell \geq 3$.
\begin{enumerate}
    \item The groups $B_i\times B_j$ and $B_k\times B_{\ell}$ are isomorphic if and only if $\lbrace i,j\rbrace=\lbrace k,\ell\rbrace$.
    \item The groups $D_i\times D_j$ and $D_k\times D_{\ell}$ are isomorphic if and only if $\lbrace i,j\rbrace=\lbrace k,\ell\rbrace$.
    \item The groups $B_i\times D_j$ and $B_k\times D_{\ell}$ are isomorphic if and only if $(i,j)=(k,\ell)$ or $i,j$ are both odd and $(i,j)=(\ell,k)$.
\end{enumerate}
\end{lemme}

The table below contains a complete list of isomorphism classes of the maximal parabolic finite subgroups of the irreducible affine Coxeter groups (except $\tilde{I}_1$). We simply obtain the list by removing exactly one vertex from the Coxeter diagram of the corresponding affine Coxeter group. Recall our conventions $A_1=B_1\simeq \mathbb{Z}/2\mathbb{Z}$ and $D_3=A_3$.

\begin{table}[h]
    \centering
    \renewcommand{\arraystretch}{1.5}
    \begin{tabular}{|l|l|}
        \hline
        \textbf{Type} & \textbf{Maximal parabolic finite subgroups} \\ \hline
        $\tilde{A}_n$ & $A_n \simeq S_{n+1}$ \\ \hline
        $\tilde{B}_n$ & $D_n$, $B_{\ell} \times D_k$ ($\ell+k=n, \ell \geq 1, k \geq 3$), $B_{n-2}\times A_1\times A_1$ or $B_n$ \\ \hline
        $\tilde{C}_n$ & $B_n$ or $B_{\ell} \times B_k$ ($\ell+k=n, \ell,k \geq 1$) \\ \hline
        $\tilde{D}_n$ & $D_n$, $A_1 \times A_1 \times D_{n-2}$, or $D_{\ell} \times D_k$ ($\ell+k=n, \ell,k \geq 3$) \\ \hline
        $\tilde{E}_6$ & $E_6$, $A_1 \times A_5$, or $A_2 \times A_2 \times A_2$ \\ \hline
        $\tilde{E}_7$ & $E_7$, $A_7$, $A_1 \times D_6$, $A_2 \times A_5$, or $A_3 \times A_3 \times B_1$ \\ \hline
        $\tilde{E}_8$ & $E_8$, $A_1 \times E_7$, $A_2 \times E_6$, $A_3 \times D_5$, $A_4 \times A_4$, $A_1 \times A_2 \times A_5$, $A_1 \times A_7$, $D_8$, or $A_8$ \\ \hline
        $\tilde{F}_4$ & $F_4$, $A_1 \times B_3$, $A_2 \times A_2$, $A_1 \times A_3$, or $B_4$ \\ \hline
        $\tilde{G}_2$ & $G_2$, $A_1 \times A_1$, or $A_2$ \\ \hline
    \end{tabular}
    %\caption{Classification of subgroup structures.}
    \caption{Maximal parabolic finite subgroups}
    \label{tab:subgroup_classification}
\end{table}

\begin{comment}
\begin{enumerate}
    \item $\tilde{A}_n$: $A_n\simeq S_{n+1}$.
    \item $\tilde{B}_n$: $D_n$ or $B_{\ell}\times D_k$ with $\ell+k=n$ and $\ell\geq 1, k\geq 4$ or $B_{n-3}\times A_3$ or $B_n$.
    \item $\tilde{C}_n$: $B_n$ or $B_{\ell}\times B_k$ with $\ell+k=n$ and $\ell,k\geq 1$.
    \item $\tilde{D}_n$: $D_n$ or $A_1\times A_1\times D_{n-2}$ or $A_3\times D_{n-3}$ or $D_{\ell}\times D_k$ with $\ell+k=n$ and $\ell,k\geq 4$.
    \item $\tilde{E}_6$: $E_6$ or $A_1\times A_5$ or $A_2\times A_2\times A_2$. 
    \item $\tilde{E}_7$: $E_7$ or $A_7$ or $A_1\times D_6$ or $A_2\times A_5$ or $A_3\times A_3\times B_1$. 
    \item $\tilde{E}_8$: $E_8$ or $A_1\times E_7$ or $A_2\times E_6$ or $A_3\times D_5$ or $A_4\times A_4$ or $A_1\times A_2\times A_5$ or $A_1\times A_7$ or $D_8$ or $A_8$.
    \item $\tilde{F}_4$: $F_4$ or $A_1\times B_3$ or $A_2\times A_2$ or $A_1\times A_3$ or $B_4$.
    \item $\tilde{G}_2$: $G_2$ or $A_1\times A_1$ or $A_2$.
\end{enumerate}
\end{comment}

Let $(G,S)$ be a Coxeter group. Let $S^G$ denote the set $\lbrace gsg^{-1} \ \vert \ g\in G, s\in S\rbrace$. Recall that an endomorphism $\varphi$ of $G$ is said to be reflection-preserving if $\varphi(S^G)\subset S^G$. We denote by $\mathrm{Aut}_r(G)$ the group of reflection-preserving automorphisms of $G$. A \emph{graph automorphism} of $G$ is an automorphism of $G$ that is induced by an automorphism of its labelled Coxeter graph. By Proposition 1.44 in \cite{franz}, if $G$ is finite, then $\mathrm{Aut}_r(G)$ is generated by the inner automorphisms and the graph automorphisms. We will need the following lemma.

\begin{lemme}\label{lemma2}Let $(G,S)$ be an irreducible affine Coxeter group. Let $H,H'$ be two maximal finite subgroups of $G$. If $H$ and $H'$ are isomorphic, then there is an automorphism $\sigma$ of $G$ such that $\sigma(H)=H'$ except if $G$ is of type $\tilde{B}_n$ with $n$ even and $H,H'$ are conjugate to $G_I,G_{I'}$ where $I=S\setminus \lbrace s_i\rbrace$ and $I'=S\setminus \lbrace s_{n-i+2}\rbrace$ with $4\leq i\leq n-2$ and $i$ even, where the numbering is shown in Figure \ref{Blemme} below. 

\begin{figure}[h!]
  \centering
    \begin{tikzpicture}
\draw[black, very thick] (1,0) -- (4,0);
\draw[dashed, very thick] (4,0) -- (5,0);
\draw[black, very thick] (5,0) -- (6,0);
\draw[black, very thick] (6,0) -- (7,1);
\draw[black, very thick] (6,0) -- (7,-1);
\node[text=black] at (1,-0.5) {$s_{n+1}$};
\node[text=black] at (2,-0.5) {$s_n$};
\node[text=black] at (1.5,0.5) {$4$};
\node[text=black] at (3,-0.5) {$s_{n-1}$};
%\node[text=black] at (4,0.5) {$a$};
\node[text=black] at (6,0.5) {$s_3$};
\node[text=black] at (5,-0.5) {$s_4$};
\node[text=black] at (7,1.5) {$s_1$};
\node[text=black] at (7,-1.5) {$s_2$};
\fill[black] (1,0) circle (0.1cm);
\fill[black] (2,0) circle (0.1cm);
\fill[black] (3,0) circle (0.1cm);
\fill[black] (4,0) circle (0.1cm);
\fill[black] (5,0) circle (0.1cm);
\fill[black] (6,0) circle (0.1cm);
\fill[black] (7,1) circle (0.1cm);
\fill[black] (7,-1) circle (0.1cm);
\end{tikzpicture}
  \caption{Affine Coxeter group $\tilde{B}_n$, with $n\geq 3$ (when $n=3$, the edge connecting $s_3$ and $s_4$ is labelled with $4$).}
  \label{Blemme}
\end{figure}
\end{lemme}

\begin{proof}Let $H,H'\subset G$ be two maximal finite subgroups of $G$. There exist two subsets $I,I'\subset S$ with $\vert I\vert +1=\vert I'\vert +1=\vert S\vert$ and two elements $g,g'\in G$ such that $gHg^{-1}=G_I$ and $g'H'g'^{-1}=G_{I'}$. We will prove by exhaustive case analysis of the irreducible affine Coxeter groups that are not of type $\tilde{B}_n$ (see Figure \ref{figure_affine}) that there exists a graph automorphism $\theta$ of $G$ such that $\theta(G_I)=G_{I'}$, so we can take $\sigma=\mathrm{ad}(g'^{-1})\circ \theta \circ \mathrm{ad}(g)$.

\smallskip

\textbf{Type $\tilde{A}_n$} (for $n\geq 1$). The Coxeter graph is a cycle of $n+1$ nodes. Removing any single node yields the finite Coxeter graph $A_n$. Thus, all maximal parabolic subgroups are isomorphic to $A_n$. The graph automorphism group acts transitively on all vertices of the cycle, which means that there always exists a graph automorphism of $G$ mapping any $A_n$ to another.

\smallskip

\textbf{Type $\tilde{B}_n$} (for $n\geq 3$). The graph has a fork at one end (see Figure \ref{Blemme} above). Let $H_i$ be the maximal finite subgroup of $G$ obtained by removing $s_i$. Note that $H_1\simeq H_2\simeq B_n$, and these two subgroups are obviously mapped to each other by the diagram symmetry swapping $s_1$ and $s_2$. Then, $H_3\simeq B_{n-2}\times A_1\times A_1$. Removing $s_i$ for $4\leq i\leq n$ yields $H_i\simeq B_{n-i+1}\times D_{i-1}$. Last, we have $H_{n+1}\simeq D_n$. By Lemma \ref{prelim_exception}, the subgroups $H_i$ and $H_j$ with $i\neq j$ are isomorphic if and only if the following conditions hold: $4\leq i\leq n-2$ and $i$ and $n$ are even (equivalently, $i-1$ and $n-i+1$ are odd) and $j-1=n-i+1$. But note that there is no graph automorphism of $G$ that induces such an isomorphism between $H_i$ and $H_j$. 

\smallskip

\textbf{Type $\tilde{C}_n$} (for $n\geq 2$). Let us number the vertices linearly from $s_1$ to $s_{n+1}$. As above, let $H_i$ be the maximal finite subgroup of $G$ obtained by removing $s_i$. Note that $H_1\simeq H_{n+1}\simeq B_n$ and that the left-right reflection symmetry of the diagram induces an isomorphism between $H_1$ and $H_{n+1}$. Removing $s_i$ for $2\leq i\leq n$ yields $H_i\simeq B_{i-1} \times B_{n-i+1}$. By Lemma \ref{prelim_exception}, the subgroups $H_i$ and $H_j$ are isomorphic if and only if $j = i$ or $j-1 = n-i+1$. If $i=j$ there is nothing to do, and if $j-1 = n-i+1$, the left-right reflection symmetry of the diagram induces an isomorphism between $H_i$ and $H_j$.

\smallskip

\textbf{Type $\tilde{D}_n$} (for $n\geq 4$). We consider the numbering of the vertices given by the following figure.

\begin{figure}[h!]
  \centering
    \begin{tikzpicture}
\draw[black, very thick] (3,1) -- (4,0);
\draw[black, very thick] (3,-1) -- (4,0);
\draw[dashed, very thick] (4,0) -- (5,0);
\draw[black, very thick] (5,0) -- (6,0);
\draw[black, very thick] (6,0) -- (7,1);
\draw[black, very thick] (6,0) -- (7,-1);
\node[text=black] at (3,-1.5) {$s_{n+1}$};
\node[text=black] at (3,1.5) {$s_n$};
%\node[text=black] at (4,0.5) {$a$};
\node[text=black] at (6,0.5) {$s_3$};
\node[text=black] at (5,0.5) {$s_4$};
\node[text=black] at (7,1.5) {$s_1$};
\node[text=black] at (7,-1.5) {$s_2$};
\fill[black] (3,1) circle (0.1cm);
\fill[black] (3,-1) circle (0.1cm);
\fill[black] (4,0) circle (0.1cm);
\fill[black] (5,0) circle (0.1cm);
\fill[black] (6,0) circle (0.1cm);
\fill[black] (7,1) circle (0.1cm);
\fill[black] (7,-1) circle (0.1cm);
\end{tikzpicture}
  \caption{Affine Coxeter group $\tilde{D}_n$, with $n\geq 4$.}
  \label{Dlemme}
\end{figure}

We still denote by $H_i$ the maximal finite subgroup of $G$ obtained by removing $s_i$. Note that $H_1\simeq H_2\simeq H_n\simeq H_{n+1}\simeq D_n$ and that these four subgroups are in the same orbit under graph automorphisms. Removing $s_i$ for $4\leq i\leq n-2$ yields $H_i\simeq D_{i-1} \times D_{n-i+1}$. By Lemma \ref{prelim_exception}, the subgroups $H_i$ and $H_j$ are isomorphic if and only if $j = i$ or $j-1 = n-i+1$. If $j-1 = n-i+1$, as in the previous case, the left-right reflection symmetry of the diagram induces an isomorphism between $H_i$ and $H_j$. This symmetry also swaps $H_3$ and $H_{s-1}$.

%Removing an internal node $i$ yields $D_i \times D_{n-i}$. The only order matches are between $D_i \times D_{n-i}$ and $D_{n-i} \times D_i$. The diagram of $\tilde{D}_n$ always contains an automorphism that reverses the central path, thereby swapping node $i$ with $n-i$. Removing any of the four leaves yields $D_n$, and one easily sees that these four subgroups are conjugate under graph automorphisms.

\smallskip

\textbf{Remaining cases.} It remains to deal with the exceptional irreducible affine Coxeter groups.

\smallskip

If $G$ is of type $\tilde{E}_8$ or $\tilde{F}_4$ or $\tilde{G}_2$, the number of isomorphism classes of maximal finite subgroups coincides with the number of vertices of the Coxeter diagram of $G$ (see Table \ref{tab:subgroup_classification}), and thus there is nothing to do. 

\smallskip

If $G$ is of type $\tilde{E}_6$, let us consider the numbering of the vertices given by Figure \ref{E6lemme} below. We still denote by $H_i$ the maximal finite subgroup of $G$ obtained by removing $s_i$. Note that $H_1\not\simeq H_2$ and $H_1\simeq H_5\simeq H_7$ and $H_2\simeq H_4\simeq H_6$, and such isomorphisms are induced by the group of graph automorphisms (which is of order 6). Moreover, the group $H_3$ is not isomorphic to $H_i$ for any $i\neq 3$.

\begin{figure}[h!]
  \centering
    \begin{tikzpicture}
\draw[black, very thick] (1,0) -- (5,0);
\draw[black, very thick] (3,0) -- (3,2);
\node[text=black] at (1,-0.5) {$s_{1}$};
\node[text=black] at (2,-0.5) {$s_2$};
\node[text=black] at (3,-0.5) {$s_3$};
\node[text=black] at (4,-0.5) {$s_4$};
\node[text=black] at (5,-0.5) {$s_5$};
\node[text=black] at (2.5,1) {$s_6$};
\node[text=black] at (2.5,2) {$s_7$};
\fill[black] (1,0) circle (0.1cm);
\fill[black] (2,0) circle (0.1cm);
\fill[black] (3,0) circle (0.1cm);
\fill[black] (4,0) circle (0.1cm);
\fill[black] (5,0) circle (0.1cm);
\fill[black] (3,2) circle (0.1cm);
\fill[black] (3,1) circle (0.1cm);
\end{tikzpicture}
  \caption{Affine Coxeter group $\tilde{E}_6$.}
  \label{E6lemme}
\end{figure}

\smallskip

If $G$ is of type $\tilde{E}_7$, let us consider the numbering of the vertices given by Figure \ref{E7lemme} below. Note that $H_1\simeq H_7$ and $H_2\simeq H_6$ and $H_3\simeq H_5$, and such isomorphisms are induced by the left-right reflection symmetry of the diagram. There are no other isomorphisms between the $H_i$.
%group of graph automorphisms (which is of order 6). Moreover, the group $H_7$ is not isomorphic to $H_i$ for any $i\neq 7$.

%By enumerating the standard maximal finite parabolic subgroups for the exceptional groups as we did for the infinite families above, we see that equality $\vert G_I\vert=\vert G_{I'}\vert$ occurs only if $G_I$ and $G_{I'}$ are in the same orbit under the action of graph automorphisms, which concludes the proof of the lemma.

\begin{figure}[h!]
  \centering
  \begin{tikzpicture}
    % Draw the main horizontal line spanning from s_0 to s_7
    \draw[black, very thick] (-3,0) -- (3,0);
    
    % Draw the vertical line to the branch node s_2
    \draw[black, very thick] (0,0) -- (0,1);
    
    % Draw the nodes (7 horizontally, 1 vertically)
    \fill[black] (-3,0) circle (0.1cm);
    \fill[black] (-2,0) circle (0.1cm);
    \fill[black] (-1,0) circle (0.1cm);
    \fill[black] (0,0) circle (0.1cm);
    \fill[black] (1,0) circle (0.1cm);
    \fill[black] (2,0) circle (0.1cm);
    \fill[black] (3,0) circle (0.1cm);
    \fill[black] (0,1) circle (0.1cm);
    
    % Add the labels using standard Bourbaki numbering + s_0 for affine
    \node[text=black] at (-3,-0.5) {$s_1$}; % Affine node
    \node[text=black] at (-2,-0.5) {$s_2$};
    \node[text=black] at (-1,-0.5) {$s_3$};
    \node[text=black] at (0,-0.5) {$s_4$};  % Central junction node
    \node[text=black] at (1,-0.5) {$s_5$};
    \node[text=black] at (2,-0.5) {$s_6$};
    \node[text=black] at (3,-0.5) {$s_7$};
    \node[text=black] at (0.5,1) {$s_8$};   % Top branch node
  \end{tikzpicture}
  \caption{Affine Coxeter group $\tilde{E}_7$.}
  \label{E7lemme}
\end{figure}
\end{proof}

The following result will be important for handling the problematic cases that appear in the statement of the previous lemma. Recall that given a subgroup $H$ of a group $G$, we denote by $C_G(H)$ the centraliser of $H$ in $G$, and we denote by $Z(H)$ the center of $H$.

\begin{lemme}\label{keykey}
Let $(G,S)$ be the irreducible affine Coxeter group of type $\tilde{B}_n$ with $n$ even. Consider the numbering of the vertices given by Figure \ref{Blemme}. Let $4\leq i\leq n$ and $i$ even. Let $H_i$ denote the maximal finite subgroups of $G$ obtained by removing $s_i$. Then $C_G(Z(H_i))$ is isomorphic to $\tilde{B}_{i-1}\times B_{n-i+1}$. In particular, $C_G(Z(H_i))$ is virtually $\mathbb{Z}^{i-1}$.
\end{lemme}

%In particular, it is virtually $\mathbb{Z}^{i-1}$ and has a maximal finite subgroup isomorphic to $$

\begin{proof}
Note that $H_i\simeq B_{k}\times D_{\ell}$ with $k=n-i+1$ and $\ell=i-1$. Since $n$ and $i$ are even, $k$ is odd, therefore $B_{k}\simeq D_k\times \mathbb{Z}/2\mathbb{Z}$ (see for instance Theorem 7.2 in \cite{Paris} and the remark following the theorem). By \cite{Zhou}, the center of $B_k$ is of order 2 and the center of $D_{\ell}$ is trivial (because $\ell$ is odd). It follows that the center $Z(H_i)$ of $H_i$ is the direct factor of order 2 in the decomposition of $B_k$. In order to determine $C_G(Z(H_i))$, it will be useful to realise $G$ concretely as the semidirect product $G=V\rtimes B_n$ where $V=\lbrace v\in\mathbb{Z}^n \ \vert \ \sum v_i\in 2\mathbb{Z}\rbrace$ and $B_n=\langle s_2,\ldots,s_{n+1}\rangle$ acts on $\mathbb{Z}^n=\langle e_1,\ldots,e_n\rangle$ as follows: for $2\leq i\leq n$, $s_i$ swaps $e_{i-1}$ and $e_i$ and fixes the remaining vectors, and $s_{n+1}$ maps $e_n$ to $-e_n$ and fixes the remaining vectors (note that $B_n$ is the group of signed permutations of $\lbrace \pm e_1,\ldots,\pm e_n\rbrace$). Writing elements of $G$ as pairs $(v,g)\in V\times B_n$, the extra generator $s_1$ is $(e_1+e_2,g)$ where $g$ acts on $\mathbb{Z}^n$ by $g(e_i)=e_i$ for $i\geq 3$ and $g(e_i)=-e_i$ for $i\in\lbrace 1,2\rbrace$. Hence, $s_1(e_1,\ldots,e_n)=(1-e_2,1-e_1,e_3,\ldots,e_n)$. This description is consistent with the numbering adopted in Figure \ref{Blemme} (in particular, one may easily verify that $s_1$ is an involution that commutes with $s_i$ for $i\neq 3$, and that $s_1s_3$ is of order 3). As already explained, the center of $H_i$ is of order 2 and coincides with the center of the $B_k$ factor, which is generated by $s_{i+1},\ldots,s_{n+1}$. Note that this group $\langle s_{i+1},\ldots,s_{n+1}\rangle$ acts trivially on $\langle e_1,\ldots,e_{i-1}\rangle$ and preserves $\langle e_i,\ldots,e_{n}\rangle$. Its center is generated by the involution $z$ acting on $\mathbb{Z}^n$ by $z(e_1,\ldots,e_n)=(e_1,\ldots,e_{i-1},-e_i,\ldots,-e_n)$ (explicitly, $z=(s_{n+1})(s_ns_{n+1}s_n)\cdots ((s_{i+1}\cdots s_n)s_{n+1}(s_{i+1}\cdots s_n)^{-1})$). An element $v\in V$ is fixed by $z$ if and only if $v$ belongs to $\langle e_1,\ldots,e_{i-1} \rangle$. Using the pair notation, we have $(v,g)(0,z)=(v,gz)$ and $(0,z)(v,g)=(z(v),zg)$, so $(v,g)$ centralises $(0,z)$ if and only if $v$ is fixed by $z$ and $g$ belongs to $C_{B_n}(z)$. It follows that $C_G(z)$ is equal to the semidirect product $\langle e_1,\ldots,e_{i-1}\rangle \rtimes C_{B_n}(z)$. Then, note that $g\in B_n$ centralises $z$ if and only if $g$ preserves $\langle e_1,\ldots,e_{i-1} \rangle$ and $\langle e_i,\ldots,e_{n} \rangle$. But recall that $B_n$ is the group of signed permutations of $\lbrace \pm e_1,\ldots,\pm e_n\rbrace$, so $C_{B_n}(z)$ is the direct product of the groups of signed permutations of $\lbrace \pm e_1,\ldots,\pm e_{i-1}\rbrace$ (that is $\langle s_1,\ldots,s_i,cs_{n+1}c^{-1}\rangle$ with $c=s_{i}\cdots s_n$) and $\lbrace \pm e_{i},\ldots,\pm e_n\rbrace$ (that is $\langle s_{i+1},\ldots,s_{n+1}\rangle$). Therefore, $C_G(z)\simeq \mathbb{Z}^{i-1}\rtimes (B_{i-1}\times B_{n-i+1})\simeq (\mathbb{Z}^{i-1}\rtimes B_{i-1})\times B_{n-i+1}\simeq \tilde{B}_{i-1}\times B_{n-i+1}$.\end{proof}

\begin{lemme}\label{max_finite}
Two irreducible affine Coxeter groups $G$ and $G'$ are isomorphic if and only if they have the same maximal finite subgroups up to isomorphism.
\end{lemme}

\begin{proof}
Write $G=\mathbb{Z}^k\rtimes G_0$ with $G_0$ finite. Note that there is a morphism with torsion-free kernel from $G$ onto $G_0$, so every finite subgroup of $G$ embeds into $G_0$. Therefore, if $G$ and $G'$ are isomorphic then $G_0$ and $G'_0$ are isomorphic. The grey boxes in Figure \ref{figure_affine} show the $G_0$ subgroups. The only case where $G_0$ and $G'_0$ are isomorphic but $G$ and $G'$ are not isomorphic is when $G$ is of type $\tilde{B}_n$ and $G'$ is of type $\tilde{C}_n$ (or vice versa). But these groups do not have the same maximal finite subgroups, as shown by Table \ref{tab:subgroup_classification}.

Alternatively, the lemma follows from Table \ref{tab:subgroup_classification}, using the fact that every non-trivial finite group can be written as a direct product of indecomposable groups in a unique way (up to isomorphism and the order of the factors), and the fact that most irreducible finite Coxeter groups are indecomposable, the only exceptions being $B_n\simeq D_n\times \mathbb{Z}/2\mathbb{Z}$ for $n\geq 3$ and $n$ odd, $H_3$ and $E_7$ (note that $H_3$ does not appear in \ref{tab:subgroup_classification}, and that $E_7$ appears only once and is not responsible for any complication).\end{proof}

\subsection{Extending automorphisms of maximal finite subgroups}\label{extension_section}

The following result will be crucial in the proof of homogeneity of irreducible affine Coxeter groups.

\begin{te}\label{extension}
Let $(G,S)$ be an irreducible affine Coxeter group and let $H$ be a maximal finite subgroup of $G$. Let $\varphi$ be a class-permuting endomorphism of $G$ such that $\varphi(H)=H$. Then there exists an automorphism $\sigma$ of $G$ such that $\sigma_{\vert H}=\varphi_{\vert H}$.
\end{te}

\begin{ex}\label{example2}
It is worth noting that a class-permuting endomorphism $\varphi$ of an irreducible affine Coxeter group $G$ is not an automorphism of $G$ in general (even if we assume that $\varphi$ preserves the translation subgroup of $G$). This is clear, for instance, if $G$ is an infinite dihedral group. Let us give another more interesting example. The group $G=\tilde{A} _2$ admits the following presentation: $\langle s_1,s_2,s_3 \ \vert \ s_i^2=(s_is_{i+1})^3=1 \text{ for } i\in\mathbb{Z}/3\mathbb{Z} \rangle$ (it is the triangle group $\Delta(3,3,3)$). We denote by $C(x)$ the centralizer of an element $x$. One can check that there exist $g,h\in G$ such that $C(s_1)=\langle s_1\rangle\times \langle g\rangle$, $C(s_2)=\langle s_1\rangle\times \langle h\rangle$ and $C(s_3)=\langle s_1\rangle\times \langle h^{-1}g\rangle$ where $g=(s_3s_1s_2)^2$ and $h=(s_3s_2s_1)^2$ (so $g^{-1}h=(s_2s_3s_1)^2$). Define $\varphi : G\rightarrow G$ by $\varphi(s_1)=s_1$, $\varphi(s_2)=s_2$ and $\varphi(s_3)=gs_3g^{-1}$. This is a well-defined morphism because $\varphi(s_1s_3)=s_1gs_3g^{-1}=g(s_1s_3)g^{-1}$ (so this element has order 3) and $\varphi(s_2s_3)=s_2gs_3g^{-1}=hs_2h^{-1}gs_3g^{-1}=h(s_2s_3)h^{-1}$ (so this element has order 3 as well). Moreover, every finite subgroup of $G$ is conjugate to a subgroup of $\langle s_1,s_2\rangle$ or $\langle s_1,s_3\rangle$ or $\langle s_2,s_3\rangle$, and $\varphi$ coincides on these subgroups, respectively, with the identity, with $\mathrm{ad}(g)$ and with $\mathrm{ad}(h)$. But one can check that $s_3$ does not belong to the image of $\varphi$, hence $\varphi$ is not an automorphism of $G$. In fact, if one writes $G=\langle t_1,t_2\rangle\rtimes \langle s_1,s_2\rangle$ with $t_1=(s_1s_2)(s_2s_3)^2$ and $t_2=(s_2s_3)(s_3s_1)^2$ (note that $\langle t_1,t_2\rangle$ is the maximal abelian subgroup of $G$ (in other words, the translation subgroup of $G$), isomorphic to $\mathbb{Z}^2$), one can see that $\varphi$ coincides with $4\mathrm{id}$ on $\langle t_1,t_2\rangle$.\end{ex}

Before proving Theorem \ref{extension}, whose proof is based on a (rather tedious) case-by-case analysis, we need to establish a few preliminary results. The problem we face in the proof of the theorem is that some graph automorphisms of $H$ may not extend to $G$ (as we will see below), so we need to prove that these problematic graph automorphisms of $H$ cannot be induced by a class-permuting endomorphism $\varphi$ of $G$. 

\begin{lemme}\label{RP}
In the context of the statement of Theorem \ref{extension}, $\varphi$ and $\varphi_{\vert H}$ are reflection-preserving.
\end{lemme}

\begin{proof}
By Lemma \ref{lemma1}, $\varphi$ is reflection-preserving. Let us prove that $\varphi_{\vert H}$ is reflection-preserving as well. We can assume that $H=G_I$ with $I\subset S$ and $\vert I\vert +1=\vert S\vert$. Define $\alpha=\varphi_{\vert H}$. For $s\in I$, $\varphi(s)$ is a reflection, so $\varphi(s)=gs'g^{-1}$ for some $g\in G$ and $s'\in S$. Moreover, $\varphi(s)$ belongs to $G_I$ by assumption. Then, by Corollary 1.28 in \cite{franz}, we have $G_I\cap gG_{\lbrace s'\rbrace}g^{-1}=uG_{\lbrace s''\rbrace} u^{-1}$ with $s''\in I$ and $u\in G_I$, and thus $\varphi(s)$ belongs to $I^{G_I}$, which proves that $\alpha$ is reflection-preserving.    
\end{proof}

\begin{lemme}\label{order4}
Let $(G,S)$ be a Coxeter group. Suppose that there exist three reflections $s_1,s_2,s_3\in S$ such that $s_1s_2$ is of order 4, $s_2s_3$ is of order 3, and $s_1$ is isolated in the odd graph (in other words, no edge incident to $s_1$ (viewed as a vertex) in the Coxeter graph is labelled with an odd integer). Suppose, moreover, that $ss'$ is of order 2, 3 or 4 for every distinct $s,s'\in S$. Then there is no class-permuting endomorphism $\varphi$ of $G$ such that $\varphi(s_2)=s_1$.
\end{lemme}

\begin{comment}
\begin{figure}[h!]
\centering
\begin{tikzpicture}
\draw[black, very thick] (1,0) -- (3,0);
\draw[dashed, very thick] (3,0) -- (4,0);
\node[text=black] at (1,-0.5) {$s_1$};
\node[text=black] at (2,-0.5) {$s_2$};
\node[text=black] at (1.5,0.5) {$4$};
\node[text=black] at (3,-0.5) {$s_{3}$};
%\node[text=black] at (4,0.5) {$a$};
\fill[black] (1,0) circle (0.1cm);
\fill[black] (2,0) circle (0.1cm);
\fill[black] (3,0) circle (0.1cm);
\end{tikzpicture}
\end{figure}
\end{comment}

\begin{rk}
The lemma applies for instance to the irreducible affine Coxeter groups $\tilde{B}_n$ and $\tilde{C}_n$ for $n\geq 3$.  
\end{rk}

\begin{proof}
Suppose that there exists a class-permuting endomorphism $\varphi$ of $G$ such that $\varphi(s_2)=s_1$. By Lemma \ref{lemma1}, $\varphi$ is reflection-preserving. The subgroup $\langle s_2,s_3\rangle$, which is isomorphic to the dihedral group $\mathrm{Dih}_6$ of order $6$, is mapped to $\langle s_1, \varphi(s_3)\rangle$ by $\varphi$, and $\varphi(s_3)$ is a reflection (i.e., $\varphi(s_3)$ is conjugate to an element of $S$). It follows from Lemma 3.2 in \cite{CM13}, that $\langle s_1, \varphi(s_3)\rangle$ is contained in $g\langle s,s'\rangle g^{-1}$ for some $g\in G$ and $s,s'\in S$ (with $s\neq s'$). Note that $\langle s,s'\rangle$ can be of order 4, 6 or 8, but since $\langle s_1, \varphi(s_3)\rangle$ is of order 6 (because $\varphi$ is injective on $\langle s_2,s_3\rangle$), $\langle s,s'\rangle$ is necessarily of order 6 and we have $\langle s_1, \varphi(s_3)\rangle=g\langle s,s'\rangle g^{-1}$. The group $\langle s,s'\rangle$ contains three elements of order 2, namely $s,s'$ and $ss's$ (which is obviously conjugate to $s'$), and thus $s_1$ is conjugate to $s$ or $s'$. But, by Lemma 3.3.3 in \cite{davis}, two elements of $S$ are conjugate if and only if they are connected in the odd graph, so $s_1$ is not conjugate to $s_i$ for $i>1$, and thus $s=s_1$ or $s'=s_1$. But there is no $s_i\in S$ that is connected to $s_1$ by an edge with label $3$, so there is no $s_i\in S$ such that $\langle s_1,s_i\rangle\simeq \mathrm{Dih}_6$, a contradiction.   
\end{proof}

We are ready to prove Theorem \ref{extension}, which is the combination of Propositions \ref{tildeAn}, \ref{tildeBn}, \ref{tildeCn}, \ref{tildeDn} and \ref{tildesporadic}. Recall that the group of reflection-preserving automorphisms $\mathrm{Aut}_r(H)$ is generated by the inner automorphisms and the graph automorphisms (see \cite{franz}). 

%Suppose, moreover, that $\varphi$ maps any pair of non-conjugate finite subgroups to a pair of non-conjugate finite subgroups.

%In fact, it seems that we simply need to suppose that $\varphi$ is reflection-preserving and injective on finite subgroups (then, adapt the proof of the last theorem).

%We will prove that $\alpha$ extends to an automorphism of $G$. Recall that the group of reflection-preserving automorphisms $\mathrm{Aut}_r(H)$ is generated by inner automorphisms and graph automorphisms (Franzsen). The problem we face is that some graph automorphisms of $H$ may not extend to $G$ (as we will see below), so we need to prove that these problematic graph automorphisms of $H$ cannot be induced by $\varphi$. The proof is based on a case-by-case analysis: see Lemmas \ref{}, \ref{}, \ref{}, \ref{} respectively for $G=\tilde{A}_n,\tilde{B}_n, \tilde{C}_n, \tilde{D}_n$.

\subsubsection{Irreducible affine Coxeter groups of type $\tilde{A}_n$}

\begin{prop}\label{tildeAn}
Let $G=\tilde{A}_n$. Let $H$ be a maximal finite subgroup of $G$ and let $\varphi$ be a class-permuting endomorphism of $G$ such that $\varphi(H)=H$. Then there exists an automorphism $\sigma$ of $G$ such that $\sigma_{\vert H}=\varphi_{\vert H}$.
\end{prop}

\begin{proof}
Define $\alpha=\varphi_{\vert H}$. Observe that $H \simeq A_n\simeq S_{n+1}$. If $n\neq 5$, it is well known that every automorphism of $H$ is inner, so there exists $h\in H$ such that $\alpha=\mathrm{ad}(h)_{\vert H}$ and thus we can take $\sigma=\mathrm{ad}(h)$. Now, suppose that $n=5$. By Lemma \ref{RP}, $\alpha$ is reflection-preserving, so $\alpha$ is inner (because the outer automorphism of $S_6$ maps any transposition to a product of three pairwise distinct transpositions, and thus is not reflection-preserving) and we conclude as above.\end{proof}

\subsubsection{Irreducible affine Coxeter groups of type $\tilde{B}_n$ for $n\geq 3$} Our goal is to prove that Proposition \ref{tildeAn} remains true for $\tilde{B}_n$. We will need the following lemmas.

\begin{lemme}\label{D_n}
Consider the Coxeter group $D_n$ (for $n\geq 4$), with the following presentation (see Figure \ref{Dn}): $\langle s_1,\ldots,s_n \ \vert \ s_i^2= (s_is_{i+1})^3=(s_1s_3)^3=(s_1s_2)^2=1, \ \forall i\in \lbrace 2,\ldots, n-1\rbrace\rangle$. Let $g\in D_n$ be such that $\mathrm{ad}(g)$ preserves $\lbrace s_1,s_2,s_4\rbrace$. Then $\mathrm{ad}(g)$ fixes $s_4$, and fixes or swaps $s_1$ and $s_2$.\end{lemme}

\begin{rk}
For $n\geq 5$, there exists an element $g$ such that $gs_4g^{-1}=s_4$, $gs_1g^{-1}=s_2$ and $gs_2g^{-1}=s_1$. For $n=4$, such an element cannot exist since $s_1,s_2,s_4$ play symmetrical roles (but note that there exists an outer graph automorphism swapping $s_1$ and $s_2$ and fixing $s_4$).
\end{rk}

\begin{figure}[h!]
  \centering
    \begin{tikzpicture}
\draw[black, very thick] (3,0) -- (4,0);
\draw[dashed, very thick] (4,0) -- (5,0);
\draw[black, very thick] (5,0) -- (6,0);
\draw[black, very thick] (6,0) -- (7,1);
\draw[black, very thick] (6,0) -- (7,-1);
\node[text=black] at (3,0.5) {$s_n$};
%\node[text=black] at (4,0.5) {$a$};
\node[text=black] at (6,0.5) {$s_3$};
\node[text=black] at (5,0.5) {$s_4$};
\node[text=black] at (7,1.5) {$s_1$};
\node[text=black] at (7,-1.5) {$s_2$};
\fill[black] (3,0) circle (0.1cm);
\fill[black] (4,0) circle (0.1cm);
\fill[black] (5,0) circle (0.1cm);
\fill[black] (6,0) circle (0.1cm);
\fill[black] (7,1) circle (0.1cm);
\fill[black] (7,-1) circle (0.1cm);
\end{tikzpicture}
  \caption{Coxeter group $D_n$, with $n\geq 4$.}
  \label{Dn}
\end{figure}

\begin{comment}
W = CoxeterGroup(["D", 5], implementation="permutation");
s = W.simple_reflections();
#le graphe est 1-2-3-4
#                  |
#                  5
for w in W.list():
    if (w*s[2] == s[3]*w and w*s[3] == s[2]*w and w*s[4] == s[4]*w and w*s[5] == s[5]*w):
        print(w)
\end{comment}

\begin{proof}
The group $D_n$ can be realized as the group of bijections of $X=\{\pm e_1,\dots,\pm e_n\}$ of the form $(x_1,\ldots,x_n)\in X \mapsto \left(\varepsilon_1 x_{\sigma(1)},\ldots, \varepsilon_n x_{\sigma(n)}\right)$, with $\sigma\in S_n$ and $\varepsilon_i \in \{\pm1\}$, with the constraint $\varepsilon_1\cdots \varepsilon_n = 1$ (note that this provides a decomposition of $D_n$ as a semidirect product $(\mathbb{Z}/2\mathbb{Z})^{n-1}\rtimes S_n$ where $(\mathbb{Z}/2\mathbb{Z})^{n-1}$ consists of sign changes of coordinates with even parity and $S_n$ acts by permuting coordinates). The generators $s_1,\ldots,s_n$ are defined as follows: $s_i=(i-1  \ i)$ for $2\leq i\leq n$, and $s_1(x_1,\ldots,x_n)=(-x_2,-x_1,x_3,\ldots,x_n)$. For $2\leq i\leq n$, $s_i$ has exactly two orbits of cardinality 2 for its action on $X$, namely $\lbrace e_{i-1},e_i\rbrace$ and $\lbrace -e_{i-1},-e_i\rbrace$. Define $X_i=\lbrace \pm e_{i-1},\pm e_i\rbrace$. Note that $s_1$ also has exactly two orbits of cardinality $2$, namely $\lbrace e_{1},-e_2\rbrace$ and $\lbrace -e_1,e_2\rbrace$, and define $X_1=\lbrace \pm e_{1},\pm e_2\rbrace=X_2$. For each integer $1\leq i\leq n
$, $X_i$ is the only subset of $X$ of cardinality 4 that is fixed under the action of $s_i$ and that is the union of two orbits of cardinality 2. Now, let $g\in D_n$ be such that $\mathrm{ad}(g)$ preserves $\lbrace s_1,s_2,s_4\rbrace$. Then there exists $\sigma\in S(\lbrace 1,2,4\rbrace)$ such that $gs_ig^{-1}=s_{\sigma(i)}$ for every $i\in \lbrace 1,2,4\rbrace$, and it follows that $g(X_i)=X_{\sigma(i)}$. But let us observe that $\#(X_i\cap X_4)=0$ for every $i\in \lbrace 1,2\rbrace$, while $X_1=X_2$. Therefore, $\sigma\in \langle (1 \ 2)\rangle$.\end{proof}

\begin{lemme}\label{B or D}
Let $(G,S)$ be $\tilde{B}_n$ for $n\geq 5$ (see Figure \ref{B} for the Coxeter diagram and a numbering of the vertices used in the proof). Let $\varphi$ be a class-permuting endomorphism of $G$. Suppose that $\varphi(\lbrace s_1,s_2,s_4\rbrace)=\lbrace s_1,s_2,s_4\rbrace$. Then $\varphi$ fixes $s_4$ and fixes or swaps $s_1$ and $s_2$.\end{lemme}

\begin{comment}
\begin{figure}[h!]
  \centering
    \begin{tikzpicture}
\draw[dashed, very thick] (4,0) -- (5,0);
\draw[black, very thick] (5,0) -- (6,0);
\draw[black, very thick] (6,0) -- (7,1);
\draw[black, very thick] (6,0) -- (7,-1);
\node[text=black] at (6,0.5) {$s_3$};
\node[text=black] at (5,0.5) {$s_4$};
\node[text=black] at (7,1.5) {$s_1$};
\node[text=black] at (7,-1.5) {$s_2$};
\fill[black] (5,0) circle (0.1cm);
\fill[black] (6,0) circle (0.1cm);
\fill[black] (7,1) circle (0.1cm);
\fill[black] (7,-1) circle (0.1cm);
\end{tikzpicture}
\caption{}
\label{B or D}
\end{figure}
\end{comment}

\begin{figure}[h!]
  \centering
    \begin{tikzpicture}
\draw[black, very thick] (1,0) -- (4,0);
\draw[dashed, very thick] (4,0) -- (5,0);
\draw[black, very thick] (5,0) -- (6,0);
\draw[black, very thick] (6,0) -- (7,1);
\draw[black, very thick] (6,0) -- (7,-1);
\node[text=black] at (1,-0.5) {$s_{n+1}$};
\node[text=black] at (2,-0.5) {$s_n$};
\node[text=black] at (1.5,0.5) {$4$};
\node[text=black] at (3,-0.5) {$s_{n-1}$};
%\node[text=black] at (4,0.5) {$a$};
\node[text=black] at (6,0.5) {$s_3$};
\node[text=black] at (5,-0.5) {$s_4$};
\node[text=black] at (7,1.5) {$s_1$};
\node[text=black] at (7,-1.5) {$s_2$};
\fill[black] (1,0) circle (0.1cm);
\fill[black] (2,0) circle (0.1cm);
\fill[black] (3,0) circle (0.1cm);
\fill[black] (4,0) circle (0.1cm);
\fill[black] (5,0) circle (0.1cm);
\fill[black] (6,0) circle (0.1cm);
\fill[black] (7,1) circle (0.1cm);
\fill[black] (7,-1) circle (0.1cm);
\end{tikzpicture}
  \caption{Affine Coxeter group $\tilde{B}_n$, with $n\geq 3$ (when $n=3$, the edge connecting $s_3$ and $s_4$ is labelled with $4$).}
  \label{B}
\end{figure}

%\begin{proof}Let $I=S\setminus \lbrace s_3\rbrace$. Note that $G_I\simeq H\times \mathbb{Z}/2\mathbb{Z} \times \mathbb{Z}/2\mathbb{Z}$ where $H$ is generated by $\langle s_{n+1},\ldots,s_4\rangle$ and the two cyclic groups of order 2 are generated by $s_1$ and $s_2$. If $G=\tilde{B}_n$ then $H$ is of type $B_{n-2}$ and if $G=\tilde{D}_n$ then $H$ is of type $D_{n-2}$ if $n\geq 6$ or $A_3$ if $n=5$. The morphism $\varphi$ being class-permuting by assumption, it maps $G_I$ (which is a maximal finite subgroup of $G$) to a maximal finite subgroup of $G$. Therefore, by Lemma \ref{lemma2}, there exist an element $g\in G$ and a graph automorphism $\sigma$ of $G$ such that $\mathrm{ad}(g)\circ \sigma \circ \varphi(G_I)=G_I$. \end{proof}

\begin{proof}Let $I=S\setminus \lbrace s_{n+1}\rbrace$. Note that $G_I$ is isomorphic to $D_n$ (because $n\geq 4$). The morphism $\varphi$ being class-permuting, and $G_I$ being a maximal finite subgroup, the subgroup $\varphi(G_I)$ is a maximal finite subgroup of $G$. Therefore, by Lemma \ref{lemma2}, there exist an element $g\in G$ and a graph automorphism $\sigma$ of $G$ such that $\mathrm{ad}(g)\circ \sigma \circ \varphi(G_I)=G_I$. By Proposition 1.44 in \cite{franz}, after replacing $g$ by $g'g$ for some $g'\in G_I$ if necessary, we can assume that $\mathrm{ad}(g)\circ \sigma \circ \varphi$ is a graph automorphism of $G_I$. In particular, $\mathrm{ad}(g)\circ \sigma \circ \varphi$ fixes $s_4$, and it fixes or swaps $s_1$ and $s_2$ (note that this is where we need the assumption that $n\geq 5$: indeed, this fact would not be true for $n=4$, because the group of graph automorphisms of $D_4$ permutes $\lbrace s_1,s_2,s_4\rbrace$ in any possible way). Note also that any graph automorphism of $G$ fixes $s_4$, and fixes or swaps $s_1$ and $s_2$, and thus $\sigma$ fixes $s_4$, and fixes or swaps $s_1$ and $s_2$. Moreover, by assumption, $\varphi(\lbrace s_1,s_2,s_4\rbrace)=\lbrace s_1,s_2,s_4\rbrace$. It follows that $\mathrm{ad}(g)$ induces a permutation of $\lbrace s_1,s_2,s_4\rbrace$. Now, let us observe that there exists a retraction $r$ from $G$ onto its subgroup $\langle s_1,\ldots,s_n\rangle$ (which is isomorphic to $D_n$), which is given by $r(s_{n+1})=1$ and $r(s_i)=s_i$ for $1\leq i\leq n$. Note that $\mathrm{ad}(r(g))$ and $\mathrm{ad}(g)$ induce the same permutation of $\lbrace s_1,s_2,s_4\rbrace$. Then, by Lemma~\ref{D_n}, $\mathrm{ad}(g)$ fixes $s_4$ and fixes or swaps $s_1$ and $s_2$. But recall that $\mathrm{ad}(g)\circ \sigma \circ \varphi$ and $\sigma$ fix $s_4$, and fix or swap $s_1$ and $s_2$, therefore $\varphi$ fixes $s_4$, and fixes or swaps $s_1$ and $s_2$.
\end{proof}

We need to prove that the previous lemma remains valid for the group $\tilde{B}_4$.

\begin{lemme}\label{B4}
Let $G=\tilde{B}_4$. Let $\varphi$ be a class-permuting endomorphism of $G$. Suppose that $\varphi(\lbrace s_1,s_2,s_4\rbrace)=\lbrace s_1,s_2,s_4\rbrace$. Then $\varphi$ fixes $s_4$ and fixes or swaps $s_1$ and $s_2$.\end{lemme}

\begin{proof}Let $I=S\setminus \lbrace s_{3} \rbrace =\lbrace s_1,s_2,s_4,s_5\rbrace$. The subgroup $\varphi(G_I)$ is a maximal finite subgroup of $G$. Therefore, by Lemma \ref{lemma2}, there exist an element $g\in G$ and a graph automorphism $\sigma$ of $G$ such that $\mathrm{ad}(g)\circ \sigma \circ \varphi(G_I)=G_I$. By Proposition 1.44 in \cite{franz}, after replacing $g$ by $g'g$ for some $g'\in G_I$ if necessary, we can assume that $\mathrm{ad}(g)\circ \sigma \circ \varphi$ is a graph automorphism of $G_I$. By Lemma \ref{order4}, $\mathrm{ad}(g)\circ \sigma \circ \varphi$ fixes $s_4$. Moreover, it fixes or swaps $s_1$ and $s_2$. Note also that any graph automorphism of $G$ fixes $s_4$, and fixes or swaps $s_1$ and $s_2$, so $\sigma$ fixes $s_4$, and fixes or swaps $s_1$ and $s_2$. Moreover, by assumption, $\varphi(\lbrace s_1,s_2,s_4\rbrace)=\lbrace s_1,s_2,s_4\rbrace$. It follows that $\mathrm{ad}(g)$ induces a permutation of $\lbrace s_1,s_2,s_4\rbrace$. As in the proof of the previous lemma, by retracting $G$ onto $\langle s_1,\ldots,s_4\rangle\simeq D_4$, we conclude by means of Lemma \ref{D_n} that $\mathrm{ad}(g)$ fixes $s_4$ and fixes or swaps $s_1$ and $s_2$, and thus that $\varphi$ fixes $s_4$, and fixes or swaps $s_1$ and $s_2$.\end{proof}

We are now ready to prove that Proposition \ref{tildeAn} remains true for $\tilde{B}_n$.

\begin{prop}\label{tildeBn}
Let $G=\tilde{B}_n$ for $n\geq 3$. Let $H$ be a maximal finite subgroup of $G$ and let $\varphi$ be a class-permuting endomorphism of $G$ such that $\varphi(H)=H$. Then there exists an automorphism $\sigma$ of $G$ such that $\sigma_{\vert H}=\varphi_{\vert H}$.
\end{prop}

\begin{proof}
Consider the numbering of the elements of $S$ shown in Figure \ref{B}. We can assume that $H=G_I$ with $I= S\setminus \lbrace s_i\rbrace$ for some $1\leq i\leq n+1$. We will distinguish several cases. We denote by $\alpha$ the automorphism $\varphi_{\vert H}$ of $H$. 

\smallskip

\noindent \textbf{Case 1}. Suppose that $i\geq 6$ (so $n\geq 5$). Note that $G_I$ splits into an internal direct product $G_I=UV$ with $U\simeq B_{n+1-i}$ and $V\simeq D_{i-1}$ (with the convention that $B_0$ is the trivial group and $B_1$ is the cyclic group of order 2). By Lemma \ref{RP}, $\alpha$ is reflection-preserving, and thus it is inner-by-graph by \cite[Proposition 1.44]{franz}. Hence, there exists an element $g\in G$ such that $\mathrm{ad}(g)\circ \alpha$ is a graph automorphism of $G_I$. But note that $U$ and $V$ are not isomorphic, so $\mathrm{ad}(g)\circ \alpha_{\vert U}$ and $\mathrm{ad}(g)\circ \varphi_{\vert V}$ are graph automorphisms of $U$ and $V$ respectively. One can write $g=uv$ with $u\in U$ and $v\in V$, so $\mathrm{ad}(u)\circ \alpha_{\vert U}$ and $\mathrm{ad}(v)\circ \varphi_{\vert V}$ are graph automorphisms of $U$ and $V$.

\noindent If $n+1-i\geq 3$ then every graph automorphism of $U$ is trivial, and if $n+1-i=2$ then there is a unique non-trivial graph automorphism of $U$ that swaps the two reflections $s_{n+1}$ and $s_n$, but this graph automorphism cannot be induced by $\varphi$ according to Lemma \ref{order4}. Hence, in both cases, we have $\mathrm{ad}(u)\circ \alpha_{\vert U}=\mathrm{id}_U$. 

\noindent We have $V\simeq D_k$ with $k\geq 5$, so there is a unique non-trivial graph automorphism $\psi$ of $V$ that swaps $s_1$ and $s_2$, so $\mathrm{ad}(v)\circ \alpha_{\vert V}=\mathrm{id}_V$ or $\mathrm{ad}(v)\circ \alpha_{\vert V}=\psi$.

\noindent Now, note that the graph automorphism $\psi$ of $V$ extends to a graph automorphism of $G$, still denoted by $\psi$ (this automorphism swaps $s_1$ and $s_2$ and fixes $s_i$ for $i\geq 3$). Therefore, we can take $\sigma=\mathrm{ad}((uv)^{-1})$ or $\sigma=\mathrm{ad}((uv)^{-1})\circ \psi$, and we have $\sigma_{\vert H}=\varphi_{\vert H}$.

\smallskip
\noindent \textbf{Case 2}. Suppose that $i=5$ (so $n\geq 4$). We keep the same notation as in Case 1. The new difficulty here comes from the fact that the group of graph automorphisms of $V\simeq D_4$ is of order 6. But we know that $\mathrm{ad}(v)\circ \alpha_{\vert V}$ is a graph automorphism of $V$, so we have $\mathrm{ad}(v)\circ \varphi(\lbrace s_1,s_2,s_4\rbrace)=\lbrace s_1,s_2,s_4\rbrace$, and thus Lemmas \ref{B or D} (for $n\geq 5$) and \ref{B4} (for $n=4$) show that $\mathrm{ad}(v)\circ \varphi$ fixes $s_4$ and fixes or swaps $s_1$ and $s_2$. Hence, we still have $\mathrm{ad}(v)\circ \alpha_{\vert V}=\mathrm{id}_V$ or $\mathrm{ad}(v)\circ \alpha_{\vert V}=\psi$. Therefore, we conclude as in Case 1.

\smallskip
\noindent \textbf{Case 3}. Suppose that $i=4$. Then $G_I$ splits into an internal direct product $G_I=UV$ with $U\simeq B_{n-3}$ and $V\simeq A_3$. Every automorphism of $A_3$ is inner, and every reflection-preserving automorphism of $B_{n-3}$ is inner-by-graph, and thus inner if $n-3\geq 3$. For $n-3=2$, $B_{n-3}$ has a graph automorphism of order 2 but Lemma \ref{order4} shows that $\alpha_{\vert U}$ is still inner. Hence, in any case, we conclude that $\alpha$ is inner. 

\smallskip \noindent \textbf{Case 4}. Suppose that $i=3$. Then $G_I$ splits into an internal direct product $G_I=UVW$ with $U\simeq B_{n-2}$ and $V\simeq \mathbb{Z}/2\mathbb{Z}$ and $W\simeq \mathbb{Z}/2\mathbb{Z}$. If $n-2\geq 3$, we conclude as above that we can take $\sigma=\mathrm{ad}(g)$ or $\sigma=\mathrm{ad}(g)\circ \psi$ for some $g\in G$, where $\psi$ denotes the graph automorphism of $G$ that swaps $s_1$ and $s_2$. If $n-2=2$, the fact that $\varphi$ does not swap $s_5$ and $s_4$ is guaranteed by Lemma \ref{order4} and we reach the same conclusion. If $n-2=1$, $\alpha$ may \emph{a priori} permute the subgroups $U=\langle s_4\rangle$, $V=\langle s_1\rangle$ and $W=\langle s_2\rangle$ in any possible way, but $s_4$ is not conjugate to $s_1$ or $s_2$ (since $s_4$ is isolated in the odd graph), and $\varphi$ is reflection-preserving, so $\varphi$ must fix $s_4$ and thus $\alpha$ fixes $s_4$ and fixes or swaps $s_1$ and $s_2$, and we conclude as above.

\smallskip
\noindent \textbf{Case 5}. Suppose that $i=1$ or $i=2$. We have $G_I\simeq B_n$. Since $n\geq 4$ and $\alpha$ is reflection-preserving, there exists $g\in G_I$ such that $\alpha=\mathrm{ad}(g)_{\vert G_I}$, therefore we can take $\sigma=\mathrm{ad}(g)$.\end{proof}

\subsubsection{Irreducible affine Coxeter groups of type $\tilde{C}_n$}

We will need the following lemma, which is comparable to Lemma \ref{order4}.

\begin{lemme}\label{order4bis}
Let $G=\tilde{C}_2=\langle s_1,s_2,s_3 \ \vert \ s_i^2=(s_1s_2)^4=(s_2s_3)^4=(s_1s_3)^2=1\rangle$ (see Figure~\ref{C3} below). Then there is no class-permuting endomorphism $\varphi$ of $G$ such that $\varphi(s_2)=s_1$ and $\varphi(s_1)=s_2$.
\end{lemme}

%\gianluca{Better to say class-permuting also here maybe? For consistency.}

\begin{figure}[h!]
  \centering
    \begin{tikzpicture}
\draw[black, very thick] (2,0) -- (4,0);
\node[text=black] at (2,-0.5) {$s_1$};
\node[text=black] at (2.5,0.5) {$4$};
\node[text=black] at (3.5,0.5) {$4$};
\node[text=black] at (3,-0.5) {$s_2$};
\node[text=black] at (4,-0.5) {$s_3$};
\fill[black] (2,0) circle (0.1cm);
\fill[black] (3,0) circle (0.1cm);
\fill[black] (4,0) circle (0.1cm);
\end{tikzpicture}
  \caption{The affine Coxeter group $\tilde{C}_2$.}
  \label{C3}
\end{figure}

\begin{proof}
Suppose that there exists an endomorphism $\varphi$ of $G$ as in the statement of the theorem. By Lemma \ref{RP}, $\varphi$ is reflection-preserving, so $\varphi(s_2)$ and $\varphi(s_3)$ are reflections. It follows from Lemma 3.2 in \cite{CM13} that $\langle \varphi(s_2), \varphi(s_3)\rangle$ is contained in $g\langle s,s'\rangle g^{-1}$ for some $g\in G$ and $s,s'\in S$. Note that $\langle s,s'\rangle$ is of order at most 8, but since $\langle \varphi(s_2), \varphi(s_3)\rangle$ is of order 8 (since $\varphi$ is injective on $\langle s_2,s_3\rangle$), $\langle s,s'\rangle$ is necessarily of order 8 and thus we have $\langle \varphi(s_2), \varphi(s_3)\rangle=g\langle s,s'\rangle g^{-1}$. Note that the group $\langle s,s'\rangle$ contains five elements of order 2, namely $s,s',ss's,s'ss', (ss')^2$.  But $(ss')^2$ belongs to the center of $\langle s,s'\rangle$, and $\varphi(s_2),\varphi(s_3)$ do not commute (again because $\varphi$ is injective on $\langle s_2,s_3\rangle$), and thus $\varphi(s_2)$ and $\varphi(s_3)$ belong to $\lbrace s,s',ss's,s'ss'\rbrace$. Then, by Lemma 3.3.3 in \cite{davis}, two elements of $S$ are conjugate if and only if they are connected in the odd graph, so $s_i$ is not conjugate to $s_j$ for $i\neq j$. Hence, since $\varphi(s_2)=s_1$, we may assume that $s=s_1$ after exchanging $s$ and $s'$ if necessary. It follows that $s'=s_2$ since $\langle s,s'\rangle$ is of order 8. Now, as $\varphi$ maps any two non-conjugate elements of finite order to non-conjugate elements, we obtain that $\varphi(s_3)$ is conjugate to $s_2$. This is a contradiction since $\varphi(s_1)=s_2$ and $s_1,s_3$ are not conjugate.\end{proof}

\begin{prop}\label{tildeCn}
Let $G=\tilde{C}_n$ for $n\geq 2$. Let $H$ be a maximal finite subgroup of $G$ and let $\varphi$ be a class-permuting endomorphism of $G$ such that $\varphi(H)=H$. Then there exists an automorphism $\sigma$ of $G$ such that $\sigma_{\vert H}=\varphi_{\vert H}$.
\end{prop}

\begin{figure}[h!]
  \centering
    \begin{tikzpicture}
\draw[black, very thick] (1,0) -- (4,0);
\draw[dashed, very thick] (4,0) -- (5,0);
\draw[black, very thick] (5,0) -- (7,0);
\node[text=black] at (1,-0.5) {$s_{1}$};
\node[text=black] at (2,-0.5) {$s_2$};
\node[text=black] at (1.5,0.5) {$4$};
\node[text=black] at (6.5,0.5) {$4$};
\node[text=black] at (3,-0.5) {$s_{3}$};
\node[text=black] at (4,-0.5) {$s_{4}$};
%\node[text=black] at (4,0.5) {$a$};
\node[text=black] at (6,-0.5) {$s_n$};
\node[text=black] at (7,-0.5) {$s_{n+1}$};
\fill[black] (1,0) circle (0.1cm);
\fill[black] (2,0) circle (0.1cm);
\fill[black] (3,0) circle (0.1cm);
\fill[black] (4,0) circle (0.1cm);
\fill[black] (5,0) circle (0.1cm);
\fill[black] (6,0) circle (0.1cm);
\fill[black] (7,0) circle (0.1cm);
\end{tikzpicture}
  \caption{Affine Coxeter group $\tilde{C}_n$, with $n\geq 2$.}
  \label{C}
\end{figure}

\begin{proof}We can assume that $H=G_I$ with $I= S\setminus \lbrace s_i\rbrace$ for some $1\leq i\leq n+1$. Write $k=i-1$ and $\ell=n-k=n+1-i$, and note that $G_I$ splits into an internal product $G_I=UV$ with $U=\langle s_1\ldots,s_{i-1}\rangle \simeq B_{k}$ (except if $i=1$, in which case we take for $U$ the trivial group) and $V=\langle s_{i+1},\ldots,s_{n+1}\rangle\simeq B_{\ell}$ (except if $i=n+1$, in which case we take for $V$ the trivial group). By Lemma \ref{RP}, $\alpha$ is reflection-preserving, and thus it is inner-by-graph by \cite{franz}. If $k\neq \ell$ then no graph automorphism of $G_I$ permutes its direct factors $U$ and $V$, thus there exist $u\in U$ and $v\in V$ such that $\mathrm{ad}(u)\circ \varphi_{\vert U}$ and $\mathrm{ad}(v)\circ \varphi_{\vert V}$ are graph automorphisms of $U$ and $V$ respectively. Subcase 1: if $k\neq 2$ and $\ell\neq 2$ then $U$ and $V$ do not have non-trivial graph automorphisms, so we can take $\sigma=\mathrm{ad}((uv)^{-1})$. Subcase 2: if $k=2$ or $\ell=2$, then we can assume, for example, that $k=2$ and $\ell\neq 2$ (since we assume that $k\neq \ell$). Note that we have $U=\langle s_1,s_2 \ \vert \ s_i^2=(s_1s_2)^4=1\rangle$. If $s_2s_3$ if of order 3, or equivalently if $G$ has at least four vertices (i.e., $n\geq 3$) then Lemma \ref{order4} shows that $\varphi$ does not swap the two vertices $s_1$ and $s_2$ of $U$ and we can still take $\sigma=\mathrm{ad}((uv)^{-1})$. If $s_2s_3$ is of order 4, we use Lemma \ref{order4bis} to reach the same conclusion.

\smallskip \noindent Now, suppose that $k=\ell$, and let $\psi$ denote the graph automorphism of $G$ of order 2 (that fixes $s_i$ and swaps $U$ and $V$). If $\varphi$ preserves $U$ and $V$, then we conclude as above. If $\varphi$ swaps $U$ and $V$, then $\psi \circ \varphi$ preserves $U$ and $V$, hence there exist $u\in U$ and $v\in V$ such that $\mathrm{ad}(u)\circ \psi\circ \varphi_{\vert U}$ and $\mathrm{ad}(v)\circ \psi \circ \varphi_{\vert V}$ are graph automorphisms of $U$ and $V$ respectively, and we conclude that we can take $\sigma=\psi \circ \mathrm{ad}((uv)^{-1})$.\end{proof}

\subsubsection{Irreducible affine Coxeter groups of type $\tilde{D}_n$ for $n\geq 4$} 

\begin{lemme}\label{B or D 2}
Let $(G,S)=\tilde{D}_n$ for $n\geq 5$ (see Figure \ref{D} for the Coxeter diagram and a numbering of the vertices used in the proof). Let $\varphi$ be a class-permuting endomorphism of $G$. If $\varphi(\lbrace s_1,s_2,s_4\rbrace)=\lbrace s_1,s_2,s_4\rbrace$, then $\varphi$ fixes $s_4$ and fixes or swaps $s_1$ and $s_2$. Similarly, if $\varphi(\lbrace s_{n-2},s_n,s_{n+1}\rbrace)=\lbrace s_{n-2},s_n,s_{n+1}\rbrace$, then $\varphi$ fixes $s_{n-2}$ and fixes or swaps $s_n$ and $s_{n+1}$.
\end{lemme}

\begin{rk}
Note that for $n=6$ the vertices $s_{n-2}$ and $s_4$ coincide, but we do not need to assume that the sets $\lbrace s_1,s_2,s_4\rbrace$ and $\lbrace s_{n-2},s_n,s_{n+1}\rbrace$ have empty intersection.
\end{rk}

\begin{figure}[h!]
  \centering
    \begin{tikzpicture}
\draw[black, very thick] (3,1) -- (4,0);
\draw[black, very thick] (3,-1) -- (4,0);
\draw[dashed, very thick] (4,0) -- (5,0);
\draw[black, very thick] (5,0) -- (6,0);
\draw[black, very thick] (6,0) -- (7,1);
\draw[black, very thick] (6,0) -- (7,-1);
\node[text=black] at (3,-1.5) {$s_{n+1}$};
\node[text=black] at (3,1.5) {$s_n$};
%\node[text=black] at (4,0.5) {$a$};
\node[text=black] at (6,0.5) {$s_3$};
\node[text=black] at (5,0.5) {$s_4$};
\node[text=black] at (7,1.5) {$s_1$};
\node[text=black] at (7,-1.5) {$s_2$};
\fill[black] (3,1) circle (0.1cm);
\fill[black] (3,-1) circle (0.1cm);
\fill[black] (4,0) circle (0.1cm);
\fill[black] (5,0) circle (0.1cm);
\fill[black] (6,0) circle (0.1cm);
\fill[black] (7,1) circle (0.1cm);
\fill[black] (7,-1) circle (0.1cm);
\end{tikzpicture}
  \caption{Affine Coxeter group $\tilde{D}_n$, with $n\geq 4$.}
  \label{D}
\end{figure}

%\begin{proof}Let $I=S\setminus \lbrace s_3\rbrace$. Note that $G_I\simeq H\times \mathbb{Z}/2\mathbb{Z} \times \mathbb{Z}/2\mathbb{Z}$ where $H$ is generated by $\langle s_{n+1},\ldots,s_4\rangle$ and the two cyclic groups of order 2 are generated by $s_1$ and $s_2$. If $G=\tilde{B}_n$ then $H$ is of type $B_{n-2}$ and if $G=\tilde{D}_n$ then $H$ is of type $D_{n-2}$ if $n\geq 6$ or $A_3$ if $n=5$. The morphism $\varphi$ being class-permuting by assumption, it maps $G_I$ (which is a maximal finite subgroup of $G$) to a maximal finite subgroup of $G$. Therefore, by Lemma \ref{lemma2}, there exist an element $g\in G$ and a graph automorphism $\sigma$ of $G$ such that $\mathrm{ad}(g)\circ \sigma \circ \varphi(G_I)=G_I$. \end{proof}

\begin{proof}The proof of this lemma is very similar to that of Lemma \ref{B or D}, but there are a few differences due to the fact that the diagram of $\tilde{D}_n$ has more symmetries than that of $\tilde{B}_n$. Let $I=S\setminus \lbrace s_{n+1}\rbrace$. Note that $G_I$ is isomorphic to $D_n$. The morphism $\varphi$ being class-permuting, and $G_I$ being a maximal finite subgroup, the subgroup $\varphi(G_I)$ is a maximal finite subgroup of $G$. Therefore, by Lemma \ref{lemma2}, there exist an element $g\in G$ and a graph automorphism $\sigma$ of $G$ such that $\mathrm{ad}(g)\circ \sigma \circ \varphi(G_I)=G_I$. By Proposition 1.44 in \cite{franz}, after replacing $g$ by $g'g$ for some $g'\in G_I$ if necessary, we can assume that $\mathrm{ad}(g)\circ \sigma \circ \varphi$ is a graph automorphism of $G_I$. In particular, $\mathrm{ad}(g)\circ \sigma \circ \varphi$ fixes $s_4$, and it fixes or swaps $s_1$ and $s_2$ (for this we need $n\geq 5$; indeed, this fact would not be true for $n=4$, because the group of graph automorphisms of $D_4$ permutes $\lbrace s_1,s_2,s_4\rbrace$ in any possible way). There are two possibilities for the graph automorphism $\sigma$:
\begin{enumerate}
    \item $\sigma$ fixes $s_4$ and fixes or swaps $s_1$ and $s_2$;
    \item $\sigma$ swaps $s_4$ and $s_{n-2}$, and swaps $\lbrace s_1,s_2\rbrace$ and $\lbrace s_n,s_{n+1}\rbrace$.
\end{enumerate}
Now, let us observe that there exists a retraction $r$ from $G$ onto its subgroup $\langle s_1,\ldots,s_n\rangle$ (which is isomorphic to $D_n$), given by $r(s_{n+1})=s_n$ and $r(s_i)=s_i$ for $1\leq i\leq n$. In the first case, we conclude exactly as in the proof of Lemma \ref{B or D}. So, suppose towards a contradiction that the second case holds. By assumption, $\varphi(\lbrace s_1,s_2,s_4\rbrace)=\lbrace s_1,s_2,s_4\rbrace$. It follows that $\mathrm{ad}(g)(\lbrace s_{n-2},s_n,s_{n+1}\rbrace)=\lbrace s_1,s_2,s_4\rbrace$. Therefore, there exist $i,j\in\lbrace 1,2,4\rbrace$ such that $\mathrm{ad}(g)(s_i)=s_n$ and $\mathrm{ad}(g)(s_j)=s_{n+1}$, with $i\neq j$. Applying the retraction $r$ to these equalities, we obtain $\mathrm{ad}(r(g))(s_i)=s_n$ and $\mathrm{ad}(r(g))(s_j)=s_{n}$ and thus $s_i=s_j$, which is a contradiction.\end{proof}

\begin{prop}\label{tildeDn}
Let $G=\tilde{D}_n$ for $n\geq 4$. Let $H$ be a maximal finite subgroup of $G$ and let $\varphi$ be a class-permuting endomorphism of $G$ such that $\varphi(H)=H$. Then there exists an automorphism $\sigma$ of $G$ such that $\sigma_{\vert H}=\varphi_{\vert H}$.
\end{prop}

\begin{proof}
Consider the numbering of the elements of $S$ shown in Figure \ref{D}. The proof of the proposition is very similar to that of Proposition \ref{tildeBn}, but there are a few differences. We still denote by $\alpha$ the automorphism $\varphi_{\vert H}$ of $H$. By Lemma \ref{RP}, $\alpha$ is reflection-preserving, and thus it is inner-by-graph by \cite[Proposition 1.44]{franz}. We can assume that $H=G_I$ with $I= S\setminus \lbrace s_i\rbrace$ for some $1\leq i\leq n+1$. 

\smallskip

\noindent \textbf{Case A}. First, suppose that $n\geq 5$. 

\smallskip

\noindent \textbf{Case A.1}. Suppose that $5\leq i \leq n-3$ (so $n\geq 8$). Then $G_I=UV$ with $U=\langle s_1,\ldots,s_{i-1}\rangle\simeq D_{i-1}$ and $V=\langle s_{i+1},\ldots,s_{n+1}\rangle\simeq D_{n-i+1}$. 

\smallskip

\noindent \textbf{Case A.1.1}. If $n\neq 2(i-1)$ then $U$ and $V$ are not isomorphic, so there exist $u\in U$ and $v\in V$ such that $\mathrm{ad}(u)\circ \alpha_{\vert U}$ and $\mathrm{ad}(v)\circ \varphi_{\vert V}$ are graph automorphisms of $U$ and $V$ respectively (see the proof of Proposition \ref{tildeBn} for more details). When $i\notin \lbrace 5, n-3\rbrace$, every inner-by-graph automorphism of $U$ and $V$ clearly extends to an automorphism of the entire group and we conclude as in Case 1 in the proof of Proposition \ref{tildeBn}. Then, suppose that $i\in \lbrace 5, n-3\rbrace$, for instance $i=5$; in this case (which is comparable to Case 2 in the proof of Proposition \ref{tildeBn}), $\mathrm{ad}(uv)\circ \varphi$ preserves the sets $\lbrace s_1,s_2,s_4\rbrace$ and $\lbrace s_{n-2},s_n,s_{n+1}\rbrace$, so Lemma \ref{B or D 2} shows that $\mathrm{ad}(uv)\circ \varphi$ fixes $s_4$ and fixes or swaps $s_1$ and $s_2$, and fixes $s_{n-2}$ and fixes or swaps $s_n$ and $s_{n+1}$, and we can still conclude that $\mathrm{ad}(uv)\circ \alpha$ extends to an automorphism of $G$.

\smallskip

\noindent \textbf{Case A.1.2}. If $n=2(i-1)$ then $U$ and $V$ are isomorphic. After composing $\varphi$ with the unique graph automorphism swapping $s_1$ and $s_n$ if necessary, we can assume that there exist $u\in U$ and $v\in V$ such that $\mathrm{ad}(u)\circ \alpha_{\vert U}$ and $\mathrm{ad}(v)\circ \varphi_{\vert V}$ are graph automorphisms of $U$ and $V$ respectively, and we proceed as in Case A.1.1.

\smallskip

\noindent \textbf{Case A.2}. Suppose that $i<5$ or $i > n-3$, for instance $i<5$. The proof works as in Cases 3, 4, 5 in the proof of Proposition \ref{tildeBn}.

\smallskip

\noindent \textbf{Case B}. Then, suppose that $n=4$. In this case, it is clear that every graph automorphism of $G_I$ extends to an automorphism of $G$.\end{proof}

\subsubsection{Remaining cases}

\begin{prop}\label{tildesporadic}
Let $G$ be an affine Coxeter group of type $\tilde{E}_6,\tilde{E}_7,\tilde{E}_8, \tilde{F}_4,\tilde{G}_2,\tilde{I}_1\simeq D_{\infty}$. Let $H$ be a maximal finite subgroup of $G$ and let $\varphi$ be a class-permuting automorphism of $G$ such that $\varphi(H)=H$. Then there exists an automorphism $\sigma$ of $G$ such that $\sigma_{\vert H}=\varphi_{\vert H}$.
\end{prop}

\begin{figure}[h!]
  \centering
    \begin{tikzpicture}
\draw[black, very thick] (1,0) -- (5,0);
\draw[black, very thick] (3,0) -- (3,2);
\node[text=black] at (1,-0.5) {$s_{1}$};
\node[text=black] at (2,-0.5) {$s_2$};
\node[text=black] at (3,-0.5) {$s_3$};
\node[text=black] at (4,-0.5) {$s_4$};
\node[text=black] at (5,-0.5) {$s_5$};
\node[text=black] at (2.5,1) {$s_6$};
\node[text=black] at (2.5,2) {$s_7$};
\fill[black] (1,0) circle (0.1cm);
\fill[black] (2,0) circle (0.1cm);
\fill[black] (3,0) circle (0.1cm);
\fill[black] (4,0) circle (0.1cm);
\fill[black] (5,0) circle (0.1cm);
\fill[black] (3,2) circle (0.1cm);
\fill[black] (3,1) circle (0.1cm);
\end{tikzpicture}
  \caption{Affine Coxeter group $\tilde{E}_6$.}
  \label{E6}
\end{figure}

\begin{proof}As before, we can assume that $H=G_I$ with $I= S\setminus \lbrace s_i\rbrace$ for some $s_i\in S$. The result is obvious for $G=D_{\infty}$, and the case $G=\tilde{E}_6$ (see Figure \ref{E6}) does not present any difficulty because every graph automorphism of a direct factor of $G_I$ extends to an automorphism of the entire group $G$. The remaining four cases are more complex.

\medskip
\noindent \textbf{Case 1}. $G=\tilde{E}_7$. 
\newline Consider the numbering of the elements of $S$ shown in Figure \ref{E7}. For reasons of symmetry, we only have to deal with $I=S\setminus \lbrace s_i\rbrace$ where $i\in \lbrace 1,2,3,4,8\rbrace$. If $s_i\neq s_2$ then every graph automorphism of a direct factor of $G_I$ extends to an automorphism of the entire group $G$. Then, suppose that $s_i=s_2$. Then $G_I$ splits into an internal direct product $G_I=UV$ with $U=\langle s_3,s_4,s_5,s_6,s_7,s_8\rangle$ and $V=\langle s_1\rangle$. Since $\varphi(G_I)=G_I$, we have $\varphi(U)=U$ and $\varphi_{\vert U}$ is reflection-preserving, so $\varphi_{\vert U}$ is inner-by-graph. Hence, there is an element $u\in U$ such that $\mathrm{ad}(u)\circ \varphi_{\vert U}$ is a graph automorphism of $U$, and thus $\mathrm{ad}(u)\circ \varphi$ fixes $s_4,\ldots,s_7$ and fixes or swaps $s_3$ and $s_8$. We will prove that $\mathrm{ad}(u)\circ \varphi$ must fix $s_3$ and $s_8$, and therefore that we can take $\sigma=\mathrm{ad}(u^{-1})$.

\begin{figure}[h!]
  \centering
    \begin{tikzpicture}
\draw[black, very thick] (1,0) -- (7,0);
\draw[black, very thick] (4,0) -- (4,1);
\node[text=black] at (1,-0.5) {$s_{1}$};
\node[text=black] at (2,-0.5) {$s_2$};
\node[text=black] at (3,-0.5) {$s_3$};
\node[text=black] at (4,-0.5) {$s_4$};
\node[text=black] at (5,-0.5) {$s_5$};
\node[text=black] at (6,-0.5) {$s_6$};
%\node[text=black] at (4,0.5) {$a$};
\node[text=black] at (7,-0.5) {$s_7$};
\node[text=black] at (4,1.5) {$s_8$};
\fill[black] (1,0) circle (0.1cm);
\fill[black] (2,0) circle (0.1cm);
\fill[black] (3,0) circle (0.1cm);
\fill[black] (4,0) circle (0.1cm);
\fill[black] (5,0) circle (0.1cm);
\fill[black] (6,0) circle (0.1cm);
\fill[black] (7,0) circle (0.1cm);
\fill[black] (4,1) circle (0.1cm);
\end{tikzpicture}
  \caption{Affine Coxeter group $\tilde{E}_7$.}
  \label{E7}
\end{figure}

%Note that $U$ has a graph automorphism that does not extend to an automorphism of $G$, namely the automorphism that fixes $s_4,s_5,s_6,s_7$ and that swaps $s_2$ and $s_8$. 

\smallskip \noindent Consider the set $J=S\setminus \lbrace s_1\rbrace$. Then $G_J$ is a maximal finite subgroup of $G$, and thus there exist an element $g\in G$ and a graph automorphism $\theta$ of $G$ such that $\mathrm{ad}(g)\circ \theta\circ \varphi$ coincides with the identity on $G_J$.  Let $r:G\rightarrow G'=\langle s_4,\ldots,s_8\rangle$ denote the retraction defined as follows: $r(s_1)=s_7$, $r(s_2)=s_6$, $r(s_3)=s_5$ and $r(s_i)=s_i$ for $i\in \lbrace 4,\ldots,8\rbrace$. We distinguish two cases. 

\smallskip \noindent First, suppose that $\theta=\mathrm{id}$. Note that $\mathrm{ad}(gu^{-1})\circ \mathrm{ad}(u)\circ  \varphi (s_5)=s_5=\mathrm{ad}(gu^{-1})(s_5)$ and that $\mathrm{ad}(gu^{-1})\circ \mathrm{ad}(u)\circ  \varphi (s_8)=s_8=\mathrm{ad}(gu^{-1})(s_k)$ with $k=8$ if $\mathrm{ad}(u)\circ \varphi$ fixes $s_3$ and $s_8$, and $k=3$ if $\mathrm{ad}(u)\circ \varphi$ swaps $s_3$ and $s_8$. Then, defining $h=r(gu^{-1})$, we obtain $s_5=\mathrm{ad}(h)(s_5)$ and $s_8=\mathrm{ad}(h)(r(s_k))$. If $k=3$ then $r(s_k)=s_5$ and thus $s_5=s_8$, a contradiction. Hence $k=8$, which proves that $\mathrm{ad}(u)\circ \varphi$ fixes $s_3$ and $s_8$.

%$\theta$ fixes $s_8$ and $s_4$ and that swaps $s_1$ and $s_7$, $s_2$ and $s_6$, $s_3$ and $s_5$

\smallskip \noindent Then, suppose that $\theta$ is the graph automorphism of $G$ of order 2. The following equalities hold: $\mathrm{ad}(g\theta(u^{-1}))\circ \theta\circ  \mathrm{ad}(u)\circ  \varphi (s_5)=s_5=\mathrm{ad}(g\theta(u^{-1}))(s_3)$ and $\mathrm{ad}(g\theta(u^{-1}))\circ \theta \circ \mathrm{ad}(u)\circ  \varphi (s_8)=s_8=\mathrm{ad}(g\theta(u^{-1}))(s_k)$ with $k=8$ if $\mathrm{ad}(u)\circ \varphi$ fixes $s_3$ and $s_8$, and $k=5$ if $\mathrm{ad}(u)\circ \varphi$ swaps $s_3$ and $s_8$. Then, defining $h=r(g\theta(u^{-1}))$, we obtain $s_5=\mathrm{ad}(h)(s_5)$ and $s_8=\mathrm{ad}(h)(r(s_k))$. If $k=5$ then $r(s_k)=s_5$ and thus $s_5=s_8$, a contradiction. Hence $k=8$, which proves that $\mathrm{ad}(u)\circ \varphi$ fixes $s_3$ and $s_8$.

%Then, suppose that $\theta$ is the automorphism of $G$ that fixes $s_8$ and $s_4$ and that swaps $s_1$ and $s_7$, $s_2$ and $s_6$, $s_3$ and $s_5$. Let $r:G\rightarrow G'=\langle s_1,\ldots,s_4, s_8\rangle$ denote the retraction defined as follows: $r(s_7)=s_1$, $r(s_6)=s_2$, $r(s_5)=s_3$ and $r(s_i)=s_i$ for $i\in \lbrace 1,\ldots,4,8\rbrace$. Note that $\mathrm{ad}(g\theta(u^{-1}))\circ \theta \circ \mathrm{ad}(u)\circ  \varphi (s_5)=s_5=\mathrm{ad}(g\theta(u^{-1}))(s_3)$ and that $\mathrm{ad}(g\theta(u^{-1}))\circ \theta\circ \mathrm{ad}(u)\circ  \varphi (s_8)=s_8=\mathrm{ad}(g\theta(u^{-1}))(s_k)$ with $k=8$ if $\mathrm{ad}(u)\circ \varphi$ fixes $s_3$ and $s_8$, and $k=5$ if $\mathrm{ad}(u)\circ \varphi$ swaps $s_3$ and $s_8$. Then, defining $h=r(g\theta(u^{-1}))$, we obtain $s_3=\mathrm{ad}(h)(s_5)$ and $s_8=\mathrm{ad}(h)(r(s_k))$. If $k=3$ then $r(s_k)=s_5$ and thus $s_5=s_8$, a contradiction. Hence $k=8$, which proves that $\mathrm{ad}(u)\circ \varphi$ fixes $s_3$ and $s_8$.

\medskip
\noindent \textbf{Case 2}. $G=\tilde{E}_8$. 
\newline Consider the numbering of the elements of $S$ shown in Figure \ref{E8}. Let us recall that $I=S\setminus \lbrace s_i\rbrace$. If $s_i$ belongs to $\lbrace s_2,s_3,s_9\rbrace$ then every automorphism of $G_I$ is inner. If $s_i$ belongs to $\lbrace s_5,s_7,s_8\rbrace$ then every reflection-preserving automorphism of $G_I$ is inner (for $s_5$, this is because the non-trivial graph automorphism of $D_5$ is inner, and for $s_7$ and $s_8$ this is because $G_I$ has no non-trivial graph automorphism). So, let us suppose that $s_i$ belongs to $\lbrace s_1,s_4,s_6\rbrace$. First, suppose that $s_i=s_1$. By assumption, $\varphi(H)=H$ and, after replacing $\varphi$ by $\mathrm{ad}(h)\circ \varphi$ for some $h\in H$, we can assume that $\varphi_{\vert H}$ is a graph automorphism of $H$. Let us prove that $\varphi_{\vert H}$ is the identity, i.e., that $\varphi$ does not swap $s_2$ and $s_9$. Suppose towards a contradiction that $\varphi(s_2)=s_9$ and $\varphi(s_9)=s_2$. Let $J=S\setminus \lbrace s_8\rbrace$ and define $K=G_J$. There exist $g\in G$ and a graph automorphism $\sigma$ of $G$ such that $\mathrm{ad}(g)\circ\sigma\circ \varphi(K)=K$. However, note that the only graph automorphism of $G$ is the identity, so $\mathrm{ad}(g)\circ \varphi(K)=K$. In addition, since every reflection-preserving automorphism of $K$ is inner, we can assume (after replacing $g$ by $kg$ for some $k\in G$) that $\mathrm{ad}(g)\circ \varphi$ is the identity of $K$. Hence, we have $\mathrm{ad}(g)\circ \varphi(s_9)=s_9=\mathrm{ad}(g)(s_2)$ and $\mathrm{ad}(g)\circ \varphi(s_2)=s_2=\mathrm{ad}(g)(s_9)$, and $\mathrm{ad}(g)\circ \varphi(s_i)=s_i=\mathrm{ad}(g)(s_i)$ for every $i\notin \lbrace 1,2,8,9\rbrace$. Then, note that $G=\mathbb{Z}^n\rtimes K$, so there is a retraction $r:G\rightarrow K$. Define $k=r(g)$. This element satisfies $ks_2k^{-1}=s_9$ and $ks_9k^{-1}=s_2$ and $ks_ik^{-1}=s_i$ for $i\in \lbrace 3,4,5,6,7\rbrace$. A computer verification using Sage shows that there is no such element $k\in K\simeq E_8$. The remaining cases $s_i=s_4$ and $s_i=s_6$ are similar: for $s_i=s_6$, we conclude using the fact that there is no $k\in K$ such that $\mathrm{ad}(k)$ swaps $s_1$ and $s_5$, $s_2$ and $s_4$, and fixes $s_3,s_9,s_7$; for $s_i=s_4$, we conclude using the fact that there is no $k\in K$ such that $\mathrm{ad}(k)$ swaps $s_2$ and $s_5$, $s_3$ and $s_6$, and $s_9$ and $s_7$ (in both cases, this can be checked using Sage).

\begin{comment}
W = CoxeterGroup(["E", 8], implementation="permutation");
s = W.simple_reflections();
# 1-3-4-5-6-7-8
#     |
#     2
# échange 3,4,2 et 6,7,8 dans cet ordre (optionnel : le carré fixe 1, i.e., w^2*s[1]==s[1]*w^2) ? Il n'existe pas de tel élément (pas besoin de la condition w^2 fixe 1)
#for w in W:
#     if (w*s[3]==s[6]*w and w*s[6]==s[3]*w and w*s[4]==s[7]*w and w*s[7]==s[4]*w and w*s[2]==s[8]*w and w*s[8]==s[2]*w):
#                    print(w);
# échange 2 et 3, et fixe tout sauf 1 ? Il n'existe pas de tel élément.
#for w in W:
#     if (w*s[2]==s[3]*w and w*s[3]==s[2]*w and w*s[4]==s[4]*w and w*s[5]==s[5]*w and w*s[6]==s[6]*w and w*s[7]==s[7]*w and w*s[8]==s[8]*w):
#                    print(w);
# échange 2 et 5, et fixe tout sauf 6 ? Il existe exactement un élément.
#for w in W:
#     if (w*s[2]==s[5]*w and w*s[5]==s[2]*w and w*s[4]==s[4]*w and w*s[7]==s[7]*w and w*s[1]==s[1]*w and w*s[3]==s[3]*w and w*s[8]==s[8]*w):
#                    print(w);
# échange 1 et 6, 3 et 5, et fixe 2, 4, 8 ? Il n'existe pas de tel élément.
#for w in W:
#     if (w*s[1]==s[6]*w and w*s[6]==s[1]*w and w*s[3]==s[5]*w and w*s[5]==s[3]*w and w*s[2]==s[2]*w and w*s[4]==s[4]*w and w*s[8]==s[8]*w):
#                    print(w);
\end{comment}

\begin{figure}[h!]
  \centering
    \begin{tikzpicture}
\draw[black, very thick] (1,0) -- (8,0);
\draw[black, very thick] (3,0) -- (3,1);
\node[text=black] at (1,-0.5) {$s_{1}$};
\node[text=black] at (2,-0.5) {$s_2$};
\node[text=black] at (3,-0.5) {$s_3$};
\node[text=black] at (4,-0.5) {$s_4$};
\node[text=black] at (5,-0.5) {$s_5$};
\node[text=black] at (6,-0.5) {$s_6$};
%\node[text=black] at (4,0.5) {$a$};
\node[text=black] at (7,-0.5) {$s_7$};
\node[text=black] at (8,-0.5) {$s_8$};
\node[text=black] at (3,1.5) {$s_9$};
\fill[black] (1,0) circle (0.1cm);
\fill[black] (2,0) circle (0.1cm);
\fill[black] (3,0) circle (0.1cm);
\fill[black] (4,0) circle (0.1cm);
\fill[black] (5,0) circle (0.1cm);
\fill[black] (6,0) circle (0.1cm);
\fill[black] (7,0) circle (0.1cm);
\fill[black] (8,0) circle (0.1cm);
\fill[black] (3,1) circle (0.1cm);
\end{tikzpicture}
  \caption{Affine Coxeter group $\tilde{E}_8$.}
  \label{E8}
\end{figure}

\medskip
\noindent \textbf{Case 3}. $G=\tilde{F}_4$. 
\newline Consider the numbering of the elements of $S$ shown in Figure \ref{F4}. If $s_i=s_1$ then every reflection-preserving automorphism of $G_I$ is inner. If $s_i=s_2$ then $G_I\simeq \mathbb{Z}/2\mathbb{Z}\times S_4$, so every automorphism of $G_I$ is inner. If $s_i=s_4$ then $G_I\simeq B_2  \times \mathbb{Z}/2\mathbb{Z}$, so every reflection-preserving automorphism of $G_I$ is inner. The remaining cases are $s_i=s_3$ and $s_i=s_5$. Suppose that $s_i=s_3$ and define $H=G_I$. By assumption, we have $\varphi(H)=H$. After replacing $\varphi$ by $\mathrm{ad}(g)\circ \varphi$ for some $g\in H$ if necessary, we can assume that $\varphi$ is a graph automorphism of $H$. We need to prove that $\varphi$ does not swap the pairs $\lbrace s_1,s_2\rbrace$ and $\lbrace s_4,s_5\rbrace$. Define $K=G_J$ where $J=S\setminus \lbrace s_4\rbrace$. There exist $g\in G$ and a graph automorphism $\sigma$ of $G$ such that $\mathrm{ad}(g)\circ\sigma\circ \varphi(K)=K$. However, note that the only graph automorphism of $G$ is the identity, so $\mathrm{ad}(g)\circ \varphi(K)=K$. In addition, since every reflection-preserving automorphism of $K$ is inner, we can assume that $\mathrm{ad}(g)\circ \varphi$ is the identity of $K$. It follows that $\mathrm{ad}(g)\circ \varphi(s_5)=s_5$, and thus that $\varphi(s_5)$ is conjugate to $s_5$. Hence, we cannot have $\varphi(\langle s_4,s_5\rangle)=\langle s_1,s_2\rangle$ since every involution of $\langle s_1,s_2\rangle$ is conjugate to $s_2$, which is not conjugate to $s_5$ (indeed, $s_2$ and $s_5$ are not connected in the odd graph). Last, suppose that $s_i=s_5$. As above, we assume that $\varphi$ is a graph automorphism of $G_I$. If $\varphi(s_4)=s_1$ then $\varphi(s_5)$ is conjugate to $s_1$ (because $s_4$ and $s_5$ are conjugate) and we get a contradiction as above; thus $\varphi_{\vert G_I}$ must be the identity.

%The only problematic cases are $s_3$ (because there is a graph automorphism of $G_I$ that swaps $s_1$ and $s_2$, and there is a graph automorphism that swaps $\langle s_1,s_2\rangle$ and $\langle s_4,s_5\rangle$) and $s_5$ (because of the graph automorphism of $G_I$ that swaps $s_1,s_2$ and $s_3,s_4$).

\begin{figure}[h!]
  \centering
    \begin{tikzpicture}
\draw[black, very thick] (1,0) -- (5,0);
\node[text=black] at (1,-0.5) {$s_{1}$};
\node[text=black] at (2,-0.5) {$s_2$};
\node[text=black] at (3,-0.5) {$s_3$};
\node[text=black] at (4,-0.5) {$s_4$};
\node[text=black] at (5,-0.5) {$s_5$};
\node[text=black] at (2.5,0.5) {$4$};
%\node[text=black] at (4,0.5) {$a$};
\fill[black] (1,0) circle (0.1cm);
\fill[black] (2,0) circle (0.1cm);
\fill[black] (3,0) circle (0.1cm);
\fill[black] (4,0) circle (0.1cm);
\fill[black] (5,0) circle (0.1cm);
\end{tikzpicture}
  \caption{Affine Coxeter group $\tilde{F}_4$.}
  \label{F4}
\end{figure}

\medskip
\noindent \textbf{Case 4}. $G=\tilde{G}_2$. 
\newline Consider the numbering of the elements of $S$ shown in Figure \ref{G2}. We distinguish several cases. If $G_I=\langle s_2,s_3\rangle$ then $G_I$ is isomorphic to $S_3$, so every automorphism of $G_I$ is inner and thus is the restriction of an automorphism of $G$. Then, suppose that $G_I=\langle s_1,s_2\rangle$. After composing $\varphi$ with $\mathrm{ad}(g)$ for some $g\in G_I$, we can assume without loss of generality that $\varphi_{\vert G_I}$ is a graph automorphism. Let us prove that $\varphi(s_1)=s_1$ and $\varphi(s_2)=s_2$. Suppose towards a contradiction that $\varphi(s_2)=s_1$. Consider $K=\langle s_2,s_3\rangle$. There exist $s,s'\in S$ and $g\in G$ such that $\varphi(K)$ is contained in $g\langle s,s'\rangle g^{-1}$. But $K$ is a maximal finite subgroup and $\varphi$ is class-permuting (see Definition \ref{class-permuting}), so $\varphi(K)=g\langle s,s'\rangle g^{-1}$, hence $\lbrace s,s'\rbrace=\lbrace s_2,s_3\rbrace$ and $\varphi(K)=\langle s_1,\varphi(s_3)\rangle=g\langle s_2,s_3\rangle g^{-1}$. But any involution of $\langle s_2,s_3\rangle$ is conjugate to $s_2$ (and $s_3$), so $s_1$ is conjugate to $s_2$, which contradicts Lemma 3.3.3 in \cite{davis} (since $s_1$ and $s_2$ are not connected in the odd graph). Last, suppose that $G_I=\langle s_1,s_3\rangle$. As in the previous case, we can assume that $\varphi_{\vert G_I}$ is a graph automorphism. If $\varphi(s_3)=s_1$ then we obtain a contradiction as above, thus $\varphi(s_1)=s_1$ and $\varphi(s_3)=s_3$. Conclusion: in all cases, there exists an automorphism $\sigma$ of $G$ such that $\sigma_{\vert G_I}=\varphi_{\vert G_I}$.

\begin{figure}[h!]
  \centering
    \begin{tikzpicture}
\draw[black, very thick] (1,0) -- (3,0);
\node[text=black] at (1,-0.5) {$s_{1}$};
\node[text=black] at (2,-0.5) {$s_2$};
\node[text=black] at (3,-0.5) {$s_3$};
%\node[text=black] at (4,0.5) {$a$};
\node[text=black] at (1.5,0.5) {$6$};
\fill[black] (1,0) circle (0.1cm);
\fill[black] (2,0) circle (0.1cm);
\fill[black] (3,0) circle (0.1cm);
\end{tikzpicture}
  \caption{Affine Coxeter group $\tilde{G}_2$.}
  \label{G2}
\end{figure}
\end{proof}

\subsection{Affine Coxeter groups are homogeneous}\label{details3}

First, we will prove the result for irreducible affine Coxeter groups.

\begin{te}\label{details}
Irreducible affine Coxeter groups are AE-homogeneous.    
\end{te}

\begin{proof}Let $(G,S)$ be an irreducible affine Coxeter group. Let $u,v$ be tuples of elements of $G$, and suppose that $u$ and $v$ have the same AE-type in $G$. If the subgroup of $G$ generated by the components of $u$ or $v$ is infinite, then we can conclude using Theorem~\ref{main_crystallo0} (proved in the previous section) that there exists an automorphism $\sigma$ of $G$ such that $\sigma(u)=v$. So, let us suppose that the subgroups of $G$ generated by the components of $u$ and $v$ respectively are finite. Note that $G$ has only finitely many conjugacy classes of finite subgroups (by Lemma~\ref{finitely_many}, or because every finite subgroup of $G$ is conjugate to a standard parabolic subgroup of $G$). Therefore, Lemma~\ref{key_lemma_1} applies and provides a class-permuting morphism $\varphi : G\rightarrow G$ such that $\varphi(u)=v$. Moreover, writing $G=\mathbb{Z}^n \rtimes G_0$ with $G_0$ finite, we can assume that $\varphi$ is injective on the $\mathbb{Z}^n$ factor (by Lemma~\ref{key_lemma_1} together with the fact that the $\mathbb{Z}^n$ factor is definable by a universal formula, see \ref{definable}); this will only be used in the second case below. Let $G_u$ be a maximal parabolic finite subgroup of $G$ containing $u$ (note that $G_u$ may not be unique). Define $G_v=\varphi(G_u)$. First, note that $G_v$ is a maximal finite subgroup: indeed, there is an integer $n\geq 1$ such that $\varphi^n(G_v)$ is conjugate to $G_u$ (since $\varphi$ is class-permuting); therefore, if $G_v$ is contained in a finite subgroup $H$ of $G$, then $G_u$ is contained in $\mathrm{ad}(g)\circ \varphi^n(H)$ for some $g\in G$ and thus, by maximality of $G_u$ for inclusion among finite subgroups of $G$, we obtain $\vert G_u\vert=\vert H\vert$, so $\vert G_v\vert=\vert H \vert$ and thus $G_v=H$. 

\emph{First case.} If $G$ is not of type $\tilde{B}_n$ with $n$ even then, by Lemma~\ref{lemma2}, there exists an automorphism $\theta$ of $G$ such that $\theta(\varphi(G_u))=G_u$ (here we do not use the fact that $\varphi$ is injective on the $\mathbb{Z}^n$ factor).

\emph{Second case.} If $G$ is of type $\tilde{B}_n$ with $n$ even, it is possible that, a priori, no automorphism of $G$ maps $G_v$ to $G_u$. Suppose, towards a contradiction, that this problematic case occurs. Then, by Lemma~\ref{lemma2}, $G_u$ and $G_v$ are obtained respectively from $S$ by removing $s_i$ and $s_{n-i+2}$ with $4\leq i\leq n-2$ and $i$ even, where the numbering of the vertices is given by Figure \ref{Blemme}. Moreover, note that $i-1\neq n-i+1$, otherwise $G_u$ and $G_v$ would be conjugate in $G$. Note also that $\varphi(C_G(Z(G_u)))$ is contained in $C_G(Z(G_v))$ because $\varphi$ maps $Z(G_u)$ isomorphically to $Z(G_v)$. Then, Lemma \ref{keykey} tells us that $C_G(Z(G_u))$ is virtually $\mathbb{Z}^{i-1}$ and that $C_G(Z(G_u))$ is virtually $\mathbb{Z}^{n-i+2}$, which forces $i-1\leq n-i+2$ (because $\varphi$ is injective on the $\mathbb{Z}^n$ factor of $G$ and $\mathbb{Z}^{i-1}$ is contained in this group). We now note that $\mathrm{ad}(g)\circ \varphi^p(G_u)=G_u$ for some $g\in G$ and $p\geq 2$ (as $\varphi$ permutes the finitely many conjugacy classes of finite subgroups of $G$), so $\mathrm{ad}(g)\circ \varphi^{p-1}(G_v)=G_u$, which gives the inequality in the other direction, namely $i-1\geq n-i+2$. Hence, $i-1=n-i+2$, which is a contradiction. Conclusion: the problematic cases do not occur and thus, by Lemma~\ref{lemma2}, there exists an automorphism $\theta$ of $G$ such that $\theta(\varphi(G_u))=G_u$. 

Define $\psi=\theta \circ \varphi$. Note that this endomorphism of $G$ is class-permuting. By Theorem \ref{extension}, there exists an automorphism $\sigma$ of $G$ such that $\sigma_{\vert G_u}=\psi_{\vert G_u}$. Last, note that the automorphism $\theta^{-1}\circ \sigma\in\mathrm{Aut}(G)$ coincides with $\theta^{-1}\circ \psi=\varphi$ on $G_u$, so $\theta^{-1}\circ \sigma(u)=v$, which concludes the proof.\end{proof}

%Let $(G,S)$ be an irreducible affine Coxeter group. Let $H,H'$ be two maximal finite subgroups of $G$. If $H$ and $H'$ are isomorphic, then there is an automorphism $\sigma$ of $G$ such that $\sigma(H)=H'$ except if $G$ is of type $\tilde{B}_n$ with $n$ even and $H,H'$ are conjugate to $G_I,G_{I'}$ where $I=S\setminus \lbrace s_i\rbrace$ and $I'=S\setminus \lbrace s_{n-i+2}\rbrace$ with $4\leq i\leq n-2$ and $i$ even, where the numbering is given by the following diagram: 

%keykey

We will prove below that the previous result remains valid without the irreducibility assumption.

\begin{te}\label{main_affine_homogeneous}
Affine Coxeter groups are AE-homogeneous.    
\end{te}

\begin{proof}
Write $G=G_1\times \cdots\times G_n$ where each $G_i$ is an irreducible affine Coxeter group. Let $S_i$ be a Coxeter generating set for $G_i$ and define $S=\cup_{1\leq i\leq n}S_i$. Let $u,v$ be tuples of elements of $G$, and suppose that $u$ and $v$ have the same AE-type in $G$. Lemma~\ref{key_lemma_1} provides a class-permuting morphism $\varphi : G\rightarrow G$ such that $\varphi(u)=v$ and $\varphi(H)\subset H$, where $H$ denotes the maximal free abelian subgroup of $G$ (i.e., $H$ is the translation subgroup of $G$). By Lemma \ref{lemma1}, $\varphi$ is reflection-preserving. It follows that, for each $s\in S$, $\varphi(s)$ belongs to some $G_i$. We claim that, for every $1\leq i\leq n$, $\varphi(G_i)$ is contained in some $G_j$. Recall that we denote by $S^G$ the set $\lbrace gsg^{-1} \ \vert \ g\in G, s\in S\rbrace$. To prove the claim, let us define an equivalence relation $\sim$ on $S^G$ as follows: for $s,s'\in S^G$, $s\sim s'$ if and only if $s=s'$ or there exist an integer $k\geq 1$ and $s_1,\ldots,s_k\in S^G$ such that $s_1=s$, $s_k=s'$ and $[s_i,s_{i+1}]\neq 1$ for every $1\leq i <k$. We will need the following two observations:
\begin{enumerate}[(1)]
    \item if $s\sim s'$ then $\varphi(s)\sim \varphi(s')$ because $\varphi(S^G)\subset S^G$ and $\varphi$ is injective on each finite dihedral group $\langle s_i,s_{i+1}\rangle$, so $\varphi(s_i)$ and $\varphi(s_{i+1})$ do not commute;
    \item if $s\sim s'$ then $s$ and $s'$ belong to the same $G_j$: indeed, each $s_i$ and $s_{i+1}$ belong to the same $G_j$ as they do not commute.
\end{enumerate}
Then, note that the reflections in the Coxeter generating set $S_i$ of $G_i$ are $\sim$-equivalent, so $\varphi(S_i)$ is contained in some $G_j$, so $\varphi(G_i)$ is contained in $G_j$.

%if $s,s'$ are two adjacent reflections in the Coxeter diagram of $G_i$, then $[s,s']\neq 1$, so $[\varphi(s),\varphi(s')]\neq 1$ since $\varphi$ is injective on the finite group $\langle s,s'\rangle$, so $\varphi(s),\varphi(s')$ belong to the same $G_j$. By induction, $\varphi(G_i)$ is contained in some $G_j$. 

\smallskip \noindent Hence, we can define a map $\theta : \lbrace 1,\ldots,n\rbrace\rightarrow \lbrace 1,\ldots,n\rbrace$ such that $\varphi(G_i)\subset G_{\theta(i)}$. 

\smallskip \noindent Let us prove that $\theta$ is a bijection. Since $\varphi$ is class-permuting, there exists an integer $N\geq 1$ such that, for each finite subgroup $H$ of $G$, $\varphi^N(H)$ is conjugate to $H$. In particular, if $H$ is contained in some $G_i$, then $\varphi^N(H)$ is contained in $G_i$ as well. Let $i,i'\in\lbrace 1,\ldots,n\rbrace$ be such that $\theta(i)=\theta(i')=j$, and let us prove that $i=i'$. Observe that $\varphi(G_i)$ and $\varphi(G_{i'})$ are contained in $G_j$. Let $H_i$ and $H_{i'}$ be non-trivial finite subgroups of $G_i$ and $G_{i'}$, respectively. The subgroups $\varphi^N(H_i)$ and $\varphi^N(H_{i'})$ are contained in $G_i$ and $G_{i'}$, respectively, and they are non-trivial as $\varphi$ is injective on finite subgroups of $G$, so $G_i$ and $G_{i'}$ are the unique direct factors of $G$ containing $\varphi^N(H_i)$ and $\varphi^N(H_{i'})$, respectively. Then, note that $\varphi^N(H_i)$ and $\varphi^N(H_{i'})$ are also contained in $\varphi^{N-1}(G_j)$. Since,  for each $k\in\lbrace 1,\ldots,n\rbrace$, $\varphi(G_k)$ is contained in $G_{\theta(k)}$, by induction $\varphi^{r}(G_k)$ is contained in $G_{\theta^r(k)}$ for every integer $r$. Therefore, $\varphi^{N-1}(G_j)$ is contained in $G_{\theta^{N-1}(j)}$, hence $\varphi^N(H_i)$ and $\varphi^N(H_{i'})$ are contained in $G_{\theta^{N-1}(j)}$. It follows that $i=\theta^{N-1}(j)=i'$, which proves that $\theta$ is injective, hence bijective. 

%For the following argument to be true we would need to add quantifiers in the formula in order to ensure that $\varphi$ is injective on $H$: observe that $\varphi$ is injectiv, so the rank of $G_i$ (that is the rank of the maximal $\mathbb{Z}^k$ contained in $G_i$) is smaller than the rank of $\varphi(G_i)$. But the only inclusions between irreducible affine Coxeter groups of same rank are $\tilde{D}_n\hookrightarrow\tilde{B}_n$ and $\tilde{D}_4\hookrightarrow\tilde{F}_4$ and $\tilde{A}_2\hookrightarrow\tilde{G}_2$ (ref?). It follows that $G_{\theta(i)}$ is isomorphic to $G_i$. 

%Note that $\theta(0)=0$ (because $G_0$ is the unique finite direct factor in the decomposition of $G$) and that the restriction $\varphi_{\vert G_0}$ is an automorphism of $G_0$. Then, l

\smallskip \noindent By Lemma \ref{max_finite}, two irreducible affine Coxeter groups $W$ and $W'$ are isomorphic if and only if they have the same maximal finite subgroups up to isomorphism (recall that each maximal finite subgroup of $W$ or $W'$ can be obtained by deleting one vertex from its Coxeter diagram). Note also that $\varphi$ preserves maximality of finite subgroups (as it is class-permuting). Therefore, $G_i$ and $G_{\theta(i)}$ must be isomorphic. Therefore, there exists an automorphism $\alpha$ of $G$ that permutes the $G_i$ in such a way that $\alpha\circ \varphi(G_i)\subset G_i$ (in other words, $\alpha$ induces a permutation of $\lbrace 1,\ldots,n\rbrace$ that is the inverse of the permutation $\theta$). Define $\psi=\alpha\circ \varphi$. For simplicity of notation, suppose that $u$ and $v$ are elements of $G$, not tuples (the general proof works in the same way). Write $u=u_1\cdots u_n$ with $u_i\in G_i$ and $v=v_1\cdots v_n$ with $v_i\in G_{\theta(i)}$. Note that $\varphi(u_i)=v_i$ for every $1\leq i\leq n$ by uniqueness of the decomposition of $v$ in the direct product $G=G_1\times \cdots\times G_n$. We will prove that, for each $1\leq i\leq n$, there exists an isomorphism $\sigma_i:G_i\rightarrow G_{\theta(i)}$ that coincides with $\varphi_{\vert G_i}$ (and thus $\sigma_i(u_i)=v_i$). Hence, the automorphism $\sigma$ of $G$ such that $\sigma_{\vert G_i}=\sigma_i$ for every $1\leq i\leq n$ satisfies $\sigma(u)=v$, which concludes the proof. It remains to find the $\sigma_i$. Fix an integer $1\leq i\leq n$. Let us distinguish two cases.

\smallskip \noindent {\bf Case 1}. $\langle u_i\rangle$ is finite.
\newline Let $U_i$ denote the maximal finite subgroup of $G_i$ containing $u_i$ and define $V_i=\psi(U_i)$ (where $\psi=\alpha\circ \varphi$ was defined in the previous paragraph). As in the proof of Theorem~\ref{details}, $V_i$ is a maximal finite subgroup of $G_i$, so there exists an automorphism $\beta_i$ of $G_i$ such that $\beta_i(V_i)=U_i$. The morphism $\beta_i\circ \psi_{\vert G_i}$ being class-permuting, Theorem \ref{extension} tells us that there exists an automorphism $\gamma_i$ of $G_i$ such that $\gamma_i$ and $\beta_i\circ \psi_{\vert G_i}$ coincide on $G_i$, i.e., $\alpha^{-1}_{\vert G_i}\circ \beta_i^{-1}\circ \gamma_i$ and $\varphi_{\vert G_i}$ coincide (on $G_i$). We can therefore define $\sigma_i$ as follows: \[\sigma_i=\alpha^{-1}_{\vert G_i}\circ \beta_i^{-1}\circ \gamma_i.\]

%Note in particular that $\sigma_i(u_i)=v_i$. 

\smallskip \noindent {\bf Case 2}. $\langle u_i\rangle$ is infinite.
\newline 
Note that $u$ and $v$ play symmetrical roles, so Lemma~\ref{key_lemma_1} provides a class-permuting morphism $\varphi' : G\rightarrow G$ such that $\varphi'(v)=u$ and $\varphi(H)\subset H$, where $H$ denotes the translation subgroup of $G$. Using the same arguments as those used for $\varphi$, this morphism $\varphi'$ permutes the direct factors of $G$, so there is an integer $K\geq 1$ such that $(\varphi'\circ \varphi)^K(G_i)\subset G_i$. Note also that $(\varphi'\circ \varphi)^K$ fixes $u$, so $(\varphi'\circ \varphi)^K$ fixes $u_i$. Moreover, $(\varphi'\circ \varphi)^K$ is class-permuting, and thus $(\varphi'\circ \varphi)^K_{\vert G_i}$ is class-permuting. Hence, Proposition \ref{15janvier2026} shows that $(\varphi'\circ \varphi)^K_{\vert G_i}$ is an automorphism of $G_i$. It follows that $\varphi_{\vert G_i}$ is an isomorphism between $G_i$ and $G_{\theta(i)}$, so we simply take $\sigma_i=\varphi_{\vert G_i}$.\end{proof}

\section{Homogeneity in torsion-generated hyperbolic groups}\label{homogeneity_hyperbolic}

Among infinite, non-affine Coxeter groups, a particularly interesting class is that of hyperbolic Coxeter groups in the sense of Gromov. At the intersection of affine Coxeter groups and hyperbolic Coxeter groups, we find only the infinite dihedral group. In this section, we study the homogeneity of hyperbolic Coxeter groups and, more generally, of torsion-generated hyperbolic groups (i.e., hyperbolic groups generated by torsion elements).

\subsection{Preliminaries}

\subsubsection{Relative co-Hopf property}

\begin{te}[see Theorem 2.31 in \cite{And18b}]\label{relative_co_Hopf}
Let $G$ be a hyperbolic group, and let $H\subset G$ be a subgroup of $G$. Suppose that $G$ is one-ended relative to $H$. Then there exists a finitely generated subgroup $H'\subset H$ such that $G$ is co-Hopfian relative to $H'$, meaning that every injective morphism $\varphi : G \rightarrow G$ such that $\varphi_{\vert H'}=\mathrm{id}_{H'}$ is an automorphism of $G$ (and thus $G$ is co-Hopfian relative to $H$).
\end{te}

\begin{rk}
Note that this theorem is stated and proved in \cite{And18b} under the stronger assumption that $H$ is finitely generated (in which case one can simply take $H'=H$), but this assumption is superfluous, as explained in the proof of Theorem 2.13 in \cite{And21b}.
\end{rk}

\begin{rk}
Note in particular that every one-ended hyperbolic group is co-Hopfian (by taking for $H$ the trivial subgroup in the previous theorem). This result was proved by Sela in \cite{Sel97} for torsion-free hyperbolic groups, and by Moioli in \cite{Moi13} for hyperbolic groups possibly with torsion. 
\end{rk}

\subsubsection{Relative Stallings, JSJ, centered splittings}

\begin{de}
Let $G$ be a group and let $H$ be a subgroup of $G$. A decomposition $\Delta$ of $G$ as a graph of groups is said to be \emph{relative to} $H$ if $H$ is contained in a conjugate of one of the vertex groups of $\Delta$, or, equivalently, if $H$ is elliptic in the Bass-Serre tree of $\Delta$.
\end{de}

\begin{de}\label{Stallings}
Let $G$ be a hyperbolic group and let $H$ be a subgroup of $G$. A \emph{Stallings decomposition} of $G$ relative to $H$, denoted by $\mathbf{S}_{G,H}$, is a decomposition of $G$ as a graph of groups with finite edge groups relative to $H$ whose vertex groups do not split non-trivially over finite groups relative to $H$. Note that such a splitting is not unique in general, but the conjugacy classes of one-ended vertex groups relative to $H$ do not depend on a particular Stallings decomposition of $G$ relative to $H$. These vertex groups are called the \emph{one-ended factors} of $G$ relative to $H$.
\end{de}

We denote by $\mathcal{Z}$ the class of groups that are virtually cyclic with infinite center. Any hyperbolic group $G$ that is one-ended relative to a subgroup $H\subset G$ has a canonical splitting over $\mathcal{Z}$ relative to $H$, that is a splitting that can be constructed in a natural and uniform way (see \cite{Sel97, Bow98, GL16}). This decomposition is a powerful tool to study the group $G$ and its first-order theory. We refer to the canonical JSJ decomposition over $\mathcal{Z}$ constructed in \cite{GL16} by means of the tree of cylinders as \emph{the} $\mathcal{Z}$-JSJ decomposition of $G$ relative to $H$, denoted by $\mathbf{JSJ}_{G,H}$. We collect below some definitions and results that will be useful in our proofs.

\begin{de}\label{def_rig_vert} Let $G$ be a hyperbolic group, let $H$ be a subgroup of $G$. Suppose that $G$ is one-ended relative to $H$. Let $T$ be a splitting of $G$ over $\mathcal{Z}$ relative to $H$. A vertex $v$ of $T$ and its stabilizer $G_v$ are said to be \emph{rigid} if $G_v$ is \emph{universally elliptic}, i.e., if it is elliptic in every splitting of $G$ over $\mathcal{Z}$ relative to $H$. Otherwise, $v$ and $G_v$ are said to be \emph{flexible}.\end{de}

A compact connected hyperbolic two-dimensional orbifold $\mathcal{O}$ is by definition the quotient of a convex subset $\bar{\mathcal{O}}$ of the hyperbolic plane $\mathbb{H}^2$, with geodesic boundary, by a non-elementary subgroup $G$ of $\mathrm{Isom}^{\pm}(\mathbb{H}^2)$ acting properly discontinuously (equivalently, $G$ is discrete), cocompactly and faithfully on $\bar{\mathcal{O}}$. This group $G$ is by definition the orbifold fundamental group of $\mathcal{O}$, denoted by $\pi_1^{\mathrm{orb}}(\mathcal{O})$. Equivalently, $G$ is a Fuchsian group. If $G$ is torsion-free, then $\mathcal{O}$ is a surface. If $G$ does not contain reflections, then the singular set of $\mathcal{O}$ is a finite collection of points ($\mathcal{O}$ has no mirrors), called the \emph{conical points} of $\mathcal{O}$, and we say that $\mathcal{O}$ is a \emph{conical orbifold}. Note that if we do not assume that the action of $G$ on $\bar{\mathcal{O}}$ is faithful, then $G$ is a finite extension of an orbifold group by the kernel of this action (which is largest normal finite subgroup of $G$), which leads to the following definition.

\begin{de}\label{FBO}A group $G$ is called a \emph{finite-by-orbifold group} if it is an extension \[1\rightarrow F\rightarrow G \rightarrow \pi_1^{\mathrm{orb}}(\mathcal{O})\rightarrow 1\]where $\mathcal{O}$ is a conical compact connected hyperbolic two-dimensional orbifold, possibly with (geodesic) boundary, and $F$ is an arbitrary finite group called the \emph{fiber}. An \emph{extended boundary subgroup} of $G$ is the preimage in $G$ of a boundary subgroup of the orbifold fundamental group $\pi_1^{\mathrm{orb}}(\mathcal{O})$. We define in the same way the \emph{extended conical subgroups} of $G$. The orbifold $\mathcal{O}$ is called \emph{the underlying orbifold of $G$}.\end{de}

The following terminology was introduced by Rips and Sela in \cite{RS97}, see also \cite[Definition 5.13]{GL16}.

\begin{de}\label{QH}Let $\Delta$ be a graph of groups and let $G$ be its fundamental group. A vertex $v$ of $\Delta$, and its stabilizer $G_v$, are said to be \emph{quadratically hanging} (denoted by QH) if $G_v$ is a finite-by-orbifold group (see Definition \ref{FBO}) such that, for any edge $e$ incident to $v$ in $\Delta$, the edge group $G_e$ is contained in an extended boundary subgroup of $G_v$ or is finite (in which case $G_e$ is contained in an extended conical subgroup or in the fiber of $G_v$).\end{de}

\begin{rk}\label{empty_boundary}
Note that the underlying orbifold of $G_v$ does not necessarily have non-empty boundary or conical points.
\end{rk}

The following theorem is one of the most important results of the theory of JSJ decompositions of hyperbolic groups: it provides a description of the flexible vertices of the canonical JSJ decomposition over $\mathcal{Z}$ (relative to a subgroup $H$). The result below is stated for a hyperbolic group, but it remains true in a much broader context. We refer the reader to \cite[Chapter III]{GL16} (see in particular \cite[Theorem 6.2]{GL16}).

\begin{te}\label{flexible implique rigide}Let $G$ be a one-ended hyperbolic group relative to a subgroup $H$. If a vertex group $G_v$ of the canonical JSJ decomposition of $G$ over $\mathcal{Z}$ relative to $H$ is flexible, then $G_v$ is quadratically hanging (see Definition \ref{QH}).\end{te}

Proposition \ref{JSJ} below summarizes the properties of $\mathbf{JSJ}_{G,H}$ that will be useful, and we refer the reader to \cite{GL16} for further details.

\begin{prop}\label{JSJ}Let $G$ be a hyperbolic group and let $H$ be a subgroup of $G$. Suppose that $G$ is one-ended relative to $H$. Let $T$ be the Bass-Serre tree of $\mathbf{JSJ}_{G,H}$.
\begin{enumerate}[(1)]
\item\textbf{Bipartition.} Every edge of $T$ joins a vertex labelled with a maximal virtually cyclic group to a vertex labelled with a group which is not virtually cyclic.
\item If $v$ is a QH vertex of $T$, then for every edge $e$ incident to $v$ in $T$, the edge group $G_e$ coincides with an extended boundary subgroup of $G_v$.
\item If $v$ is a QH vertex of $T$, then for every extended boundary subgroup $B$ of $G_v$, there exists an edge $e$ incident to $v$ in $T$ such that $G_e=B$. Moreover, the edge $e$ is unique.
\item \textbf{Acylindricity.} If an element $g\in G$ of infinite order fixes a segment of length $\geq 2$ in $T$, then this segment has length exactly 2 and its midpoint has virtually cyclic stabilizer. 
%\item Let $v$ be a vertex of $T$, and let $e,e'$ be two distinct edges incident to $v$. If $G_v$ is not virtually cyclic, then the group $\langle G_e,G_{e'}\rangle$ is not virtually cyclic. 
\item If $G_v$ is a flexible vertex group of $T$, then it is QH.
\end{enumerate}
\end{prop}

When the hyperbolic group $G$ is not one-ended relative to $H$, we consider decompositions of $G$ relative to $H$ over the class of groups that are either finite or virtually cyclic with infinite center, denoted by $\overline{\mathcal{Z}}$. Such a decomposition of $G$ can be obtained from a reduced Stallings splitting of $G$ relative to $H$, say $\mathbf{S}_{G,H}$, by replacing each vertex $x$ such that $G_x$ is one-ended by $\mathbf{JSJ}_{G_x,H}$ (the canonical JSJ decomposition of $G_x$ over $\mathcal{Z}$ relative to $H$) if $H$ is contained in $G_x$ and by $\mathbf{JSJ}_{G_x}$ otherwise. The new edge groups are defined as follows: if $e=[x,y]$ is an edge of (the Bass-Serre tree of) $\mathbf{S}_{G,H}$, then $G_e$ is finite, so $G_e$ fixes a vertex $x'$ in $\mathbf{JSJ}_{G_x,H}$ (or $\mathbf{JSJ}_{G_x}$) and a vertex $y'$ in $\mathbf{JSJ}_{G_y,H}$ (or $\mathbf{JSJ}_{G_y}$), and the edge $e$ in $\mathbf{S}_{G,H}$ is simply replaced by the edge $e'=[x',y']$ in $\mathbf{JSJ}_{G,H}$. Note that the vertices $x'$ and $y'$ fixed by $G_e$ are not necessarily unique, and that the reduced Stallings splitting $\mathbf{S}_{G,H}$ is not unique in general, therefore, the resulting splitting of $G$ is not unique in general, but for convenience we will still use the notation $\mathbf{JSJ}_{G,H}$ to denote one such splitting.

By collapsing certain subgraphs of $\mathbf{JSJ}_{G,H}$ to a single vertex, we obtain a new graph that enjoys nice properties. This leads to the concept of a \emph{centered splitting} (see Definition \ref{centered} below). More precisely, a centered splitting of a group $G$ (relative to a subgroup $H$) is a splitting over $\overline{\mathcal{Z}}$ (the class of groups that are either finite or virtually cyclic with infinite center) that satisfies a list of nice properties inherited from the canonical JSJ decomposition of a one-ended hyperbolic group, even though the group $G$ is not assumed to be one-ended (relative to $H$) and finite edge groups are allowed. The following definition is a slight variant of Definition 3.8 in \cite{And18}. Note that we allow QH vertex groups to have empty boundary (which is not the case in \cite{And18}, see \cite{AP24} for more details). In \cite[Section 4.4.2]{AP24}, we explain in more detail how to construct a centered splitting of a hyperbolic group from a JSJ decomposition over $\overline{\mathcal{Z}}$. The construction works in exactly the same way for splittings relative to a subgroup, and we refer the reader to \cite[Section 4.4.2]{AP24} for more details.

\begin{de}\label{centered}Let $G$ be a group and let $H$ be a subgroup of $G$. Let $\Delta$ be a decomposition of $G$ as a graph of groups relative to $H$. Let $V$ be the set of vertices of (the underlying graph of) $\Delta$. We suppose that $\vert V\vert \geq 2$. The graph $\Delta$ is said to be \emph{centered} if the following conditions hold.
\begin{enumerate}[(1)]
\item \textbf{Strong bipartition.} The underlying graph is bipartite in a strong sense: there exists a vertex $v\in V$, called the \emph{central vertex}, such that every edge connects $v$ to a vertex in $V\setminus \lbrace v\rbrace$. Moreover, the vertex $v$ is QH and $H$ is not contained in a conjugate of $G_v$. %In the rest of the paper, its stabilizer $G_v$ will often be denoted by $Q$ ($Q$ as in QH).
\item For every edge $e$, the edge group $G_e$ coincides with an extended boundary or conical subgroup of $G_v$.
\item For every extended boundary or conical subgroup $K$ of $G_v$, there exists an edge $e$ such that $G_e$ is conjugate to $K$ in $G_v$. Moreover, the edge $e$ is unique.
\item \textbf{Acylindricity.} Let $K$ be a subgroup of $G$, and suppose that $K$ is not contained in the fiber of $G_v$. If $K$ fixes a segment of length $\geq 2$ in the Bass-Serre tree of the splitting, then this segment has length exactly 2 and its endpoints are translates of $v$.
\end{enumerate}
In this paper, a centered splitting of $G$ relative to $H$ will often be denoted by $\mathbf{C}_{G,H}$.
\end{de} 

\subsubsection{The relative modular group}

\begin{de}Let $G$ be a hyperbolic group, and let $H\subset G$ be a subgroup of $G$. Suppose that $G$ is one-ended relative to $H$. We denote by $\mathrm{Aut}_H(G)$ the subgroup of $\mathrm{Aut}(G)$ consisting of all automorphisms whose restriction to $H$ is the conjugacy by an element of $G$. The modular group $\mathrm{Mod}_H(G)$ of $G$ relative to $H$ is the subgroup of $\mathrm{Aut}_H(G)$ consisting of all automorphisms $\sigma$ satisfying the following conditions:
\begin{itemize}
    \item the restriction of $\sigma$ to each non-QH vertex group of $\mathbf{JSJ}_{G,H}$ is the conjugacy by an element of $G$;
    \item the restriction of $\sigma$ to each finite subgroup of $G$ is the conjugacy by an element of $G$;
    \item $\sigma$ acts trivially on the underlying graph of $\mathbf{JSJ}_{G,H}$.
\end{itemize}
When $G$ is one-ended, the modular group of $G$ relative to the trivial subgroup is simply called the modular group, denoted by $\mathbf{Mod}(G)$.
\end{de}

%\begin{te}\label{shortening_0}Let $\Gamma_1$ and $\Gamma_2$ be hyperbolic groups. Assume that $\Gamma_1$ is one-ended. Then there exists a finite subset $F \subset \Gamma_1$ of non-trivial elements such that, for any non-injective morphism $\varphi:\Gamma_1\rightarrow \Gamma_2$, there exists a modular automorphism $\sigma\in\mathrm{Mod}(\Gamma_1)$ such that $\varphi\circ\sigma$ kills an element of $F$.\end{te}

%We also need relative versions of the previous definition and theorem.

%\begin{de}The relative modular group.\end{de}

\begin{te}[see Theorem 2.32 in \cite{And18b}]\label{shortening}Let $\Gamma_1$ and $\Gamma_2$ be hyperbolic groups. Let $H$ be a subgroup of $\Gamma_1$ that embeds into $\Gamma_2$, and let us fix an embedding $i:H\rightarrow \Gamma_2$. Assume that $\Gamma_1$ is one-ended relative to $H$. Then there exist a finite subset $F \subset \Gamma_1\setminus\lbrace 1\rbrace$ and a finitely generated subgroup $H'\subset H$ such that, for any non-injective morphism $\varphi:\Gamma_1\rightarrow \Gamma_2$ that coincides with $i$ on $H'$ up to conjugation,
there exists a relative modular automorphism $\sigma\in\mathrm{Mod}_H(\Gamma_1)$ such that $\varphi\circ\sigma$ kills an element of $F$.
\end{te}

\begin{rk}
Note that this theorem is stated and proved in \cite{And18b} under the stronger assumption that $H$ is finitely generated (in which case one can simply take $H'=H$), but this assumption is superfluous, as explained in the proof of Theorem 2.13 in \cite{And21b}.
\end{rk}

\subsubsection{Preretractions}

%\begin{de}JSJ (relatif), scindement centré (relatif), morphismes reliés (relatifs), prérétraction (relative). Include in the definition of a $\mathbf{JSJ}_{G,u}$-preretraction that $u$ is mapped to (a conjugate of) $u$.\end{de}

%and let $\mathbf{JSJ}_{G,H}$ be a relative JSJ splitting of $G$ over $\overline{\mathcal{Z}}$ relative to $H$

The following two definitions are inspired by Definitions 5.15 and 5.9 in \cite{Per11}, and Lemma \ref{injective_relative} also appears in \cite{Per11} in the context of torsion-free hyperbolic groups. 

\begin{de}\label{related}Let $G$ be a hyperbolic group and let $H\subset G$ be a subgroup of $G$. Two endomorphisms $\varphi, \psi$ of $G$ are said to be \emph{$\mathbf{JSJ}_{G,H}$-related} if, for every finite subgroup $K$ of $G$ or non-QH vertex group $K$ of $\mathbf{JSJ}_{G,H}$, there exists an element $g\in G$ such that $\varphi_{\vert K}=\mathrm{ad}(g)\circ \psi_{\vert K}$.\end{de}

\begin{rk}
Note in particular that there exists an element $g\in G$ such that $\varphi_{\vert H}=\mathrm{ad}(g)\circ \psi_{\vert H}$, since $H$ is contained in a non-QH vertex group of $\mathbf{JSJ}_{G,H}$
\end{rk}

\begin{de}\label{pre_JSJ}Let $G$ be a hyperbolic group and let $H\subset G$ be a subgroup of $G$. An endomorphism $\varphi$ of $G$ is called a \emph{$\mathbf{JSJ}_{G,H}$-preretraction} if it is $\mathbf{JSJ}_{G,H}$-related to $\mathrm{id}_G$, i.e.\ if it coincides with an inner automorphism on every non-QH vertex group of $\mathbf{JSJ}_{G,H}$ (and thus on $H$) and on every finite subgroup of $G$. A $\mathbf{JSJ}_{G,H}$-preretraction is said to be \emph{non-degenerate} if it sends each QH vertex group isomorphically to a conjugate of itself.\end{de}

We need a similar notion for centered splittings.

\begin{de}\label{pre_centered}Let $G$ be a group, let $H$ be a subgroup of $G$ and let $\mathbf{C}_{G,H}$ be a centered splitting of $G$ relative to $H$. An endomorphism $\varphi$ of $G$ is called a \emph{$\mathbf{C}_{G,H}$-preretraction} if it coincides with an inner automorphism on every non-central vertex group of $\mathbf{C}_{G,H}$ (and thus on every finite subgroup of $G$ and on $H$). A $\mathbf{C}_{G,H}$-preretraction is said to be \emph{non-degenerate} if it sends the central vertex group isomorphically to a conjugate of itself.\end{de}

%\begin{lemme}[Lemma 5.6 in \cite{AP24}]\label{injective}Let $G$ be a one-ended hyperbolic group. Then every non-degenerate $\mathbf{JSJ}_G$-preretraction is injective.\end{lemme}

\begin{lemme}\label{injective_relative}
Let $G$ be a hyperbolic group and let $H\subset G$ be a subgroup. Suppose that $G$ is one-ended relative to $H$. Then every non-degenerate $\mathbf{JSJ}_{G,H}$-preretraction is injective.\end{lemme}

\begin{proof}
The non-relative version of this lemma (that is, when $H$ is the trivial subgroup) is an immediate consequence of \cite[Proposition 7.1]{And18} (this proposition applies because a hyperbolic does not contain $\mathbb{Z}^2$ and is $K$-CSA for some $K$ large enough), and the proof works in the same way when the subgroup $H$ is non-trivial.
\end{proof}

\begin{lemme}\label{injective2}
Let $G$ be a finitely torsion-generated group and let $H\subset G$ be a subgroup. Let $\mathbf{C}_{G,H}$ be a centered splitting of $G$ relative to $H$, with central vertex $v$. Then every $\mathbf{C}_{G,H}$-preretraction is non-degenerate (i.e.\ the central vertex group $G_v$ is mapped isomorphically to a conjugate of itself ).\end{lemme}

\begin{proof}
The lemma is a relative version of Corollary 5.11 in \cite{AP24}; adapting the proof is immediate.
\end{proof}

\subsection{Isomorphisms between vertex groups of relative Stallings splittings}

The following lemma is similar to Lemma \ref{key_lemma_1}.

\begin{lemme}\label{non_conj_to_non_conj}
Let $G=\langle s_1,\ldots,s_n \ \vert \ \Sigma(s_1,\ldots,s_n)=1\rangle$ be a finitely presented group that has only a finite number of conjugacy classes of finite subgroups. Then there exists a universal formula $\mathrm{Finite}_G(x_1,\ldots,x_n)$ such that, for any $(g_1,\ldots,g_n)\in G^n$, we have: $G\models \mathrm{Finite}_G(g_1,\ldots,g_n)$ if and only if the map $ \lbrace s_1,\ldots,s_n \rbrace \rightarrow G : s_i\mapsto g_i$ extends to a class-permuting endomorphism $\varphi$ of $G$ (see Definition \ref{class-permuting}).\end{lemme}

\begin{proof}An easy but crucial observation is that there is a one-to-one correspondence between $\mathrm{Hom}(G,G)$ and the set $\lbrace (g_1,\ldots,g_n)\in G^n, \ \Sigma(g_1,\ldots,g_n)=1\rbrace$. Let $H_1,\ldots, H_k$ be non-conjugate finite subgroups of $G$ such that any finite subgroup of $G$ is conjugate to some $H_i$. For each $1\leq i\leq k$, let $o_i$ denote the order of $H_i$, set $H_i=\lbrace g_{i,1},\ldots,g_{i,o_i}\rbrace $ and write every $g_{i,j}$ as a word $w_{i,j}(s_1,\ldots,s_n)$. The universal formula $\mathrm{Finite}_G(\bar{x})$ is as follows:\[(\Sigma(\bar{x})=1) \bigwedge_{i=1}^k\bigwedge_{j=1}^{o_i}(w_{i,j}(\bar{x})\neq 1) \wedge \left( \forall g \  \bigwedge_{i=1}^k\bigwedge_{\substack{i'=1\\ i'\neq i \\ o_i=o_{i'}}}^k\bigvee_{j=1}^{o_i}\bigwedge_{j'=1}^{o_{i'}} (gw_{i,j}(\bar{x})g^{-1}\neq w_{j',i'}(\bar{x}))\right).\]\end{proof}

The following lemma is an adaptation of Lemma 5.18 in \cite{Per11}.

\begin{lemme}\label{formula_related}Let $G=\langle s_1,\ldots,s_n \rangle$ be a hyperbolic group and let $H\subset G$ be a subgroup of $G$. There exists an existential formula $\mathrm{Related}_{G,H}(x_1,\ldots,x_n,y_1,\ldots,y_n)$ such that, for any two endomorphisms $\varphi, \psi$ of $G$ defined by $\varphi(s_i)=a_i\in G$ and $\psi(s_i)=b_i\in G$ for $1\leq i\leq n$, $\varphi$ and $\psi$ are $\mathbf{JSJ}_{G,H}$-related if and only if $G\models \mathrm{Related}_{G,H}(a_1,\ldots,a_n,b_1,\ldots,b_n)$.\end{lemme}

\begin{proof}See, for instance, \cite{And18} for details in the non-relative case; again, adapting the proof to the relative case is immediate.\end{proof}

Proposition \ref{main_theorem1} and Corollary \ref{main_theorem1.1} below can be compared with Proposition 6.1 and Corollary 6.3 in \cite{AP24}. The proof of Proposition \ref{main_theorem1} is strongly inspired by the proof of the homogeneity of finitely generated free groups by Perin and Sklinos in \cite{PS12}, and is also based on works of the first author on hyperbolic groups with torsion (see for instance \cite{And20}).

\begin{prop}\label{main_theorem1}
Let $G$ be a torsion-generated  hyperbolic group. Let $u,u'$ be two tuples of elements of $G$ with $\vert u\vert =\vert u'\vert $. Let $G_1,\ldots,G_n$ and $G'_1,\ldots,G'_{n'}$ be the one-ended factors of $G$ relative to $\langle u\rangle$ and $\langle u'\rangle$ respectively, with $\langle u\rangle\subset G_1$ and $\langle u'\rangle\subset G'_1$. Suppose that $u$ and $u'$ have the same $\mathrm{AE}$-type. Then $n=n'$ and there exist two endomorphisms $\varphi$ and $\varphi'$ of $G$ such that the following conditions hold:
\begin{enumerate}[(1)]
    \item $\varphi(u)=u'$ and $\varphi'(u')=u$;
    \item $\varphi$ and $\varphi'$ are injective on the finite subgroups of $G$ and map any two non-conjugate finite subgroups to non-conjugate finite subgroups;
    \item there exist two permutations $\sigma,\sigma'\in S_n$ with $\sigma(1)=\sigma'(1)=1$ such that $\varphi$ induces an isomorphism between $G_i$ and $G'_{\sigma(i)}$ and $\varphi'$ induces an isomorphism between $G'_i$ and $G_{\sigma'(i)}$.
\end{enumerate}
\end{prop}

\begin{proof}The QH one-ended factors of $G$ and $G'$ must be treated separately from the non-QH one-ended factors. After renumbering $G_2,\ldots,G_n$ and $G'_2,\ldots,G'_{n'}$ if necessary, without changing $G_1$ and $G'_1$, one can assume that $G_2,\ldots,G_m$ and $G'_2,\ldots,G'_{m'}$ are the non-QH one-ended factors, with $m\leq n$ and $m'\leq n'$. The first step of the proof consists in proving that there exist two endomorphisms $\varphi$  and $\varphi'$ of $G$ that satisfy the conditions (1) and (2) above, and that satisfy the condition (3') below (which is weaker than condition (3) as it says nothing about the QH one-ended factors):

\begin{enumerate}[(3')]
    \item[(3')] there exist two permutations $\alpha,\alpha'\in S_m,S_{m'}$ with $\alpha(1)=\alpha'(1)=1$ such that $\varphi$ induces an isomorphism between $G_i$ and $G'_{\alpha(i)}$ and $\varphi'$ induces an isomorphism between $G'_i$ and $G_{\alpha'(i)}$. In particular, $m=m'$.
\end{enumerate}

%\smallskip nt The proof of this first step is very similar to the proof of Lemma 5.11 in \cite{And18}, so we will only explain how to adapt the proof and we refer the reader to \cite{And18} for details. A similar argument also appears in the proof of Proposition 6.1 in \cite{AP24}.

\smallskip \noindent
 Since $u$ and $u'$ have the same $\mathrm{AE}$-type, the map $u\rightarrow u'$ extends to an injective morphism $i : \langle u\rangle\rightarrow G$ mapping $u$ to $u'$. In what follows, we denote by $F_1$ the subset of $G_1$ given by Theorem \ref{shortening} (with $\Gamma_1=G_1$, $\Gamma_2=G$, $H=\langle u\rangle$ and $i:H\rightarrow G$ defined above). Thus, for every non-injective morphism $\varphi: G_1\rightarrow G$ such that $\varphi(u)=u'$, there exists a modular automorphism $\sigma\in\mathrm{Mod}_{\langle u\rangle}(G_1)$ relative to $\langle u\rangle$ such that $\varphi\circ\sigma$ kills an element of $F_1$. 

\smallskip \noindent
Similarly, for every $2\leq k\leq m$, we denote by $F_k$ the subset of $G_k$ given by Theorem \ref{shortening} (with $\Gamma_1=G_k$, $\Gamma_2=G$, $H$ the trivial subgroup and $i:H\rightarrow G$ the trivial morphism), so that the following holds: for every $2\leq k\leq m$ and for every non-injective morphism $\varphi: G_k\rightarrow G$, there exists a modular automorphism $\sigma\in\mathrm{Mod}(G_k)$ such that $\varphi\circ\sigma$ kills an element of $F_k$. 

\smallskip \noindent
In what follows, we say that an endomorphism $\varphi$ of $G$ has property $(\ast)$ if $\varphi$ is injective on the finite subgroups of $G$ and maps any two non-conjugate finite subgroups of $G$ to non-conjugate finite subgroups. Suppose towards a contradiction that every endomorphism $\varphi$ of $G$ with property $(\ast)$ and such that $\varphi(u)=u'$ is non-injective on some $G_i$, with $1\leq i\leq m$. Then, for every such morphism, there exist an integer $1\leq i\leq m$ and a (relative) modular automorphism $\sigma$ of $G_i$ such that $\varphi_{\vert G_i}\circ \sigma$ kills an element of $F_i$ (defined in the previous paragraph). By definition of a modular automorphism, $\sigma$ is a conjugation on each finite subgroup of $G_i$, and thus $\sigma$ can be naturally extended to an automorphism of $G$, still denoted by $\sigma$, and the morphism $\psi=\varphi\circ\alpha$ still kills an element of $F_i$. Note that $\varphi$ and $\psi$ are $\mathbf{JSJ}_{G,\langle u\rangle}$-related in the sense of Definition \ref{related}.

\smallskip \noindent
We will now see that the result established in the previous paragraph is expressible via an $\mathrm{AE}$ formula with $u'$ as a parameter. Since $G$ is hyperbolic, it admits a finite presentation $\langle s_1,\ldots,s_n \ \vert \ \Sigma(s_1,\ldots,s_n)=1\rangle$. Write $u$ as a $\vert u\vert$-tuple of words $w(s_1,\ldots,s_n)$ in the generators of $G$, and let $\mathrm{Finite}_G(x_1,\ldots,x_n)$ and $\mathrm{Related}_{G,\langle u\rangle}(x_1,\ldots,x_n,y_1,\ldots,y_n)$ denote the formulas defined in Lemmas \ref{non_conj_to_non_conj} and \ref{formula_related} respectively. For each $1\leq i\leq m$, the set $F_i$ (defined in the previous paragraph) can be written as a collection of words $\lbrace w_{i,1}(s_1,\ldots,s_n),\ldots,w_{i,f_i}(s_1,\ldots,s_n)\rbrace$ where $f_i=\vert F_i\vert$. Now, consider the following $\mathrm{AE}$ formula, where $v$ denotes a $\vert u\vert$-tuple of variables: 

\begin{align*} 
\delta(v) & : \forall\bar{x} \ ((\Sigma(\bar{x})=1) \wedge (v=w(\bar{x})) \ \wedge \ \mathrm{Finite}_G(\bar{x})) \\ 
 & \implies \left(\exists \bar{y} \ \left((\Sigma(\bar{y})=1) \ \wedge \ \mathrm{Related}_{G,\langle u\rangle}(\bar{x},\bar{y}) \wedge \left(\bigvee_{i=1}^m\bigvee_{j=1}^{f_i} w_{i,j}(\bar{y})=1\right)\right)\right).
\end{align*}

\noindent The previous paragraphs tell us that $G$ satisfies $\delta(u')$. Hence, as $u$ and $u'$ have the same $\mathrm{AE}$-type, $G$ satisfies $\delta(u)$ as well, which means that for every class-permuting endomorphism $\varphi$ of $G$ fixing $u$, there exists an endomorphism $\psi$ of $G$ fixing $u$ (up to conjugation), such that $\varphi$ and $\psi$ are $\mathbf{JSJ}_{G,\langle u\rangle}$-related and such that $\psi$ kills an element of some $F_i$ (with $1\leq i\leq m$).

%$\varphi_{\vert G_1}$ and $\psi_{\vert G_1}$ are $\mathbf{JSJ}_{G_1,\langle u\rangle}$-related and $\varphi_{\vert G_i}$ and $\psi_{\vert G_i}$ are $\mathbf{JSJ}_{G_i}$-related for every $2\leq i\leq m$, and such that $\psi$ kills an element of some $F_i$ (with $1\leq i\leq m$).

%we get an endomorphism $\psi$ of $G$ such that $\psi_{\vert G_1}$ is a $\mathbf{JSJ}_{G_1,\langle u\rangle}$-preretraction and $\psi_{\vert G_i}$ is a $\mathbf{JSJ}_{G_i}$-preretraction for every $2\leq i\leq m$, and

\smallskip \noindent
Now, take for $\varphi$ the identity of $G$: we get a $\mathbf{JSJ}_{G_,\langle u\rangle}$-preretraction $\psi : G \rightarrow G$ that is non-injective on some $G_i$ (with $1\leq i\leq m$). Note that $\psi(G_i)$ is contained in a conjugate of $G_i$, therefore one can suppose, after composing $\psi$ by an inner automorphism of $G$ if necessary, that $\psi_{\vert G_i}$ is a non-injective $\mathbf{JSJ}_{G_1,\langle u\rangle}$-preretraction of $G_1$ if $i=1$ or a non-injective $\mathbf{JSJ}_{G_i}$-preretraction of $G_i$ if $i\geq 2$. By Lemma \ref{injective_relative}, $\psi_{\vert G_i}$ is degenerate (recall that this means that there is a QH vertex $v$ (in $\mathbf{JSJ}_{G_1,\langle u\rangle}$ if $i=1$ or $\mathbf{JSJ}_{G_i}$ if $i\geq 2$) such that $G_v$ is not mapped isomorphically to a conjugate of itself by $\psi_{\vert G_i}$).

\smallskip \noindent
Then, let $\mathbf{C}_{G,\langle u\rangle}$ be the centered splitting of $G$ obtained from $\mathbf{JSJ}_{G,\langle u\rangle}$ and from the QH vertex $v$, whose construction is described in Subsection 4.4.2 in \cite{AP24} (the construction remains the same in the relative setting). Using the degenerate $\mathbf{JSJ}_{G_1,\langle u\rangle}$-preretraction or $\mathbf{JSJ}_{G_i}$-preretraction $\psi$, one can define a $\mathbf{C}_{G,\langle u\rangle}$-preretraction that coincides with $\psi$ on $G_v$. This $\mathbf{C}_G$-preretraction is degenerate, which contradicts Lemma \ref{injective2}.

\smallskip \noindent
Hence, there exists an endomorphism $\varphi$ of $G$ with property $(\ast)$, such that $\varphi(u)=u'$, and that is injective on the non-QH one-ended factors of $G$ relative to $\langle u\rangle$, and similarly there exists an endomorphism $\varphi'$ of $G$ with property $(\ast)$, such that $\varphi'(u')=u$, and that is injective on the non-QH one-ended factors of $G$ relative to $\langle u'\rangle$. Note that $\varphi(G_1)$ is contained in $G'_1$ as $\varphi(u)=u'$, and that $\varphi'(G'_1)$ is contained in $G_1$ as $\varphi'(u')=u$. Hence $\varphi'\circ \varphi$ is an injective endomorphism of $G_1$ fixing $u$, therefore by Theorem \ref{relative_co_Hopf} $\varphi'\circ \varphi$ is an automorphism of $G_1$ and thus $\varphi$ and $\varphi'$ induce isomorphisms between $G_1$ and $G'_1$.

\smallskip \noindent
Moreover, exactly as in the proof of \cite[Lemma 5.11]{And18}, we can adapt the argument above so that the morphisms $\varphi$ and $\varphi'$ are not only injective on $G_1,\ldots,G_m$ and $G'_1,\ldots,G'_{m'}$ respectively, but also satisfy the condition (3'): there exist two permutations $\alpha,\alpha'\in S_m,S_{m'}$ with $\alpha(1)=\alpha'(1)=1$ such that $\varphi$ induces an isomorphism between $G_i$ and $G'_{\alpha(i)}$ and $\varphi'$ induces an isomorphism between $G'_i$ and $G_{\alpha'(i)}$ (we refer the reader to Lemma 5.11 in \cite{And18} for details). In particular, $m=m'$.

\smallskip \noindent
It remains to deal with the QH one-ended factors. The rest of the proof is almost identical to the end of the proof of \cite[Proposition 6.1]{AP24}, but we include it for completeness. We will prove that $\varphi$ maps every QH one-ended factor of $G$ relative to $\langle u\rangle$ isomorphically to a QH one-ended factor of $G$ relative to $\langle u'\rangle$, and that $\varphi'\circ\varphi$ and $\varphi\circ\varphi'$ induce permutations of the conjugacy classes of QH one-ended factors of $G$ relative to $\langle u\rangle$ and $\langle u'\rangle$ respectively.

\smallskip \noindent
We denote by $\mathbf{S}_{G,\langle u\rangle}$ and $\mathbf{S}_{G,\langle u'\rangle}$ two Stallings decompositions of $G$ relative to $\langle u\rangle$ and $\langle u'\rangle$ respectively. For a conical hyperbolic orbifold $\mathcal{O}$, we denote by $\chi(\mathcal{O})$ the opposite of the orbifold Euler characteristic of $\mathcal{O}$, and if $E$ is an extension of $\pi_1^{\mathrm{orb}}(\mathcal{O})$ by a finite group $F$, we define $\chi(E)=\chi(\mathcal{O})/\vert F\vert$, called the \emph{complexity} of $G$. Let $c$ (respectively $c'$) be the smallest complexity of a QH factor of $\mathbf{S}_{G,\langle u\rangle}$ (respectively $\mathbf{S}_{G,\langle u'\rangle}$) in the sense of \cite[Definition 3.4]{AP24}. Suppose without loss of generality that $c\leq c'$ and let $v$ be a vertex of $\mathbf{S}_{G,\langle u\rangle}$ such that $G_v$ is a QH group of complexity $c$. As $G$ has only a finite number of conjugacy classes of finite subgroups, and since $\varphi'\circ\varphi$ maps two non-conjugate finite subgroups to non-conjugate subgroups, there exists an integer $N\geq 1$ such that the endomorphism $p:=(\varphi'\circ\varphi)^N$ of $G$ coincides with an inner automorphism on each finite subgroup of $G$, and thus on each conical subgroup of $G_v$.

\smallskip \noindent
Note that $G_v$ has at least one conical point (indeed, by \cite[Lemma 5.9]{AP24}, since $G$ is torsion-generated by assumption, the underlying orbifold of $G_v$ has genus $0$, and moreover its boundary is empty because $G_v$ is one-ended), thus the construction described in \cite[Subsection 4.4.2]{AP24} applies and produces a centered splitting $\mathbf{C}_{G,\langle u\rangle}$ of $G$ relative to $\langle u\rangle$. We can define a $C_{G,\langle u\rangle}$-preretraction $q$ that coincides with $p$ on $G_v$. By Lemma \ref{injective2}, $q$ is non-degenerate, which means that it maps $G_v$ isomorphically to a conjugate of $G_v$, and therefore $p$ maps $G_v$ isomorphically to a conjugate of $G_v$. In particular $p$ is non-pinching on $G_v$, and thus $\varphi$ is non-pinching on $G_v$. It follows that $\varphi(G_v)$ is contained in a conjugate of some vertex group $G'_w$ of $\mathbf{S}_{G,\langle u'\rangle}$ (by \cite[Proposition 2.31]{And18}). Clearly, this vertex group is QH, otherwise $p(G_v)$ would be contained in a non-QH vertex group of $G$ relative to $\langle u\rangle$, contradicting the fact that $p(G_v)$ is a conjugate of $G_v$. But the complexity of $G_v$ is minimal among the QH vertex groups of $G$ relative to $\langle u\rangle$ or $\langle u'\rangle$ so $\chi(G'_w)\geq \chi(G_v)$, but $\chi(G_v)\geq \chi(G'_w)$ by \cite[Lemma 3.5]{AP24}, so $\chi(G_v)=\chi(G'_w)$ and thus $\varphi$ induces an isomorphism between $G_v$ and $G'_w$ (again by \cite[Lemma 3.5]{AP24}). Then, we can repeat the same process with the smallest complexity $>c$, and so on.\end{proof}

Recall that a graph of groups $\Delta$ is said to be \emph{reduced} if, for any edge of $\Delta$ with distinct endpoints, the edge group is strictly contained in the vertex groups.

We deduce the following corollary from Proposition \ref{main_theorem1} in the exact same way as we deduced Corollary 6.3 from Proposition 6.1 in \cite{AP24}, and we refer the reader to \cite{AP24} for details. The only difference between Proposition \ref{main_theorem1} and Corollary \ref{main_theorem1.1} is that condition (2) on finite subgroups in Proposition \ref{main_theorem1} is replaced with a condition on the vertex groups of reduced Stallings splittings in Corollary \ref{main_theorem1.1}. 

\begin{co}\label{main_theorem1.1}
Let $G$ be a torsion-generated  hyperbolic group. Let $u,u'$ be two tuples of elements of $G$ with $\vert u\vert =\vert u'\vert $. Let $\mathrm{S}_{G,\langle u\rangle}$ and $\mathrm{S}_{G,\langle u'\rangle}$ be reduced Stallings splittings of $G$ relative to $\langle u\rangle$ and $\langle u'\rangle$ respectively. Suppose that $u$ and $u'$ have the same $\mathrm{AE}$-type. Then there exist two endomorphisms $\varphi,\varphi'$ of $G$ such that the following conditions holds:
\begin{enumerate}[(1)]
    \item $\varphi(u)=u'$ and $\varphi'(u')=u$;
    \item $\varphi$ maps each vertex group of $\mathrm{S}_{G,\langle u\rangle}$ isomorphically to a vertex group of $\mathrm{S}_{G,\langle u'\rangle}$ and $\varphi'$ maps each vertex group of $\mathrm{S}_{G,\langle u'\rangle}$ isomorphically to a vertex group of $\mathrm{S}_{G,\langle u\rangle}$;
    \item $\varphi$ and $\varphi'$ induce one-to-one correspondences between the conjugacy classes of vertex groups of $\mathrm{S}_{G,\langle u\rangle}$ and $\mathrm{S}_{G,\langle u'\rangle}$.
\end{enumerate}
\end{co}

\subsection{Homogeneity in torsion-generated hyperbolic groups}

\begin{te}\label{main_theorem}
Let $G$ be a torsion-generated hyperbolic group. Suppose that the following condition holds: for every edge group $F$ of a reduced Stallings splitting of $G$, the image of the natural map $N_G(F)\rightarrow \mathrm{Aut}(F)$ is equal to $\mathrm{Inn}(F)$. Then $G$ is $\mathrm{AE}$-homogeneous.
\end{te}

Recall that a Coxeter presentation $\langle s_1,\ldots,s_n \ \vert \ (s_is_j)^{m_{ij}}=1, \ \text{for all} \ i,j \rangle$ is even if $m_{ij}$ is \emph{even} or $\infty$ for every $1\leq i,j\leq n$, in which case it can be proved that this presentation is the only Coxeter presentation for the Coxeter group it defines, and therefore it makes sense to say that $G$ is an even Coxeter group. Hyperbolic Coxeter groups form an important subclass of Coxeter groups. Moussong proved in his PhD thesis that a Coxeter group is hyperbolic if and only if it does not contain $\mathbb{Z}^2$, if and only if there is no pair of disjoint subsets $T_1,T_2\subseteq T$ such that $\langle T_1\rangle$ and $\langle T_2\rangle$ commute and are infinite, and there is no subset $T\subseteq S$ such that $(\langle T\rangle,T)$ is an affine Coxeter system of rank $\geq 3$ (note that affine Coxeter systems have been completely classified, and that affine Coxeter groups are virtually abelian).

%A Coxeter system $(W,S)$ is said to be \emph{2-spherical} if $m_{ij}$ is finite for every $1\leq i,j\leq n$. Note that such a Coxeter group $W$ has property (FA) of Serre (and therefore $W$ is one-ended) because it is generated by a set $S$ composed of involutions such that, for every $s_1,s_2\in S$, $s_1s_2$ has finite order. 

\begin{co}\label{corollary}
The following groups are $\mathrm{AE}$-homogeneous:
\begin{enumerate}[(1)]
    \item[$\bullet$]  hyperbolic even Coxeter groups are homogeneous;
    \item[$\bullet$] torsion-generated hyperbolic one-ended groups.
\end{enumerate}
\end{co}

In the next section \ref{example}, we will give an example of a hyperbolic Coxeter group that is not $\mathrm{AE}$-homogeneous (and in fact not $\mathrm{EAE}$-homogeneous), and we conjecture that this example is not homogeneous in the absolute sense.

%Note : regarder ce que font Dente-Byron et Perin dans le cas hyperbolique sans torsion et voir si l'existence d'une surface de genre 0 peut être une obstruction à l'homogénéité dans leur caractérisation (si tel est le cas, c'est mauvais signe...).

Before proving Theorem \ref{main_theorem}, let us deduce Corollary \ref{corollary} from Theorem \ref{main_theorem}. We deduce this corollary in exactly the same way as we deduced Corollary 8.2 from Theorem 8.1 in \cite{AP24}, but we still include a proof of the corollary below for the convenience of the reader.

\begin{proof}[Proof of Corollary \ref{corollary}]
The second point is an immediate consequence of Theorem \ref{main_theorem} since a reduced Stallings splitting of a one-ended group is simply a point, so the condition on the edge groups is empty. 

\smallskip \noindent
Let us prove the first point. By \cite[Proposition 8.8.2]{davis}, an edge group $F$ in a Stallings splitting of a Coxeter group $G=\langle S\rangle$ is a special finite subgroup, which means that there exists a subset $T\subseteq S$ such that $F=\langle T\rangle$. As observed by Bahls in \cite[Proposition 5.1]{Bahls}, one can define a retraction $\rho : G\rightarrow F$ by $\rho(s)=s$ if $s\in T$ and $\rho(s)=1$ otherwise (this morphism is well-defined because every defining relation in $G$ is of the form $(ss')^m=1$ with $s,s'\in S$ and $m$ even). Now, for $g\in N_G(F)$ and for every $h\in F$, we have $ghg^{-1}=\rho(g)h\rho(g)^{-1}$, which shows that the image of the natural map $N_G(F)\rightarrow \mathrm{Aut}(F)$ is contained in $\mathrm{Inn}(F)$. Hence Theorem \ref{main_theorem} applies.
\end{proof}

%\begin{rk}Note that the existence of the retraction $\rho : G\rightarrow F$ is not true if $G$ is not assumed to be even. For instance, the finite dihedral group $G=D_3=\langle a,b \ \vert \ a^2=b^2=(ab)^3=1\rangle$ does not retract onto $\langle a \rangle$.\end{rk}

\begin{proof}[Proof of Theorem \ref{main_theorem}]Let $u,u'$ be two tuples of elements of $G$, and suppose that $u,u'$ have the same $\mathrm{AE}$-type. Let $\mathrm{S}_{G,\langle u\rangle}$ and $\mathrm{S}_{G,\langle u'\rangle}$ be reduced Stallings splittings of $G$ relative to $\langle u\rangle$ and $\langle u'\rangle$ respectively. Let $\varphi,\varphi'$ denote the endomorphisms of $G$ given by Corollary \ref{main_theorem1.1}. By \cite[Proposition 8.15]{AP24}, there exist an automorphism $\psi$ of $G$ and an integer $m\geq 0$ such that $\psi$ coincides with $\varphi\circ (\varphi'\circ \varphi)^m$ up to conjugation on each vertex group of $\mathrm{S}_{G,\langle u\rangle}$. Since $\varphi(u)=u'$ and $\varphi'(u')=u$, we have $\psi(u)=u'$. Hence, $G$ is $\mathrm{AE}$-homogeneous.\end{proof}

\subsection{Torsion-generated hyperbolic groups are strictly minimal}\label{strictly_minimal}

\begin{de}\label{strictly_minimal_def}A group is said to be strictly minimal if it has no proper elementarily embedded subgroup.\end{de}

\begin{te}\label{strictly_minimal_theorem}Every torsion-generated hyperbolic group $G$ is strictly minimal. In fact, $G$ has no proper $\mathrm{AE}$-embedded subgroup.\end{te}

\begin{proof}Let $H$ be an $\mathrm{AE}$-embedded subgroup of $G$. If $G$ is finite the result is obvious, and if $G$ is virtually cyclic infinite the result is not much more difficult, so let us suppose that $G$ is non-elementary. Hence, $H$ is non-elementary as well. 

\smallskip \noindent
First, observe that every finite subgroup of $G$ is conjugate to a subgroup of $H$. Indeed, it is not hard to see that $H$ and $G$ have the same number of conjugacy classes of finite subgroups since they are $\mathrm{AE}$-equivalent. Moreover, if two finite subgroups $A = \lbrace a_1,\cdots, a_m\rbrace$ and $B = \lbrace b_1,\cdots, b_m\rbrace$ of $H$ are not conjugate in $H$, then $H$ satisfies a universal formula $\theta(a_1,\cdots, a_m, b_1,\cdots, b_m)$ expressing the fact, for every $h\in H$, there is an integer $1\leq i\leq m$ such that for every $1\leq i\leq m$, $a_i\neq hb_jh^{-1}$. Since $H$ is $\mathrm{AE}$-embedded in $G$, this formula is true in $G$ as well, therefore $A$ and $B$ are not conjugate in $G$.

\smallskip \noindent
Let us prove that $H$ is a one-ended factor of $G$, that is a one-ended vertex group in a Stallings splitting of $G$. Suppose toward contradiction that $H$ is not a one-ended factor of $G$. Then, by (relative versions of) Lemmas 3.5 and 3.7 in \cite{And21b}, there exist a centered splitting $\mathbf{C}_{G,H}$ of $G$ relative to $H$ and a degenerate $\mathbf{C}_{G,H}$-preretraction, which contradicts Lemma \ref{injective2} (please note that the definition of a degenerate $\mathbf{C}_{G,H}$-preretraction \cite[Definition 2.27]{And21b} is the inverse of Definition \ref{pre_centered} in this paper (which is consistent with Definition 3.6 in \cite{And18})). Hence, $H$ is a one-ended factor of $G$. 

\smallskip \noindent
Then, suppose towards a contradiction that $G$ admits a centered splitting $\mathbf{C}_{G,H}$ relative to $H$. Since $G$ is torsion-generated, by \cite[Lemma 5.9]{AP24} the underlying graph of $\mathbf{C}_{G,H}$ is a tree and the underlying orbifold of the central vertex group $G_v$ of $\mathbf{C}_{G,H}$ is orientable of genus 0, therefore it has at least three conical points or boundary components. Let $w$ denote the unique vertex of the Bass-Serre tree of $\mathbf{C}_{G,H}$ that is fixed by $H$ (uniqueness follows from the fact that $H$ is non-elementary whereas edges of $\mathbf{C}_{G,H}$ are virtually cyclic (possibly finite)). By definition of a centered splitting relative to $H$, this vertex $w$ is not in the orbit of $v$. After replacing $H$ with a conjugate if necessary, we can assume that $w$ is a vertex of $\mathbf{C}_{G,H}$ adjacent to the central vertex $v$, and let $e$ denote the edge $[v,w]$. Now, let $K$ be an extended boundary or conical subgroup of $G_v$ that is not conjugate to $G_e$, and let $k\in K$ be an element that is not in the fiber of $G_e$. Then $k$ cannot be written as a product of conjugates of elements of $H$ (which is contained in $G_w$), contradicting the fact that $G$ is torsion-generated and that every finite subgroup of $G$ is conjugate to a subgroup of $H$. Hence, $G$ does not admit a centered splitting relative to $H$. It follows that $G$ is a quasiprototype relative to $H$ (see Definition 5.9 in \cite{And18}) and that $G$ is its own quasicore (see Definition 5.11 in \cite{And18}).

\smallskip \noindent
Finally, by Proposition 6.9 and Lemma 6.11 in \cite{And18}, since $G$ and $H$ are their own quasicores relative to $H$, the one-ended factors of $G$ that are not conjugate to $H$ are finite-by-orbifold groups, and the underlying orbifolds are orientable of genus 0 (because $G$ is torsion-generated). Therefore, every one-ended factor of $G$ that is not conjugate to $H$ has at least one (in fact, two) extended conical subgroups that are not conjugate to $H$, which contradicts the fact that every finite subgroup of $G$ is conjugate to a subgroup of $H$. For the same reason, there is no zero-ended (in other words, finite) factor. Hence any reduced Stallings splitting of $G$ relative to $H$ is reduced to a point, and since $H$ is a one-ended factor we obtain $G=H$.\end{proof}

\subsection{A non-homogeneous hyperbolic Coxeter groups}\label{example}

\subsubsection{Strategy of proof}

We will construct a (virtually free) hyperbolic Coxeter group $G$ that is not $\mathrm{EAE}$-homogeneous (this is comparable to the virtually free group constructed in the sections 5 and 6 of \cite{And18b} by the first author, but performing the construction among Coxeter groups adds new constraints). As explained in the introduction, the proof consists in constructing an \emph{$\mathrm{EAE}$-extension} $G'$ of $G$ (see Definition \ref{EAE_def}) and two elements of $G$ that are automorphic in $G'$ but not in $G$. Hence, these elements have the same type in $G'$ and the same $\mathrm{EAE}$-type in $G$, which proves that $G$ is not EAE-homogeneous.

\begin{de}\label{EAE_def}
Let $G'$ be a group and let $G$ be a subgroup of $G'$. We say that $G'$ is an \emph{$\mathrm{EAE}$-extension} of $G$, or that $G$ is \emph{$\mathrm{EAE}$-embedded} in $G'$, if the following condition holds: for every $\mathrm{EAE}$-formula $\phi(x)$ and finite tuple $u\in G^{\vert x\vert}$, if $\phi(u)$ is true in $G$ then $\phi(u)$ is true in $G'$ (note in particular that if $\phi(x)$ is $\mathrm{AE}$, then $\phi(u)$ is true in $G$ if and only if it is true in $G'$).\end{de}

It is worth recalling that, as already mentioned in the introduction, Sela developed a quantifier elimination procedure down to the Boolean algebra of $\mathrm{AE}$-definable sets for torsion-free hyperbolic groups. More precisely, for any fixed torsion-free hyperbolic group $G$ and for every definable set $D(x)$ in the language of groups (where $x$ denotes a finite tuple of free variables), there is a definable set $D'(x)$ that belongs to the Boolean algebra of $\mathrm{AE}$-definable sets, such that \[\lbrace u\in G^{\vert x\vert} \ \vert \ G\models D(u)\rbrace=\lbrace u\in G^{\vert x\vert} \ \vert \ G\models D'(u)\rbrace.\]
In particular, two finite tuples of elements of $G$ have the same type if and only if they have the same $\mathrm{AE}$-type. We refer the reader to \cite{Sel09} and the previous papers in the series of papers on the Tarski problem for further details. It seems reasonable to conjecture that a similar phenomenon persists in the presence of torsion (even though not much has been proved in this direction yet). If this is indeed the case, then the group $G$ constructed in the present section is not homogeneous (not only not $\mathrm{EAE}$-homogeneous).

\subsubsection{Definition of the group}

Consider the finite Coxeter groups $A,B,C$ given by the following diagrams (recall that the vertices represent the elements of a Coxeter system $S$, and that there is no edge between two vertices if and only if the corresponding generators $s_i,s_j$ commute, an edge with no label if and only if $(s_is_j)^3=1$, and an edge with label $n\geq 4$ if and only if $(s_is_j)^n=1$):

\begin{figure}[h!]
  \centering
    \begin{tikzpicture}
\draw[black, very thick] (2,0) -- (6,0);
\draw[black, very thick] (8,0) -- (12,0);
\node[text=black] at (2,-0.5) {$e_1$};
\node[text=black] at (4,-0.5) {$e_2$};
\node[text=black] at (8,-0.5) {$e_3$};
\node[text=black] at (10,-0.5) {$e_4$};
\node[text=black] at (2,0.5) {$a_1$};
\node[text=black] at (3,0.5) {$a_2$};
\node[text=black] at (4,0.5) {$a_3$};
\node[text=black] at (5,0.5) {$a_4$};
\node[text=black] at (6,0.5) {$a_5$};
\node[text=black] at (8,0.5) {$a_6$};
\node[text=black] at (9,0.5) {$a_7$};
\node[text=black] at (10,0.5) {$a_8$};
\node[text=black] at (11,0.5) {$a_9$};
\node[text=black] at (12,0.5) {$a_{10}$};
\fill[black] (2,0) circle (0.1cm);
\fill[black] (3,0) circle (0.1cm);
\fill[black] (4,0) circle (0.1cm);
\fill[black] (5,0) circle (0.1cm);
\fill[black] (6,0) circle (0.1cm);
\fill[black] (8,0) circle (0.1cm);
\fill[black] (9,0) circle (0.1cm);
\fill[black] (10,0) circle (0.1cm);
\fill[black] (11,0) circle (0.1cm);
\fill[black] (12,0) circle (0.1cm);
\end{tikzpicture}
  \caption{The finite group $A$, isomorphic to $S_6\times S_6$.}
  \label{A}
\end{figure}

\begin{figure}[h!]
  \centering
    \begin{tikzpicture}
\draw[black, very thick] (2,0) -- (4,0);
\draw[black, very thick] (6,0) -- (7,0);
\draw[black, very thick] (9,0) -- (12,0);
\node[text=black] at (2,-0.5) {$e_2$};
\node[text=black] at (4,-0.5) {$e_3$};
\node[text=black] at (6,-0.5) {$e_1$};
\node[text=black] at (9,-0.5) {$e_4$};
\node[text=black] at (2,0.5) {$b_1$};
\node[text=black] at (3,0.5) {$b_2$};
\node[text=black] at (4,0.5) {$b_3$};
\node[text=black] at (6,0.5) {$b_4$};
\node[text=black] at (7,0.5) {$b_5$};
\node[text=black] at (9,0.5) {$b_6$};
\node[text=black] at (10,0.5) {$b_7$};
\node[text=black] at (11,0.5) {$b_8$};
\node[text=black] at (12,0.5) {$b_9$};
\fill[black] (2,0) circle (0.1cm);
\fill[black] (3,0) circle (0.1cm);
\fill[black] (4,0) circle (0.1cm);
\fill[black] (6,0) circle (0.1cm);
\fill[black] (7,0) circle (0.1cm);
\fill[black] (9,0) circle (0.1cm);
\fill[black] (10,0) circle (0.1cm);
\fill[black] (11,0) circle (0.1cm);
\fill[black] (12,0) circle (0.1cm);
\end{tikzpicture}
  \caption{The finite group $B$, isomorphic to $S_4\times S_3\times S_5$.}
  \label{BB}
\end{figure}

\begin{figure}[h!]
  \centering
    \begin{tikzpicture}
\fill[black] (0,0) circle (0.1cm);
\fill[black] (1,0) circle (0.1cm);
\fill[black] (2,0) circle (0.1cm);
\fill[black] (3,0) circle (0.1cm);
\node[text=black] at (0,-0.5) {$e_1$};
\node[text=black] at (1,-0.5) {$e_2$};
\node[text=black] at (2,-0.5) {$e_3$};
\node[text=black] at (3,-0.5) {$e_4$};
\end{tikzpicture}
  \caption{The finite group $C$, isomorphic to $(\mathbb{Z}/2\mathbb{Z})^4$.}
  \label{CC}
\end{figure}

\noindent Let $S_A=\lbrace a_1,\ldots,a_{10}\rbrace$, $S_B = \lbrace b_1,\ldots,b_{9}\rbrace$ and $\lbrace e_1,\ldots,e_4\rbrace$ be the Coxeter generating sets of $A,B,C$ respectively shown in Figures \ref{A},\ref{BB},\ref{CC}, respectively. These figures illustrate two embeddings $i:C\rightarrow A:e_1\mapsto a_1,e_2\mapsto a_3, e_3\mapsto a_6,e_4\mapsto a_8$ and $j:C\rightarrow B:e_1\mapsto b_4,e_2\mapsto b_1, e_3\mapsto b_3,e_4\mapsto b_6$. Using these two embeddings, we define a group $G=A\ast_C B$. First, let us briefly explain why $G$ is a Coxeter group. Let  $A=\langle a_1,\ldots,a_m \ \vert \ (a_ia_j)^{r_{ij}}=1\rangle$ and $B=\langle b_1,\ldots,b_n \ \vert \ (b_ib_j)^{s_{ij}}=1\rangle$ be Coxeter presentations for $A$ and $B$ corresponding to the generating sets $S_A,S_B$. Define $I=\lbrace 1,\ldots ,10\rbrace$, $J=\lbrace 2,5,7,8,9\rbrace$, $K=\lbrace 1,3,4,6\rbrace$, and $\sigma : K \rightarrow I$ such that $\sigma(1)=3$, $\sigma(3)=6$, $\sigma(4)=1$, $\sigma(6)=8$, so that the relation $b_k=a_{\sigma(k)}$ holds in $G$, for every $k\in K$. The group $G$ admits the following presentation:

%\[\langle a_1,\ldots,a_{10},b_2,b_5, b_7\ldots,b_9 \ \vert \ (a_ia_j)^{r_{ij}}=1 \ \forall (i,j)\in I^2, \ (b_ib_j)^{s_{ij}}=1 \ \forall (i,j)\in J^2, \ (a_ib_j)^{s_{ij}}=1 \ \forall (i,j)\in I\times K \rangle\]

\[
\left\langle
\begin{aligned}
  & a_1,\ldots,a_{10},\, b_2,b_5,b_7,\ldots,b_9 \ \vert \  (a_i a_j)^{r_{ij}}=1 \ \forall (i,j)\in I^2, \ (b_i b_j)^{s_{ij}}=1 \ \forall (i,j)\in J^2,\\
  & \ \ \ \ \ \ \ \ \ \ \ \ \ \ \ \ \ \ \ \ \ \ \ \ \ \ \ \ \ \ \ \ \ \ \ \ \ \ (b_i a_{\sigma(j)})^{s_{ij}}=1 \ \forall (i,j)\in J\times K, \\
  & \ \ \ \ \ \ \ \ \ \ \ \ \ \ \ \ \ \ \ \ \ \ \ \ \ \ \ \ \ \ \ \ \ \ \ \ \ \ (a_{\sigma(i)} b_j)^{s_{ij}}=1 \ \forall (i,j)\in K\times J
\end{aligned}
\right\rangle.
\]
This is clearly a Coxeter presentation, so $G$ is a Coxeter group. Moreover, as $A$ and $B$ are finite, $A\ast_C B$ is a virtually free group.

\smallskip \noindent
Write $A=A_1\times A_2$ with $e_1,e_2\in A_1$ and $e_3,e_4\in A_2$. Identifying $A_1$ and $A_2$ with $S_6$, write $e_1=(1 \ 2)\in A_1$, $e_2=(3 \ 4)\in A_1$, $e_3=(1 \ 2)\in A_2$, $e_4=(3 \ 4)\in A_2$. Then, consider the elements $x=(1 \ 3)(2 \ 4)(5 \ 6)\in A_1$ and $y=(1 \ 3)(2 \ 4)(5 \ 6)\in A_2$. We will prove that the obvious automorphism $\sigma$ of $A$ that exchanges $x$ and $y$ (that is the automorphism exchanging the two direct factors $A_1$ and $A_2$ in the obvious way) extends to an automorphism of an $\mathrm{EAE}$-extension $G'$ of $G$ (given by Theorem \ref{EAE} below), proving that $x$ and $y$ have the same $\mathrm{EAE}$-type. But we will prove that $x$ to $y$ are not automorphic in $G$. Hence, this will prove that the (odd) Coxeter group $G$ is not $\mathrm{EAE}$-homogeneous. Note that, contrary to the case of even Coxeter groups, there is no retraction from $A$ and $B$ onto the edge group $C$ here.

\subsubsection{An endomorphism of $G$ exchanging $x$ and $y$}\label{endo}

Let $\sigma$ be the obvious automorphism of $A$ that exchanges $x$ and $y$. Note that $\sigma(C)=C$. In fact, it is clear that $\sigma$ preserves the generating set $E=\lbrace e_1,e_2,e_3,e_4\rbrace$ of $C$, and that it induces the permutation $(e_1 \ e_3)(e_2 \ e_4)$.

\smallskip \noindent
Write $B=B_1\times B_2\times B_3$ with $e_2,e_3\in B_1$, $e_1\in B_2$ and $e_4\in B_3$. Identifying $B_1$ with $S_4$, write $e_2=(1 \ 2)\in B_1$ and $e_3=(3 \ 4)\in B_1$. An immediate calculation shows that the element $b=(1 \ 3)(2 \ 4)\in B_1$ satisfies $be_2b^{-1}=e_3$ and $be_3b^{-1}=e_2$. Moreover, $b$ commutes with $e_1$ and $e_4$ as these elements belong to $B_2$ and $B_3$ respectively. Hence, the element $b$ induces by conjugation the permutation $(e_2 \ e_3)$ of $E$. 

\smallskip \noindent
Then, observe that the elements $e_1=(1 \ 2), \ e_2=(3 \ 4), \ x=(1 \ 3)(2 \ 4)(5 \ 6)\in A_1$ satisfy the relations $xe_1x^{-1}=e_2$ and $xe_2x^{-1}=e_1$. Moreover, $x$ commutes with $e_3$ and $e_4$ as these elements belong to $A_2$. It follows that $x$ induces by conjugation the permutation $(e_1 \ e_2)$ of the set $E$. A similar argument shows that $y$ induces by conjugation the permutation $(e_3 \ e_4)$ of $E$.

\smallskip \noindent
Define $u=bxyb\in BAB$. One can easily verify that the restriction to $C$ of the inner automorphism $\mathrm{ad}(u)\in\mathrm{Inn}(G)$ coincides with $\sigma_{\vert C}$, therefore the morphism $\bar{\sigma}: G\rightarrow G$ defined by $\bar{\sigma}_{\vert A}=\sigma$ and $\bar{\sigma}_{\vert B}=\mathrm{ad}(u)_{\vert B}$ is well-defined.

%because $(e_2 \ e_3)(e_1 \ e_2)(e_3 \ e_4)(e_2 \ e_3)=(e_1 \ e_3)(e_2 \ e_4)$

\smallskip \noindent
Note that $\bar{\sigma}$ is not an automorphism of $G$ (and we will prove below that $x$ and $y$ are not automorphic in $G$). It is indeed not hard to see that $B$ is not contained in the image of $\bar{\sigma}$. On the other hand, it is worth observing (even though this will not be used in the rest of the proof) that $\bar{\sigma}$ is injective as it maps any element of $G$ written in non-trivial reduced normal form (with respect to the splitting $A\ast_C B$ of $G$) to an element written in non-trivial reduced normal form. Since $\bar{\sigma}$ swaps $x$ and $y$, this shows that $x$ and $y$ have the same existential type. 

\subsubsection{The elements $x$ and $y$ have the same $\mathrm{EAE}$-type}

In order to prove that $x$ and $y$ have the same $\mathrm{EAE}$-type, we will consider an $\mathrm{EAE}$-extension $G'$ of $G$ (given by Theorem \ref{EAE} below) and prove that the endomorphism $\bar{\sigma}$ of $G$ defined in the previous paragraph can be modified in order to get an automorphism of $G'$ that maps $x$ to $y$. Recall that a hyperbolic group is either finite or virtually cyclic infinite or contains a non-abelian free group, and in the latter case we say that this hyperbolic group is non-elementary. Recall also that, given a non-elementary subgroup $H$ of a hyperbolic group, there exists a unique maximal finite subgroup of the hyperbolic group that is normalized by $H$. We will need the following theorem.

\begin{te}[{\cite[Theorem 1.10]{And19a}}]\label{EAE}Let $G$ be a non-elementary hyperbolic group, and let $C$ be a finite subgroup of $G$. Suppose that $N_G(C)$ is non-elementary and that $C$ is the maximal finite subgroup of $G$ normalized by $N_G(C)$. Then the inclusion of $G$ into the group \[G'=\langle G,t \ \vert \ tg=gt, \ \forall g\in C\rangle\] is an $\mathrm{EAE}$-embedding (see Definition \ref{EAE_def}).\end{te}

\begin{rk}
Note that the group $G'$ is simply the HNN extension of $G$ over the identity of $C$. We conjecture that $G$ is elementarily embedded into $G'$, but the proof of this conjecture would require a quantifier elimination procedure which is currently only known for torsion-free hyperbolic groups, by the work of Sela \cite{Sel09}.
\end{rk}

For the Coxeter group $G$ and its subgroup $C$ defined at the beginning of Subsection \ref{example}, we want to prove that the assumptions of Theorem \ref{EAE} are verified. Let $C'$ denote the maximal finite subgroup of $G$ normalized by the normalizer $N_G(C)$ of $C$ in $G$ (as already recalled above, such a maximal finite subgroup always exists in a hyperbolic group). Clearly, $C$ is contained in $C'$, and we need to prove that $C$ is equal to  $C'$. Note that $G$ acts transitively on the edges of the Bass-Serre tree $T$ of the splitting $A\ast_C B$, so $N_G(C)$ acts transitively on the edges of $\mathrm{Fix}_T(C)$. But $G$ does not act transitively on the vertices of $T$, so $N_G(C)$ does not act transitively on the vertices of $\mathrm{Fix}_T(C)$. It follows that $\mathrm{Fix}_T(C)/N_G(C)$ is simply an edge, and thus that $N_G(C)=N_A(C)\ast_C N_B(C)$. 

\smallskip \noindent
Let $v_A$ and $v_B$ be the unique vertices of $T$ fixed by $A$ and $B$ respectively, which are also the unique vertices fixed by $x$ and $y$ respectively. Suppose towards a contradiction that $C$ is a proper subgroup of $C'$. This group $C'$ being finite, it fixes a vertex $v$ of $T$. Note that $x$ belongs to $N_A(C)$, so $x$ normalizes $C'$ and it follows that $C'$ fixes $xv$. Suppose that $xv\neq v$, then $C'$ is contained in the stabilizer of the path $[v,xv]$ whose order is $\leq \vert C\vert$ (because all the edge groups of $T$ are conjugates of $C$), contradicting our assumption that $C$ is strictly contained in $C'$. It follows that $xv=v$. But as $y$ belongs to $N_B(C)$, the same argument shows that $yv=v$. This is a contradiction because there is no vertex in $T$ that is fixed by both $x$ and $y$. Hence $C'=C$, and so Theorem \ref{EAE} can be applied to our Coxeter group $G$ and its finite subgroup $C$.

\smallskip \noindent
Recall that $u$ denotes the element $bxyb$ defined in the previous subsection, and define \[G'= \langle G, t \ \vert \ \mathrm{ad}(t)_{\vert C}=\mathrm{ad}(u)_{\vert C}\rangle=\langle G,t' \ \vert \ \mathrm{ad}(t')_{\vert C}=\mathrm{id}_C\rangle.\] The last equality is obtained by taking $t'=u^{-1}t$. By Theorem \ref{EAE}, the inclusion of $G$ into $G'$ is an $\mathrm{EAE}$-elementary embedding. 

\smallskip \noindent
We will prove that $x$ and $y$ are automorphic in $G'$. The idea consists in modifying the endomorphism $\bar{\sigma}$ of $G$ defined in the paragraph \ref{endo}. We define a morphism $\theta : G \rightarrow G'$ by $\theta_{\vert A}=\sigma$ and $\theta_{\vert B}=\mathrm{ad}(t)_{\vert B}$ (so we replace the $u$ in the definition of $\bar{\sigma}$ with the new letter $t$). This morphism is well-defined since $\sigma$ and $\mathrm{ad}(t)$ coincide (with $\mathrm{ad}(u)$) on $C$. Moreover, recall that $\sigma$ swaps $x$ and $y$, therefore $\theta$ swaps $x$ and $y$ as well. 

\smallskip \noindent
We will prove that this morphism $\theta : G \rightarrow G'$ extends to an endomorphism of $G'$. First, let us prove that $\theta(u)$ and $u$ induce the same action on $C$ by conjugation. Note that $\theta(u)=\theta(bxyb)=tbt^{-1}yxtbt^{-1}$, so $\theta(u)$ acts on $C$ by conjugation in the same way as $ubu^{-1}yxubu^{-1}$ (since $\mathrm{ad}(t)_{\vert C}=\mathrm{ad}(u)_{\vert C}$). But each of the elements $u,b,x,y$ preserves the set $E=\lbrace e_1,e_2,e_3,e_4\rbrace$, so $ubu^{-1}yxubu^{-1}$ induces a permutation of $E$. This permutation is the same as the one induced by $bxybxybxybxybxyb=(bxy)^5b$ (using the fact that $u=bxyb$, $\mathrm{ad}(u)_{\vert C}=\mathrm{ad}(u^{-1})_{\vert C}$ and $xy=yx$). Then, note that $bxy$ acts on $E$ as the 4-cycle $(e_3 \ e_4 \ e_2 \ e_1)$, hence $(bxy)^5$ acts on $E$ as $bxy$. It follows that $\theta(u)$ acts on $C$ as $bxyb=u$, which allows us to define $\theta(t)=t$ (indeed, the defining relation of the HNN extension, namely $ucu^{-1}=tct^{-1}$ for $c\in C$, is preserved by $\theta$ since we have $\theta(ucu^{-1})=\theta(u)\theta(c)\theta(u)^{-1}=u\theta(c)u^{-1}=t\theta(c)t^{-1}=\theta(tct^{-1})$ for every $c\in C$).

\smallskip \noindent
Then, observe that $\theta$ is surjective since $\theta(A)=A$, $\theta(B)=tBt^{-1}$ and $\theta(t)=t$, and $A\cup B\cup\lbrace t\rbrace$ is a generating set of $G'$. But $G'$ is a virtually free group, so it is Hopfian, therefore $\theta$ is an automorphism. It follows that $x$ and $y$ have the same type in $G'$. Moreover, $G$ being $\mathrm{EAE}$-embedded into $G'$, $x$ and $y$ have the same $\mathrm{EAE}$-type in $G$.

%In fact, it seems reasonable to conjecture that $G$ is elementarily embedded into $G'$ and hence that $x$ and $y$ have the same type in $G$, and that $G$ is not homogeneous.

\subsubsection{There is no automorphism of $G$ mapping $x$ to $y$}

Suppose towards a contradiction that there is an automorphism $\sigma$ of $G$ such that $\sigma(x)=y$. Note in particular that $y$ belongs to $\sigma(A)\cap A$, as $x$ and $y$ belong to $A$. Let $T$ denote the Bass-Serre tree of the splitting $A\ast_C B$ of $G$. Note that the edge stabilizers of $T$ are the conjugates of $C$. 

\smallskip \noindent
Let $v_A$ denote the unique vertex of $T$ fixed by $A$. Note that $v_A$ is also the unique vertex of $T$ fixed by $y$: indeed, let $v$ be a vertex of $T$ fixed by $y$ and suppose that $v\neq v_A$. Then $y$ fixes the path joining $v$ to $v_A$, in particular $y$ fixes an edge adjacent to $v_A$, whose stabilizer is a conjugate of $C$ by an element $a\in A$. It follows that $a^{-1}ya$ belongs to $C$. But recall that $A=A_1\times A_2$ with $y\in A_2$, so one can write $a=a_1a_2$ with $a_1\in A_1$ and $a_2\in A_2$ and one gets that $a_2^{-1}ya_2$ belongs to $C$, which is impossible as $y$ is a product of three transpositions with disjoint supports in $A_2\simeq S_6$ whereas $C$ is generated by two commuting transpositions. Hence $v=v_A$ and so $v_A$ is the unique vertex of $T$ fixed by $y$. Therefore, since $y$ belongs to $\sigma(A)\cap A$, the vertex $v_A$ is also the only vertex of $T$ fixed by $\sigma(A)$, and it follows that $\sigma(A)=A$. Then, $\sigma(B)$ being elliptic in $T$ and not isomorphic to a subgroup of $A$, there is an element $g\in G$ such that $\sigma(B)=gBg^{-1}$, and since $\sigma$ is surjective it is not hard to prove that $g=ab$ with $a\in A$ and $b\in B$, thus $\sigma(B)=abBb^{-1}a^{-1}=aBa^{-1}$. 

\smallskip \noindent
Define $\sigma'=\mathrm{ad}(a^{-1})\circ \sigma$, so that $\sigma'(A)=A$ and $\sigma'(B)=B$ (and thus $\sigma'(C)=C$ as $C=A\cap B$). Write $B=B_1\times B_2\times B_3$ where $B_1,B_2,B_3$ denote the direct factors that appear in Figure \ref{B}, with $e_2,e_3\in B_1$, $e_1\in B_2$ and $e_4\in B_3$. By \cite{Bidwell,Bidwell2}, since $B_1,B_2,B_3$ are pairwise non-isomorphic, have trivial center, and have no common non-trivial direct factor, we have $\sigma'(B_i)=B_i$ for every $i\in\lbrace 1,2,3\rbrace$. Moreover, $\sigma'(C)=C$, so $\sigma'(e_1)$ belongs to $B_2\cap C=\langle e_1\rangle$ and thus $\sigma'(e_1)=e_1$. Now, recall that $A=A_1\times A_2$ where $A_1$ and $A_2$ denote the direct factors that appear in Figure \ref{A}, with $e_1,e_2\in A_1$ and $e_3,e_4\in A_2$. Still by \cite{Bidwell,Bidwell2}, we have $\sigma'(A_1)=A_1$ or $\sigma'(A_1)=A_2$, but since $\sigma'(e_1)=e_1$ the second option does not occur and thus $\sigma'(A_1)=A_1$ and $\sigma(A_1)=aA_1a^{-1}=A_1$, which contradicts our initial assumption that $\sigma(x)=y$ as $x$ belongs to $A_1$ and $y$ belongs to $A_2$.

\section{Direct products}\label{direct_products}

The main aim of this section is to prove that under certain conditions, homogeneity behaves well with regard to direct products, which allows us to combine results from Sections \ref{homogeneity_crystallographic} and \ref{homogeneity_hyperbolic}. The following definition was introduced in \cite{alg_geom_over_groups3}.

\begin{de} Let $G$ be a group. For $x, y \in G$ we define the following operation:
	$$x \diamond y = [\mathrm{gp}_G(x), \mathrm{gp}_G(y)],$$
where $\mathrm{gp}_G(z)$ denotes the normal closure of $x$ in $G$, that is, the smallest normal subgroup of $G$ containing $x$. We call a non-trivial element $x \in G$ a zero-divisor in $G$ if there exists a non-trivial element
$y \in G$ such that $x \diamond y = 1$. We say that the group $G$ is a domain if it has no zero-divisors. Finally, we write $x \perp y$ when $x \diamond y = 1$.
\end{de} 

The following lemma follows from standard techniques.

\begin{lemme}\label{standard}Every non-elementary hyperbolic group without a non-trivial normal finite subgroup is a domain. 
\end{lemme}

\begin{proof}
Let $G$ be a non-elementary hyperbolic group without a non-trivial normal finite subgroup. By \cite{ol93}, there exist three elements $g_1,g_2,g_3\in G$ of infinite order such that, for every $i,j\in\lbrace 1,2,3\rbrace$ with $i\neq j$, the following conditions hold:
\begin{itemize}
    \item the maximal virtually cyclic subgroup $M(g_i)$ of $G$ containing $g_i$ is equal to $\langle g_i\rangle$;
    \item $\langle g_i\rangle \cap \langle g_j\rangle$ is trivial.
\end{itemize}
Let $x,y$ be non-trivial elements of $G$. Pick $g_i$ with $i\in\lbrace 1,2,3\rbrace$ such that $x\notin M(g_i)$ and $y\notin M(g_j)$ (this is possible by the conditions above). By Baumslag's Lemma (see \cite[Corollary 2.20]{And20} or \cite{ol93}), the element $[x,g_i^nyg_i^{-n}]=xg_i^nyg_i^{-n}x^{-1}g_i^ny^{-1}g_i^{-n}$ is non-trivial for $n\geq 1$ sufficiently large. Therefore, $x$ is not a zero-divisor in $G$, and thus $G$ is a domain.
\end{proof}

	\begin{notation} 
	$\mathrm{Comp}(x,z) = \forall y (y \diamond z = 1 \rightarrow x \diamond y = 1)$ (cf. \cite[Proof of Lemma~4]{alg_geom_over_groups3}).
\end{notation}

	\begin{de}\label{Dk_notation} As in \cite{alg_geom_over_groups3}, we denote by $D_k$ the groups of the forms $G_1 \times \cdots \times G_k$ with each $G_i$ a domain.
\end{de}

	\begin{fact}[{\cite[Theorem~10.4]{AP24}}]\label{the_direct_dec_fact} Let $H = H_1 \times H_2$ with $H_2 \in D_k$ and $H_1$ abelian-by-finite. Then for every $K \equiv H$ there are $K_1, K_2 \leq K$ such that we have:
	\begin{enumerate}[(1)]
	\item $K = K_1 \times K_2$;
	\item $K_2 \in D_k$;
	\item $K_1 \equiv H_1$ and $K_2 \equiv H_2$. 
\end{enumerate}	
	\end{fact}

    In the proof of the following proposition, when we refer to be a monster model of $G$, we mean a sufficiently saturated model of $\mathrm{Th}(G)$, i.e., a $\kappa$-saturated model of $\mathrm{Th}(G)$ for $\kappa$ large enough. This is standard technique in model theory, for more on the notion of monster model of a complete first-order theory see e.g. \cite[pg.~218-219]{marker}.

\begin{prop}\label{hom_direct_prod} Let $G = G_1 \times G_2$ with $G_1$ abelian-by-finite and $G_2 \in D_k$, for some $0 < k < \omega$. If $G_1$ and $G_2$ are homogeneous, then so is $G$.
\end{prop}

	\begin{proof} Let $\mathfrak{M}$ be a monster model of $G$. Then by \ref{the_direct_dec_fact} there are $\mathfrak{M}_1$ and $\mathfrak{M}_2$ such that:
	\begin{enumerate}[(1)]
	\item $\mathfrak{M} = \mathfrak{M}_1 \times \mathfrak{M}_2$;
	\item $\mathfrak{M}_2 \in D_k$;
	\item $\mathfrak{M}_1 \equiv G_1$ and $\mathfrak{M}_2 \equiv G_1$. 
\end{enumerate}	
In particular, easily $\mathfrak{M}_1$ is abelian-by-finite. Notice that w.l.o.g. we can suppose that, for $i = 1, 2$, $\mathfrak{M}_i$ is a monster model of $G_i$. Why? For $i = 1, 2$, replace $\mathfrak{M}_i$ with an elementary extension $\mathfrak{M}'_i$ that is saturated enough and embeds elementarily $G_i$. Then, for $i = 1, 2$, $\mathfrak{M}_i$ is a monster model of $G_i$ and furthermore $\mathfrak{M}_1 \times \mathfrak{M}_2 \preccurlyeq \mathfrak{M}'_1 \times \mathfrak{M}'_2$ (see e.g. \cite[Section~3.3, Exercise~7]{hodges}) and so, as $\mathfrak{M} = \mathfrak{M}_1 \times \mathfrak{M}_2$ was a monster model for $G$, so is $\mathfrak{M}' = \mathfrak{M}'_1 \times \mathfrak{M}'_2$. So, we can indeed assume that for $i = 1, 2$, $\mathfrak{M}_i$ is a monster model of $G_i$ and, renaming elements, we can also assume that $G = G_1 \times G_2 \preccurlyeq \mathfrak{M}_1 \times \mathfrak{M}_2 = \mathfrak{M}$.

\smallskip \noindent
Now clearly in $\mathfrak{M}$, having the same $\emptyset$-type (i.e., the type consisting of formulae without parameters) and being automorphic are the same equivalence relation, and so, in order to conclude that $G$ is homogeneous, it suffices to show that if $\bar{a}, \bar{b} \in G^\ell$, for some $\ell < \omega$, and there is an $\alpha \in \mathrm{Aut}(\mathfrak{M})$ which maps $\bar{a}$ to $\bar{b}$, then there is an $\beta \in \mathrm{Aut}(G)$ which maps $\bar{a}$ to $\bar{b}$. So suppose that $\bar{a}, \bar{b} \in G^\ell$ and $\alpha \in \mathrm{Aut}(\mathfrak{M})$ which maps $\bar{a}$ to $\bar{b}$ are given. Let $\bar{a} = (a_1, ..., a_\ell)$, then for every $1 \leq i \leq \ell$ there are unique $a^1_i \in G_1$ and $a^2_i \in G_2$ such that $a_i = a^1_i \cdot_G a^2_i$. Recall now that, as observed above, we have that $\mathfrak{M}_1$ is abelian-by-finite and $\mathfrak{M}_2 \in D_k$, hence necessarily we have that $\alpha(\mathfrak{M}_i) = \mathfrak{M}_i$, for $i = 1, 2$, that is $\alpha = \alpha_1 \times \alpha_2$ with $\alpha_i = \alpha \restriction \mathfrak{M}_i$. But by assumption $G_1$ and $G_2$ are homogeneous and, for $i = 1, 2$, $\mathfrak{M}_i$ is a monster model of $G_i$, hence we can find $\beta_i \in \mathrm{Aut}(G_i)$ which maps $(a^i_1, ..., a^i_\ell)$ to $(\alpha(a^i_1), ..., \alpha(a^i_\ell))$, for $i = 1, 2$. But then $\beta = \beta_1 \times \beta_2 \in \mathrm{Aut}(G)$ \mbox{is such that $\beta(\bar{a}) = \bar{b}$.}
	\end{proof}
	
\begin{co}\label{final_co} Let $W_1$ be an affine or spherical (that is, finite) Coxeter group and let $W_2$ be a finite product of non-elementary irreducible hyperbolic Coxeter groups. If $W_2$ is homogeneous then $W = W_1 \times W_2$ is homogeneous.
\end{co}

\begin{proof} Clearly $W_1$ is abelian-by-finite and by Lemma \ref{standard} and Definition \ref{Dk_notation} $W_2$ belongs to $D_k$. Moreover, by Theorem~\ref{theorem1}, $W_1$ is homogeneous. Hence $W_1\times W_2$ is homogeneous by Proposition \ref{hom_direct_prod}.
\end{proof}

As explained in the introduction, recall that the direct product of two homogeneous hyperbolic groups is not homogeneous in general. The example given in the introduction, namely $F_3\times F_2$, is based on the fact that $F_3$ is not strictly minimal, so we ask the following question.

\begin{qu}Let $G_1$ and $G_2$ be homogeneous strictly minimal hyperbolic groups. Is $G_1\times G_2$ homogeneous?
\end{qu}

Recall that we proved in the previous section that torsion-generated hyperbolic groups are strictly minimal, so the following question (which is a particular case of the question above) is natural.

\begin{qu}Let $G_1$ and $G_2$ be homogeneous torsion-generated hyperbolic groups. Is $G_1\times G_2$ homogeneous?
\end{qu}

Recall that a hyperbolic group is called \emph{rigid} if it is not virtually cyclic (finite or infinite) and if it does not split non-trivially as an amalgamated product or as an HNN extension over a virtually cyclic group (finite or infinite). We conclude this paper with the following result, which gives a partial answer to the first question above (indeed, it is not difficult to prove that rigid hyperbolic groups are homogeneous and strictly minimal).

\begin{prop}\label{prod_rig_hyp}
Let $G_1$ and $G_2$ be rigid hyperbolic groups. Suppose that each non-trivial finite subgroup of $G_1$ and $G_2$ has virtually cyclic (possibly finite) normalizer. Then $G_1\times G_2$ is homogeneous. Moreover, the result remains true for any finite number of such direct factors. 
\end{prop}

\begin{rk}
For example, we can take for $G_1$ and $G_2$ two hyperbolic Coxeter triangle groups. 
\end{rk}

\begin{proof}We will prove the result for two direct factors, and the proof can be easily extended to any finite number of direct factors.

\smallskip \noindent
According to Theorem \ref{shortening} (where we take for $H$ the trivial subgroup), for every $i,j\in\lbrace 1,2\rbrace$, there exists a finite subset $F_{i,j}\subset G_i\setminus \lbrace 1\rbrace$ with the property that any morphism $\varphi : G_i\rightarrow G_j$ such that $\ker(\varphi)\cap F_{i,j}=\varnothing$ is injective. Let $u,u'\in G^{\ell}$ (with $\ell\geq 1)$ such that $\mathrm{tp}(u)=\mathrm{tp}(u')$. For simplicity of notation, assume that $\ell=1$ (the case where $\ell\geq 2$ works in the same way). 

\smallskip \noindent
Let us start with the following preliminary observation: for any $g\in G$, the centralizer $C_G(g)$ is not virtually abelian if and only if $g$ belongs to $G_1$ or $G_2$. Indeed, writing $g=g_1g_2$ with $g_1\in G_1$ and $g_2\in G_2$, one easily sees that $C_G(g)=C_{G_1}(g_1)\times C_{G_2}(g_2)$. Note that if $g_i\neq 1$ then $C_{G_i}(g_i)$ is virtually cyclic (indeed, since $G_i$ is hyperbolic, the centralizer of any element of $G_i$ of infinite order is virtually cyclic infinite, and by assumption the centralizer of any non-trivial element of $G_i$ of finite order is virtually cyclic (finite or infinite)). Therefore, $C_G(g)$ is virtually abelian if and only if $g_1\neq 1$ and $g_2\neq 1$ if and only if $g$ does not belong to $G_1$ or $G_2$

\smallskip \noindent
Note that $u$ can be written in a unique way as $u=u_1u_2$ with $u_1\in G_1$ and $u_2\in G_2$. Moreover, if $u_1\neq 1$ and $u_2\neq 1$ then it is not difficult to see that the (unordered) pair $\lbrace u_1,u_2\rbrace$ can be characterized as follows: if $u=g_1g_2$ with $g_1,g_2\in G$ such that $C_G(g_1)$ is not virtually abelian and $C_G(g_2)$ is not virtually abelian, then $\lbrace g_1,g_2\rbrace=\lbrace u_{1},u_{2}\rbrace$. %Indeed, if $u=g_1g_2$ with $g_1,g_2\in G$ such that $C_G(g_1)$ is not virtually abelian and $C_G(g_2)$ is not virtually abelian, then, by the preliminary observation above, $g_1$ belongs to $G_1$ or $G_2$ and $g_2$ belongs to $G_1$ or $G_2$. But $g_1$ and $g_2$ do not belong to the same $G_i$ since $u$ does not belong to $G_1$ or $G_2$ (as by assumption $u_1\neq 1$ and $u_2\neq 1$), so $\lbrace g_1,g_2\rbrace=\lbrace u_{1},u_{2}\rbrace$ due to the uniqueness of the decomposition of $u$.

\smallskip \noindent
Write $u'=u'_1u'_2$ with $u'_1\in G_1$ and $u'_2\in G_2$. Let $S_1, S_2$ be finite generating sets for $G_1,G_2$. Define $S=S_1\cup S_2$ and let $G=\langle S \ \vert \ R\rangle$ be a finite presentation. Furthermore, we can assume that, for every $i\in\lbrace 1,2\rbrace$ and for every $s,s'\in S_i$, $s$ and $s'$ do not commute: indeed, if two generators $s,s'\in S_i$ commute, pick an element $g\in G_i$ that does not commute with any of the elements of $S_i$, and replace $S_i$ with $(S_1\setminus \lbrace s\rbrace)\cup \lbrace g,gs\rbrace$ (it is possible to find such an element $g$ because $G_i$ has no non-trivial normal finite subgroup, and so there exists a sequence $(g_n)_{n\in \mathbb{N}}\in G_i^{\mathbb{N}}$ of elements of infinite order such that, for every $n\in\mathbb{N}$, the maximal virtually cyclic subgroup of $G_i$ containing $g_n$ is equal to $\langle g_n\rangle$, and for every $n,m\in\mathbb{N}$ with $n\neq m$ the intersection of $\langle g_n\rangle$ and $\langle g_m\rangle$ is trivial). Observe that $\varphi=\mathrm{id}_G$ satisfies $\varphi(u)=u$ and the following conditions, for every $i,j\in\lbrace 1,2\rbrace$:
\begin{enumerate}[(1)]
    \item $\ker(\varphi)\cap F_{i,j}=\varnothing$;
    \item for every $s\in S_i$, 
    $\varphi(s)\neq 1$ and the centralizer of $\varphi(s)$ in $G$ is not virtually abelian;
    \item for every $s,s'\in S_i$, $\varphi(s)$ and $\varphi(s')$ do not commute;
    \item if $u_i\neq 1$ then the centralizer of $\varphi(u_i)$ in $G$ is not virtually abelian.
\end{enumerate}
This statement is expressible using a first-order formula $\theta(u)$ (the key point being that, given a finite presentation $\langle s_1,\ldots ,s_n \ \vert \ R(s_1,\ldots,s_n)=1\rangle$ for $G$, $\mathrm{Hom}(G,G)$ is in one-to-one correspondence with the solutions $(x_1,\ldots,x_n)$ in $G^n$ to the system of equations $R(x_1,\ldots,x_n)=1$). Hence, since $u$ and $u'$ have the same type by assumption, $G\models\theta(u')$ and thus there is an endomorphism $\varphi$ of $G$ such that $\varphi(u)=u'$ and the four conditions above hold.

\smallskip \noindent
\underline{\textbf{First case}}: suppose that $u_1\neq 1$ and $u_2\neq 1$. Then, for every $i\in\lbrace 1,2\rbrace$, the centralizer of $\varphi(u_i)\neq 1$ in $G$ is not virtually abelian (by the fourth point). Note also that $u'=u'_1u'_2$ with $u'_1\neq 1$ and $u'_2\neq 1$ (otherwise $C_G(u')$ would not be virtually abelian, contradicting the fact that $C_G(u)$ is virtually abelian and that $u$ and $u'$ have the same type in $G$). From the preliminary observation, as $u'=u'_1u'_2=\varphi(u_1)\varphi(u_2)$, it follows that $\lbrace u'_1,u'_2\rbrace=\lbrace \varphi(u_1),\varphi(u_2)\rbrace$. 

\smallskip \noindent
From the preliminary observation and from the second point above, for every $s\in S_1$, $\varphi(s)$ is contained in $G_1$ or $G_2$. Moreover, by the third point, if $s$ and $s'$ belong to $S_1$ then $\varphi(s),\varphi(s')$ do not commute. Therefore, $\varphi(G_1)$ is contained in $G_1$ or $G_2$. For the same reason, $\varphi(G_2)$ is contained in $G_1$ or $G_2$.
Moreover, as $\lbrace u'_1,u'_2\rbrace=\lbrace \varphi(u_1),\varphi(u_2)\rbrace$, we must have $\varphi(G_i)\subset G_{\sigma(i)}$ where $\sigma$ is a bijection of $\lbrace 1,2\rbrace$.

\smallskip \noindent
Furthermore, note that $\varphi_{\vert G_i}$ is injective since, by the first point, $\ker(\varphi)\cap F_{i,j}=\varnothing$ (for every $i,j\in\lbrace 1,2\rbrace$). 

\smallskip \noindent
\emph{First subcase}: if $\sigma$ is the identity of $ \lbrace 1,2\rbrace$ then $\varphi_{\vert G_i}$ is an automorphism of $G_i$ (indeed, as a one-ended hyperbolic group, $G_i$ is co-Hopfian). Hence $\varphi$ is an automorphism of $G$.

\smallskip \noindent
\emph{Second subcase}: if $\sigma(1)=2$ and $\sigma(2)=1$ then $G_1$ embeds into $G_2$ and $G_2$ embeds into $G_1$, so $G_1$ and $G_2$ are isomorphic (still by the co-Hopf property) and we conclude as in the first subcase.

%Now, let $\pi_i$ denote the projection $G\rightarrow G_i$, let $H_1=\varphi(G_1)$ and let $K_i=\lbrace h\in H_1 \ \vert \ \pi_i(h)=1\rbrace$. Note that $K_1\cap K_2$ is trivial. Moreover, for any $k_1\in K_1$ and $k_2\in K_2$, we have $\pi_1([k_1,k_2])=\pi_2([k_1,k_2])=1$, hence $K_1$ and $K_2$ are two normal subgroups of $H_1$ centralizing each other. Since $H_1$ is not virtually abelian (by the third item above) and since $K_1$ and $K_2$ are normal in $H_1$, they cannot be virtually cyclic infinite. Moreover, $K_1$ and $K_2$ cannot be simultaneously non-elementary, so $K_1$ or $K_2$ is finite. Without loss of generality, suppose that $K_1$ is finite, then it is trivial because $H_1$, which is non-elementary, is contained in $N_{G}(K_i)$. Hence $H_1=\varphi(G_1)$ is contained in $G_1$ or $G_2$, and moreover $\varphi_{\vert G_1}$ is injective as $\ker(\varphi)\cap F_{1,j}=\varnothing$. Similarly, $\varphi$ maps $G_2$ injectively in $G_1$ or $G_2$. 

\smallskip \noindent
\underline{\textbf{Second case}}: suppose that $u_1=1$ or $u_2=1$. Without loss of generality, suppose that $u_2=1$, or equivalently that $u$ belongs to $G_1$. Then $u'_1=1$ or $u'_2=1$. 

\smallskip \noindent 
\underline{\emph{First subcase}}: if $u'_2=1$ then, arguing as the first case, we get an endomorphism $\varphi$ of $G$ such that $\varphi_{\vert G_1}$ is an automorphism from $G_1$ to $G_1$ mapping $u$ to $u'$, and we can extend it to an automorphism of $G$ that is the identity map on $G_2$ and that maps $u$ to $u'$.

\smallskip \noindent
\underline{\emph{Second subcase}}: if $u'_1=1$ then, arguing as in the first case, we get an endomorphism $\varphi$ of $G$ such that $\varphi_{\vert G_1}$ is an injection from $G_1$ into $G_2$ mapping $u$ to $u'$. However, we cannot immediately conclude that $\varphi_{\vert G_1}$ is an isomorphism between $G_1$ and $G_2$. However, in the same way, we prove that there is an endomorphism $\varphi'$ of $G$ such that $\varphi'(u')=u$ and that injectively maps $G_2$ to $G_1$. The morphism $\varphi\circ \varphi'$ maps $G_2$ injectively into $G_2$, therefore $\varphi_{\vert G_1}\circ \varphi'_{\vert G_2}$ is an automorphism of $G_2$ (by the co-Hopf property), hence $\varphi_{\vert G_1}: G_1\rightarrow G_2$ is surjective and thus bijective, and we conclude as above.\end{proof}

Proposition~\ref{prod_rig_hyp} allows us to derive the following corollary.

\begin{co}\label{cor_triangle}
Any direct product of finitely many triangle groups is homogeneous.   
\end{co}

\begin{proof}
By Proposition~\ref{prod_rig_hyp}, any direct product of finitely many hyperbolic triangle groups is homogeneous. Moreover, any direct product of finitely many finite or irreducible affine triangle groups is homogeneous according to Theorem \ref{main_affine_homogeneous} (indeed, such a direct product is an affine Coxeter group). We conclude using Corollary \ref{final_co}.\end{proof}

%the main reason being that a morphism $\varphi :G_1 \times \cdots \times G_n\rightarrow G_1 \times \cdots \times G_n$ that is injective on $G_i$ satisfies $\varphi(G_i)\subset G_j$ for some $j$ (since the kernel of each projection $\pi_j\circ \varphi_{\vert G_i}$ is a normal subgroup of $G_i$). 

%Let $a_1b_1\cdots a_k$ be a reduced normal form of $g$, with $k\geq 1$ and $a_1$ possibly trivial. Consider the automorphism $\sigma'=\mathrm{ad}(a_1^{-1})\circ \sigma$ of $G$. Let $b\in B\setminus C$. Since $\sigma'$ is surjective, there is a non-trivial element $h\in G$ such that $\sigma'(h)=b$. Write $h=a'_1b'_1\cdots a'_{\ell}b'_{\ell}$. If $g'$ is non-trivial, the length of ...

%Here is an alternative, more conceptual argument: the splitting $A\ast_C B$ is the only reduced Stallings splitting of $G$, moreover the Stallings deformation space of $G$ is invariant under $\mathrm{Aut}(G)$ and any automorphism of $G$ maps a reduced tree to a reduced tree, so the tree $T$ is invariant under $\mathrm{Aut}(G)$. Therefore, there is a $\sigma$-equivariant isometry $f : T\rightarrow T$, and since $G$ acts transitively on the set ...

\renewcommand{\refname}{Bibliography}
\bibliographystyle{alpha}
\bibliography{biblio}

\vspace{2cm}

\setlength{\parindent}{0pt}
Simon Andr{\' e}
\\
Sorbonne Universit{\' e} and Universit{\' e} Paris Cit{\' e}
\\
CNRS, Institut de mathématiques de Jussieu - Paris Rive Gauche
\\
F-75005 Paris, France.
\\
Email address: \href{mailto:simon.andre@imj-prg.fr}{simon.andre@imj-prg.fr}

\bigskip

Gianluca Paolini
\\
Department of Mathematics ``Giuseppe Peano'', University of Torino
\\
Via Carlo Alberto 10, 10123, Italy.
\\
Email address: \href{mailto:gianluca.paolini@unito.it}{gianluca.paolini@unito.it}

\end{document}